\documentclass{siamltex}

\newcommand{\mathbb}[1]{\mathsf{I\!#1}}

\newcommand{\BYDEF}{\,\shortstack{\textup{\tiny def} \\ = }\,}

\newtheorem{a-1}{Assumption}[section]
\newtheorem{a-2}[a-1]{Assumption}
\newtheorem{a-3}[a-1]{Assumption}
\newtheorem{a-4}[a-1]{Assumption}
\newtheorem{a-5}[a-1]{Assumption}

\newtheorem{l-1a}[a-1]{Lemma}
\newtheorem{t-1new}[a-1]{Theorem}
\newtheorem{t-2new}[a-1]{Theorem}
\newtheorem{d-1}[a-1]{Definition}
\newtheorem{t-1}[a-1]{Theorem}
\newtheorem{p-1opt}[a-1]{Proposition}
\newtheorem{p-limav}[a-1]{Proposition}

\newtheorem{l-1}[a-1]{Lemma}
\newtheorem{l-2}[a-1]{Lemma}
\newtheorem{c-1}[a-1]{Corollary}
\newtheorem{c-2}[a-1]{Corollary}
\newtheorem{t-2}[a-1]{Theorem}
\newtheorem{p-1}[a-1]{Proposition}
\newtheorem{t-3new}[a-1]{Theorem}

\newtheorem{a-1x}[a-1]{Assumption}
\newtheorem{a-2x}[a-1]{Assumption}
\newtheorem{a-3x}[a-1]{Assumption}
\newtheorem{a-4x}[a-1]{Assumption}
\newtheorem{a-5x}[a-1]{Assumption}
\newtheorem{a-6x}[a-1]{Assumption}
\newtheorem{a-7x}[a-1]{Assumption}
\newtheorem{p-1x}[a-1]{Proposition}
\newtheorem{t-1x}[a-1]{Theorem}

\newtheorem{r-1}[a-1]{Remark}
\newtheorem{r-2}[a-1]{Remark}
\newtheorem{r-3}[a-1]{Remark}

\newtheorem{l-81}[a-1]{Lemma}
\newtheorem{l-82}[a-1]{Lemma}

\usepackage{graphicx,color}
\usepackage{latexsym}
\usepackage{amsfonts}
\def\Re{{\rm I\kern-0.2em R}}

\def\R{\Re}

\newtheorem{Theorem}{Theorem}[section]
\newtheorem{Definition}[Theorem]{Definition}
\newtheorem{Proposition}[Theorem]{Proposition}
\newtheorem{Lemma}[Theorem]{Lemma}
\newtheorem{Corollary}[Theorem]{Corollary}

\newtheorem{Assumption}[Theorem]{Assumption}

\usepackage{amsmath}
\usepackage{amssymb}
\usepackage{epsf}

\title{Averaging and linear programming in some singularly perturbed  problems of optimal control}
\author{Vladimir Gaitsgory\thanks{Flinders Mathematical Sciences Laboratory, School of Computer Science, Engineering and Mathematics, Flinders University, GPO Box 2100, Adelaide SA 5001, Australia, vladimir.gaitsgory@flinders.edu.au; The work of V. Gaitsgory was
supported by the Australian Research Council Discovery Grants DP130104432 and
DP120100532} \and
Sergey Rossomakhine
\thanks{Flinders Mathematical Sciences Laboratory, School of Computer Science, Engineering and Mathematics, Flinders University, GPO Box 2100, Adelaide SA 5001, Australia, serguei.rossomakhine@flinders.edu.au; The work of S. Rossomakhine was
supported by the Australian Research Council Discovery Grant
DP120100532}}
\begin{document}

\maketitle
\begin{abstract}
The paper aims at the development of an apparatus  for analysis and construction of near optimal solutions   of singularly
perturbed (SP) optimal controls problems (that is, problems of optimal control of SP systems) considered on the infinite time horizon.
 We mostly focus on problems with
time discounting criteria but  a possibility of the extension of results to periodic optimization problems is discussed as well.
Our consideration is based on earlier results on averaging of SP control systems and on linear programming formulations of optimal control problems.
The idea that we exploit is to first
asymptotically approximate a given problem of optimal control of the SP system by a certain averaged
optimal control problem, then reformulate this averaged problem as an infinite-dimensional (ID) linear programming (LP) problem,
and then approximate the latter by semi-infinite LP problems. We show that the optimal solution of these semi-infinite LP problems and their duals
(that can be found with the help of a modification of an available LP software) allow one to construct near optimal controls of the SP system. We demonstrate the construction with two numerical examples.
\end{abstract}

\bigskip


\bigskip


\bigskip
\bigskip

{\bf I. Introduction and preliminaries.}

\section{Contents of the paper}\label{Sec-Contents}
The paper aims at the development of an apparatus  for analysis and construction of near optimal solutions   of singularly
perturbed (SP) optimal controls problems (that is, problems of optimal control of SP systems) considered on the infinite time horizon.
 We mostly focus on problems with
time discounting criteria but  a possibility of the extension of results to periodic optimization problems is discussed as well. Our consideration is
 based on earlier results on averaging of SP control systems and on linear programming formulations of optimal control problems.
The idea that we exploit is to first
asymptotically approximate a given problem of optimal control of the SP system by a certain averaged
optimal control problem, then reformulate this averaged problem as an infinite-dimensional (ID) linear programming (LP) problem,
and then approximate the latter by semi-infinite LP problems. We show that the optimal solution of these semi-infinite LP problems and their duals
(that can be found with the help of a modification of an available LP software) allow one to construct near optimal controls of the SP system.

We will be considering the SP system written in the form
\begin{eqnarray}\label{e:intro-0-1}
\epsilon y'(t) &=&f(u(t),y(t),z(t)) ,   \\
z'(t)&=&g(u(t),y(t),z(t)),   \label{e:intro-0-2}
\end{eqnarray}
where  $\epsilon > 0$ is a small parameter; $\ f(\cdot) :U\times
\mathbb{R}^m \times \mathbb{R}^n \to \mathbb{R}^m, \ g(\cdot):U\times
\mathbb{R}^m \times \mathbb{R}^n\to \mathbb{R}^n$ are continuous
vector functions satisfying Lipschitz conditions in $z$ and $y$; and where
controls $u(\cdot) $ are  measurable functions of time
 satisfying  the inclusion
\begin{equation}
 u(t) \in U,
 \end{equation}
 $U$ being a given compact metric space. The system (\ref{e:intro-0-1})-(\ref{e:intro-0-2}) will be considered with the initial condition
 \begin{equation}\label{e-initial-SP}
 (y_{\epsilon}(0),
z_{\epsilon}(0)) =(y_0,
z_0). \end{equation}
We are assuming that  all  solutions of the system obtained with this initial condition
  satisfy the inclusion
\begin{equation}\label{equ-Y}
(y_{\epsilon}(t),z_{\epsilon}(t))\in Y\times Z \ \ \ \ \forall t\in [0,\infty),
\end{equation}
 where $Y$ is a  compact
 subset of $\mathbb{R}^m$ and $Z $ is a  compact subset of $\mathbb{R}^n$ (the consideration is readily extendable to the case when only optimal and near optimal solutions satisfy (\ref{equ-Y})).
 We will  be mostly dealing with the problem of optimal control
\begin{equation}\label{Vy-perturbed}
\displaystyle{
  \inf _{u(\cdot) }\int_0 ^ { + \infty }e^{-C t}G(u(t), y_{\epsilon}(t),
z_{\epsilon}(t))dt} \stackrel{\rm def}{=} V_{di}^*(\epsilon, y_0, z_0) ,
\end{equation}
where $G(\cdot): U\times \R^m \times  \R^n \mapsto \R$ is a continuous function,
$C >0$ is a discount rate,  and   $inf$ is sought over all controls and the corresponding solutions
of (\ref{e:intro-0-1})-(\ref{e:intro-0-2}) that satisfy the initial condition(\ref{e-initial-SP}).
However, the approach that we are developing is applicable to other classes of SP optimal control problems   as well. To demonstrate this point, we will indicate a way how results obtained for the  problem with time discounting criterion (\ref{Vy-perturbed})
can be extended to the periodic optimization setting, and we will consider an example of a SP periodic optimization problem which is numerically solved
with the help of the proposed technique.

  The presence of $\epsilon$ in the  system  (\ref{e:intro-0-1})-(\ref{e:intro-0-2})) implies that
the rate with which the $y$-components of the state variables
change their values is of the order $\frac{1}{\epsilon}$ and is, therefore,
much higher than the rate of changes of the $z$-components (since
$\epsilon$ is assumed to be small). Accordingly, the
$y$-components and $z$-components of the state variables  are
referred to as {\it fast} and {\it slow}, respectively.

Problems of optimal controls of singularly perturbed systems appear in a variety of  applications and have received a
great deal of attention in the literature (see, e.g., \cite{Alv-1}, \cite{Art3}, \cite{Ben},
\cite{Bor-Gai}, \cite{Col}, \cite{Dmitriev}, \cite{Dont-Don}, \cite{DonDonSla},  \cite{Fil4}, \cite{Gai8}, \cite{Gra2}, \cite{Kab2}, \cite{Kok1},
\cite{Kus3}, \cite{Lei}, \cite{Nai}, \cite{plot}, \cite{QW}, \cite{Reo1}, \cite{Vel}, \cite{Vig},
\cite{Yin} and references therein).
A most common  approach to such problems is  based on the idea of approximating
the slow $z$-components of the solutions of the SP system (\ref{e:intro-0-1})-(\ref{e:intro-0-2}) by
the solutions of the so-called reduced system
\begin{equation}\label{e:red-1} z'(t)= g(u(t),q(u(t),z(t)),z(t)),  \end{equation}
   which is obtained from (\ref{e:intro-0-1}) via   the replacement of $\ y(t)\ $ by   $\ q(u(t),z(t)) \ $, with $q(u,z)$ being the  root of the equation
\begin{equation}\label{e:red-2} f(u,y,z) = 0  . \end{equation}
Note that the equation (\ref{e:red-2})  can be
 obtained by  formally equating  $\epsilon$  to zero in (\ref{e:intro-0-1}).

Being very efficient in dealing with many important classes of optimal control problems (see, e.g.,
 \cite{Ben}, \cite{DonDonSla}, \cite{Kab2}, \cite{Kok1}, \cite{Kus3}, \cite{Nai},  \cite{Reo1}, \cite{Vel},
 \cite{Yin}), this approach may not be applicable in the general case (see examples in  \cite{Art2},
\cite{Gai0}, \cite{Gai1}, \cite{Lei}).
 In fact, the validity of the assertion that the system (\ref{e:red-1}) can be used for finding a near optimal control of the SP system (\ref{e:intro-0-1})-(\ref{e:intro-0-2}) is related to the
  validity of the hypothesis that the optimal control of the latter is in some sense
slow and that (in the optimal or near optimal regime) the fast state variables converge rapidly to their quasi steady states
 defined by the root of (\ref{e:red-2}) and remain in a neighborhood of this root, while the
slow variables  are changing in accordance with (\ref{e:red-1}). While the validity of such a hypothesis has been established
under natural stability conditions  by famous Tichonov's theorem in the case of uncontrolled dynamics
(see  \cite{Reo} and \cite{Vas}), this hypothesis may not be valid in  the control setting if the dynamics is nonlinear and/or  the objective function is non-convex,
the reason for this being the fact that the use of rapidly oscillating controls may lead to  significant (not tending to zero with $\epsilon $)
improvements of the performance indexes.

Various averaging type approaches allowing one to deal with the fact that the optimal
or near optimal controls can take the form of rapidly oscillating
functions have been proposed and studied by a number of researchers (see  \cite{Alv}, \cite{Alv-1}, \cite{Art2}, \cite{Art0},  \cite{Art3}, \cite{Art1}, \cite{Bor-Gai}, \cite{Bor-Gai-1}, \cite{Bor-Gai-2}, \cite{Don}, \cite{Dont-Don}, \cite{Fil4}, \cite{Filatov},
\cite{Gai0},
 \cite{Gai1}, \cite{Gai8}, \cite{Gai-Leiz}, \cite{Gai-Ng}, \cite{Gra}, \cite{Gra2}, \cite{plot}, \cite{QW},   \cite{Vig} and references therein). This collective effort
lead to a good understanding of what the \lq\lq true limit" problems, optimal solutions of which approximate optimal solutions of the SP problems with small $\epsilon $, are. However, to the best of our knowledge, no algorithms
for finding such approximating solutions (in case fast oscillations may lead to a significant improvement of the performance) have been discussed in the literature. In this paper, we fill this gap by developing an apparatus for construction of such  algorithms, our development being based
  on results of \cite{FinGaiLeb}, \cite{GQ},  \cite{GQ-1} and  \cite{GR} establishing the equivalence of optimal control
problems to certain IDLP problems (related results on linear programming formulations of optimal control problems in both deterministic and stochastic settings can be found
in \cite{Borkar1}, \cite{Borkar}, \cite{BhBo}, \cite{BGQ}, \cite{Evans}, \cite{F-V}, \cite{Goreac-Serea}, \cite{Her-Her-Lasserre}, \cite{Kurtz}, \cite{Lass-Trelat},
\cite{Rubio}, \cite{Stockbridge}, \cite{Stockbridge1} and \cite{Vinter}).

The paper is organized as follows. It consists of five parts.  Part I (Sections \ref{Sec-Contents} - \ref{Sec-Preliminaries})
is introductory. Section \ref{Sec-Contents} is this description of the contents of the paper. In Section \ref{Sec-Two-examples}, we consider two  examples of SP optimal control problems, in which fast oscillations lead to improvements of the performance indexes. Near optimal solutions of these problems obtained with the proposed technique are exhibited  later in the text (Section \ref{Sec-construction-SP-examples}). In Section \ref{Sec-Preliminaries} some notations and definitions used in the paper are introduced.

In Part II (Sections \ref{Sec-Augm-reduced} and \ref{Sec-Ave-Aug}), we build a foundation for the subsequent developments  by considering two problems that describe an asymptotic behavior of the IDLP
problem related to the SP optimal control problem (\ref{Vy-perturbed}). One is the augmented reduced IDLP problem obtained
via adding some extra constraints to the problem resulted from  equating of the small parameter to zero (Section \ref{Sec-Augm-reduced}) and the other
is the \lq\lq averaged" IDLP problem, which  is related to the averaged problem of optimal control (Section \ref{Sec-Ave-Aug}). We show that
these two problems are
equivalent   and that both of them characterize the limit behavior of the SP problem when $\epsilon\rightarrow 0$ provided that the slow dynamics of the SP system is approximated by the averaged system on finite time intervals (see Definition \ref{Def-Average-Approximation} and  Propositions \ref{Prop-ave-disc}, \ref{Prop-ave-disc-1},  \ref{Prop-SP-2}, \ref{Prop-present-3}).

In Part III (Sections \ref{Sec-ACG-nec-opt} - \ref{Sec-ACG-construction}), we introduce the concept of an
average control generating (ACG) family (the key building block of the paper), and we use duality results for IDLP problems involved and their semi infinite approximations to characterize and construct optimal and near optimal ACG families. More specifically,
in Section \ref{Sec-ACG-nec-opt}, the definitions of an ACG family and of  optimal/ near optimal ACG families are given (Definitions \ref{Def-ACG} and \ref{Def-ACG-opt}). Also in this section, averaged and associated dual problems are introduced and a necessary optimality
condition for an ACG family to be optimal   is  established under the assumption that solutions  of these duals exist (Proposition \ref{Prop-necessary-opt-cond}). In Section \ref{N-approx-dual-opt}, approximating averaged semi infinite LP problem and the corresponding approximating averaged and associated dual problems are introduced. In Section \ref{Sec-Existence-Controllability} it is proved that solutions of these approximating dual problems exist under natural controllability
conditions (Propositions \ref{Prop-existence-disc} and \ref{dual-existence-average}). In Section \ref{Sec-ACG-construction}, it is established that
solutions of the approximating averaged and associated dual problems can be used for construction of near optimal ACG families (Theorem \ref{Main-SP-Nemeric}).

In Part IV (Sections \ref{Sec-SP-ACG-theorem} - \ref{Sec-LP-based-algorithm}), we indicate a way how asymptotically optimal  (near optimal)  controls of SP problems can be constructed  on the basis of  optimal (near optimal) ACG families. In Section \ref{Sec-SP-ACG-theorem}, we describe the construction of a control of the SP system based on an ACG family and establish its asymptotic optimality/near optimality if the ACG family is optimal/near optimal (Theorem \ref{Prop-convergence-measures-discounted} and Corollary \ref{Cor-asym-near-opt}). In Section \ref{Sec-construction-SP-examples}, we discuss the process
of construction of asymptotically near optimal controls  using solutions
of the approximating  averaged and associated dual problems, and we illustrate the construction with two numerical examples.   A linear programming based algorithm allowing one to find solutions of approximating averaged problem and solutions of the corresponding approximating (averaged and associated) dual problems numerically is outlined in Section \ref{Sec-LP-based-algorithm}.

 Part V (Sections \ref{Sec-Phi-Map} - \ref{Sec-Main-AVE}) contains some  technical proofs. Namely, the proofs of Propositions \ref{Prop-ave-disc} and  \ref{Prop-SP-2} are given in Section \ref{Sec-Phi-Map}, and
the proofs of Theorems \ref{Main-SP-Nemeric} and \ref{Prop-convergence-measures-discounted}
are given in Sections \ref{Sec-Main-Main} and \ref{Sec-Main-AVE}, respectively.

\section{Two examples}\label{Sec-Two-examples}\

EXAMPLE 1. Consider the  optimal control problem
\begin{equation}\label{e:EXAMPL2-4}
\inf_{u(\cdot)}\int_0^{\infty}e^{-0.1t}(u_1(t)^2+u_2(t)^2+y_1(t)^2+y_2(t)^2+z^2(t))dt \BYDEF V_{di}^*(\epsilon , y_0,z_0) ,
\end{equation}
 with minimization being over the controls $(u_1(\cdot),u_2(\cdot)) $,
\begin{equation}\label{e:ex-4-1-rep-1-0}
(u_1(t),u_2(t))\in U\BYDEF \{(u_1,u_2)\ : \ |u_i|\leq 1, \ i=1,2\},
\end{equation}
   and the corresponding solutions $(y(\cdot), z(\cdot)) $ of the
SP  system
\begin{equation}\label{e:ex-4-2-repeat}
\epsilon y'_i(t)= -y_i(t) + u_i(t), \ \ \ i=1,2,
\end{equation}
\vspace{-.3in}
\begin{equation}\label{e:ex-4-1-rep}
z'(t)= -y_1(t)u_2(t) + y_2(t)u_1(t) \ ,
\end{equation}
where
\vspace{-.2in}
$$
y(0)=y_0 =(y_{1,0},y_{1,0}) =(0.5, 0.5), \ \ \ \ \ z(0)=z_0=2
$$
and
\vspace{-.2in}
$$
(y_1(t),y_2(t))\in Y=\{(y_1,y_2)\ : \ |y_i|\leq 1, \ i=1,2\}\ , \ \ \ \ z(t)\in
Z=\{z\ : \ |z|\leq 2.5 \}  \ .
$$
By taking $\epsilon = 0$ in (\ref{e:ex-4-2-repeat}), one obtains that $y_i(t) = u_i(t), \ i=1,2, $ and, thus, arrives at the equality
 \begin{equation}\label{e:ex-4-1-rep-1}
  -y_1(t)u_2(t)+y_2(t)u_1(t)= 0 \ \ \ \forall t ,
\end{equation}
which  makes the slow dynamics  uncontrolled and leads to the equality $\ z(t) = z_0=2 \ \forall t >0 $. The latter, in turn, implies that
$$
V_{di}^*(0, y_0,z_0) =\inf_{u(\cdot)}\int_0^{\infty}e^{-0.1t}(2u_1^2(t)+2u_2^2(t)+ z_0^2)dt= 10 z_0^2=40
$$
To see that this value is not even approximately optimal for small but non-zero $\epsilon $, let us consider the controls
$\
\bar{u}_1(t)= cos(\frac{t}{\epsilon}), \ \  \bar{u}_2(t)= sin(\frac{t}{\epsilon})
$.
The solution of the SP system (\ref{e:ex-4-2-repeat}), (\ref{e:ex-4-1-rep}) obtained with this control can be verified to be of the form
$$
\bar{y}_1(t)= \frac{1}{2}sin(\frac{t}{\epsilon}) - \frac{1}{2}cos(\frac{t}{\epsilon}) +O(e^{-\frac{t}{\epsilon}}), \ \ \ \ \ \bar{y}_2(t)= -\frac{1}{2}sin(\frac{t}{\epsilon}) - \frac{1}{2}cos(\frac{t}{\epsilon})+O(e^{-\frac{t}{\epsilon}}), \ \ \ \  \bar{z}(t)= 2 - \frac{1}{2}t + O(\epsilon),
$$
with the slow state variable $z(t)$ decreasing in time and reaching zero at the moment $\bar t(\epsilon)= 4 +O(\epsilon) $. The value of the objective function  obtained with using these controls until the moment when the slow component reaches zero and with applying \lq\lq zero controls" after that moment is equal to $10.3+ O(\epsilon)$.
Thus, in the given example,
$
\ \lim_{\epsilon\rightarrow 0}V_{di}^*(\epsilon, y_0,z_0) < V_{di}^*(0, y_0,z_0).
$

Controls $u_{i, \epsilon}(t), \ i=1,2, $ and the corresponding state components $y_{i, \epsilon}(t), \ i=1,2, $ and  $\ z_{ \epsilon}(t)$ (that are verified
to be \lq\lq near optimal" in the given example) have been numerically constructed  with the help of proposed technique
 for $\epsilon=0.1 $ (see Figures 1,2,3 and 7 in Section \ref{Sec-construction-SP-examples}) and for $\epsilon=0.01 $ (see Figures 4,5,6 and 8 in Section \ref{Sec-construction-SP-examples}). The corresponding  values of the objective function obtained with these two values of  $\epsilon $ are approximately equal to $8.56 $ and to $ 8.54 $ (respectively).

EXAMPLE 2. Assume that the fast dynamics and the controls are as in Example 1 (that is, they are described by (\ref{e:ex-4-1-rep-1-0})
and (\ref{e:ex-4-2-repeat})). Assume that the slow dynamics  is two-dimensional and is described by the equations
\begin{equation}\label{e:ex-4-1-rep-101-1}
z_1'(t)=z_2(t), \ \ \ \ \ \ \ z_2'(t)= - 4z_1(t) - 0.3z_2(t)-y_1(t)u_2(t)+y_2(t)u_1(t),
\end{equation}
\vspace{-.2in}
with
$$
(z_1(t), z_2(t))\in Z \BYDEF \{(z_1,z_2)\ : \ |z_1|\leq 2.5, \ \ |z_1|\leq 4.5\}.
$$
Consider the periodic optimization problem
\begin{equation}\label{e:ex-4-1-rep-101-2}
\inf_{T, u_1(\cdot),u_2(\cdot)} \frac{1}{T}\int_0^T (0.1u_1^2(t)+ 0.1u_2^2(t) -z_1^2(t))dt
= V_{per}^*(\epsilon),
\end{equation}
where minimization is over the length of the time interval $T$ and over the controls $u_i(\cdot), \ i=1,2, $ defined on this interval subject to
the periodicity conditions: $y_i(0)=y_i(T), \ i=1,2, $ and $z_i(0)=z_i(T), \ i=1,2 $.

As in Example 1, equating $\epsilon$ to zero leads to (\ref{e:ex-4-1-rep-1}), which makes the slow dynamics uncontrollable. The optimal periodic solution
of (\ref{e:ex-4-1-rep-101-1}) in this case is the \lq\lq trivial" steady state  one:
$ \ u_1(t)=u_2(t)=y_1(t)=y_2(t)=z_1(t)=z_2(t)= 0 \ \  \forall t $, which leads one to the conclusion that
 $V_{per}^*(0)=0 $.

Note that the \lq\lq slow subsystem" (\ref{e:ex-4-1-rep-101-1}) is equivalent to the second order differential equation
$\ z''(t) + 0.3z'(t) + 4z(t) =  -y_1(t)u_2(t)+y_2(t)u_1(t) \ $ that describes a linear oscillator influenced by the controls and the fast state variables. One can expect, therefore, that, if, due to a combined influence of the latter, some sort of  resonance oscillations of the slow state variables are achievable, then the value of the objective can be negative (see Example 1 in \cite{GR}).
This type of a near optimal oscillatory regime (with rapid oscillations of $u_{i, \epsilon}(t), \ i=1,2, $ and
$y_{i, \epsilon}(t), \ i=1,2, $ and with slow oscillations of $\ z_{i, \epsilon}(t), \ i=1,2, $) was obtained  with the use of the proposed technique.
The images of  state trajectories constructed numerically for $\epsilon=0.01 $ and $\epsilon=0.001 $  are depicted in Figures 9 and 10
 in Section \ref{Sec-construction-SP-examples}. The  values of the objective function for these two cases  are approximately equal $-1.177$.

\section{Some notations and definitions}\label{Sec-Preliminaries}
Given a compact metric space $W$, $\mathcal{B}(W)$ will stand for
the $\sigma$-algebra of its Borel subsets and $\mathcal{P}(W)$
will denote the set of probability measures defined on
$\mathcal{B}(W)$. The set $\mathcal{P}(W)$ will always be treated
as a compact metric space with a metric $\rho$,
 which is  consistent
 with its weak$^*$  topology. That is, a sequence
$\gamma^k \in \mathcal{P}(W), k =1,2,... ,$ converges to $\gamma
\in \mathcal{P}(W)$ in this metric if and only if
$$
 \lim_{k\rightarrow \infty}\int_{W} \phi(w) \gamma^k (dw) \ = \
 \int_{W} \phi(w) \gamma (dw),
$$
for any continuous $\phi(w): W \rightarrow \mathbb{R}^1$.
There are many ways of how such a metric $\rho$ can be defined. In this paper, we will use the following definition: $\forall \gamma',\gamma''  \in \mathcal{P}(W)$,
\begin{equation}\label{e:intro-2}
\rho(\gamma', \gamma'') \BYDEF
\sum^\infty_{l=1} \frac{1}{2^l} \; \mid \int_{W} q_l(w) \gamma'(dw)
-  \int_{W} q_l(w) \gamma''(dw)\mid \ ,
\end{equation}
where  $q_l(\cdot), l=1,2,... \ ,$ is a sequence
of Lipschitz continuous functions which is dense in the unit ball
of $C(W)$ (the space of continuous functions on $W$).

Using this metric $\rho$,
one can define the Hausdorff metric
$\rho_H$ on the set of subsets of $\mathcal{P}(W)$ as follows: $\forall \Gamma_i \subset \mathcal{P}(W) \ , \ i=1,2 \ ,$
\vspace{-0.15cm}
\begin{equation}\label{e:intro-3}
\rho_H(\Gamma_1, \Gamma_2) \stackrel{def}{=} \max \{\sup_{\gamma \in
\Gamma_1} \rho(\gamma,\Gamma_2), \sup_{\gamma \in \Gamma_2}
\rho(\gamma,\Gamma_1)\}, \
\end{equation}
where $\rho(\gamma, \Gamma_i) \BYDEF
\inf_{\gamma' \in \Gamma_i} \rho(\gamma,\gamma') \ .$
Note that, although, by some abuse of terminology,  we refer to
$\rho_H(\cdot,\cdot)$ as to a metric on the set of subsets of
${\cal P} (Y \times U)$, it is, in fact, a semi metric on this set
(since $\rho_H(\Gamma_1, \Gamma_2)=0$ is equivalent to  $\Gamma_1
= \Gamma_2$ if and only if $\Gamma_1$ and $\Gamma_2$ are closed).
It can be verified (see e.g. Lemma $\Pi$2.4 in \cite{Gai0},  p.205)
that, with the definition of the metric $\rho$  as in (\ref{e:intro-2}),
\begin{equation}\label{e:intro-3-1}
\rho_H(\bar{co} \Gamma_1, \bar{co} \Gamma_2) \leq \rho_H(\Gamma_1, \Gamma_2) \ ,
\end{equation}
where $\bar{co}$ stands for the closed convex hull of the corresponding set.

 Given a measurable function $w(\cdot): [0,\infty) \rightarrow W$, the {\it occupational measure} generated by
 this function on the interval $[0,S] $ is the probability measure $\gamma^{w(\cdot),S}\in \mathcal{P}(W)$ defined by the equation
 \begin{equation}\label{e:occup-S}
\gamma^{w(\cdot),S} (Q) \BYDEF \frac{1}{S} \int_0^{S}  1_Q(w(t))dt ,
 \ \ \forall Q \in \mathcal{B}(W),
\end{equation}
where $1_Q(\cdot) $ is the indicator function. The occupational measure  generated by
 this function on the interval $[0,\infty) $  is the probability measure $\gamma^{w(\cdot)}\in \mathcal{P}(W)$ defined as the limit (assumed to exist)
 \begin{equation}\label{e:occup-S-infty}
\gamma^{w(\cdot)} (Q) \BYDEF \lim_{S\rightarrow\infty}\frac{1}{S} \int_0^{S}  1_Q(w(t))dt ,
 \ \ \forall Q \in \mathcal{B}(W).
\end{equation}
Note that  (\ref{e:occup-S}) is equivalent to that
\begin{equation}\label{e:oms-0-1}
\int_{W} q(w) \gamma^{w(\cdot),S}(dw) = \frac{1}{S} \int _0 ^ S
q (w(t)) dt
\end{equation}
for any $q(\cdot)\in C(W)$, and
(\ref{e:occup-S-infty}) is equivalent to that
\begin{equation}\label{e:oms-0-1-infy}
\int_{W} q(w) \gamma^{w(\cdot)}(dw) =  \lim_{S\rightarrow\infty}\frac{1}{S} \int _0 ^ S q (w(t)) dt
\end{equation}
 for any $q(\cdot)\in C(W)$.

The {\it discounted occupational measure} generated by $w(\cdot) $ is the probability measure $\gamma_{di}^{w(\cdot)}\in \mathcal{P}(W)$  defined by the equation
\begin{equation}\label{e:occup-C}
\gamma_{di}^{w(\cdot)} (Q) \BYDEF C \int_0^{\infty} e^{-Ct} 1_Q(w(t))dt\ ,
 \ \ \forall Q \in \mathcal{B}(W),
\end{equation}
the latter being equivalent to that
\begin{equation}\label{e:oms-0-2}
\int_{W} q(w) \gamma_{di}^{w(\cdot)}(dw) = C \int _0 ^ \infty
e^{-C t} q (w(t) ) dt
\end{equation}
for any $q(\cdot)\in C(W)$.

\bigskip

{\bf II. Augmented reduced and averaged IDLP problems.}

\section{Family of IDLP problems related the SP problem  and the  augmented reduced IDLP problem}\label{Sec-Augm-reduced}
Denote by $\Gamma_{di}(\epsilon , y_0,z_0) $
 the set of discounted occupational measures generated by all controls
$u(\cdot) $ and the corresponding solutions $(y_{\epsilon}(\cdot),z_{\epsilon}(\cdot)) $ of the SP system (\ref{e:intro-0-1})-(\ref{e:intro-0-2}).
That is,
$$
\Gamma_{di}(\epsilon , y_0,z_0) \BYDEF \bigcup_{u(\cdot)}\{ \gamma^{u(\cdot), y_{\epsilon}(\cdot),z_{\epsilon}(\cdot)}_{di}\}
\subset \mathcal{P}(U\times Y\times Z),
$$
where $\gamma^{u(\cdot), y_{\epsilon}(\cdot),z_{\epsilon}(\cdot)}_{di} $ is the occupational measure generated by
$(u(\cdot), y_{\epsilon}(\cdot),z_{\epsilon}(\cdot)) $ and the union is over all  controls.

Based on (\ref{e:oms-0-2}), one can rewrite the
problem (\ref{Vy-perturbed}) in terms of minimization over measures from the set $\Gamma_{di}(\epsilon , y_0,z_0) $ as follows
\begin{equation}\label{e:oms-4}
 \inf_{\gamma\in \Gamma_{di}(\epsilon , y_0,z_0)}\int_{U\times Y\times Z} G(u,y,z)\gamma(du,dy,dz) = CV_{di}^*(\epsilon, y_0, z_0),
\end{equation}
the latter implying (due to linearity of the objective function in (\ref{e:oms-4})) that
\begin{equation}\label{e:oms-4-1}
 \min_{\gamma\in \bar{co}\Gamma_{di}(\epsilon , y_0,z_0)}\int_{U\times Y\times Z} G(u,y,z)\gamma(du,dy,dz) = CV_{di}^*(\epsilon, y_0, z_0).
\end{equation}
Let us also consider the parameterized by $\epsilon $ family of IDLP problems
\begin{equation}\label{SP-IDLP-di}
 \min _{\gamma \in \mathcal{D}_{di}(\epsilon , y_0, z_0)}\int_{U \times Y\times Z }G(u,
y,z)\gamma(du,dy,dz)\BYDEF G^*_{di}(\epsilon , y_0, z_0),
\end{equation}
where the set $\mathcal{D}_{di}(\epsilon ,  y_0, z_0)\subset \mathcal{P}(U \times Y\times Z) $ is defined by the equation
\begin{equation}\label{e:SP-W-1}
\begin{array}{c}
\displaystyle \mathcal{D}_{di}(\epsilon ,  y_0, z_0) \BYDEF  \{\gamma \in \mathcal{P}(U \times Y\times Z) \ :
\\[12pt]
\displaystyle \int_{U\times Y\times Z} [\nabla (\phi(y)\psi(z) )^T \chi_{\epsilon}(u,y,z) \
+ \ C (\phi(y_0)\psi(z_0)-\phi(y)\psi(z) )]\gamma(du,dy,dz) =  0
 \\[12pt]
\forall \phi(\cdot)\in C^1 (\R^m), \ \ \forall \psi(\cdot)\in C^1 (\R^n) \} , \ \ \ \ \ \ \ \ \ \chi_{\epsilon}(u,y,z)^T\BYDEF ( \frac{1}{\epsilon}\ f(u,y,z)^T,
\ g(u,y,z)^T).
\end{array}
\end{equation}
Note that both the objective and the constraints above are linear in the \lq\lq decision variable" $\gamma$ (that is, these problems are indeed belong to the class
of infinite dimensional LP problems; see \cite{And}).

From Proposition 2.2 of  \cite{GQ} it follows that
\begin{equation}\label{e:oms-6-0}
\bar{co}\Gamma_{di}(\epsilon , y_0,z_0)\subset \mathcal{D}_{di}(\epsilon ,  y_0, z_0) \ \ \ \ \ \Rightarrow \ \ \ \ \ C V_{di}^*(\epsilon , y_0, z_0) \geq
 G^*_{di}(\epsilon , y_0, z_0).
\end{equation}
Moreover, by Theorem 4.4 of \cite{GQ} (see also Theorem 2.2 in \cite{GQ-1}), under mild conditions, the
equality
\begin{equation}\label{e:oms-6}
\bar{co}\Gamma_{di}(\epsilon , y_0,z_0)= \mathcal{D}_{di}(\epsilon ,  y_0, z_0),
\end{equation}
is valid,
which implies that
 \begin{equation}\label{C-result-di-SP}
  C V_{di}^*(\epsilon , y_0, z_0) = G^*_{di}(\epsilon , y_0, z_0)
\end{equation}
(the validity of the latter for an arbitrary $G(u,y,z) $ being equivalent to (\ref{e:oms-6})).

Due to (\ref{e:oms-6-0}) and (\ref{e:oms-6}), (\ref{C-result-di-SP}), one can use the IDLP problem (\ref{SP-IDLP-di})
to obtain some asymptotic properties of (\ref{e:oms-4-1}). To this end, let us
 multiply the constraints in (\ref{e:SP-W-1}) by $\epsilon$
and take  into account the fact that
$$ \ \nabla (\phi(y)\psi(z) )^T  = (\psi(z)\nabla \phi(y)^T, \ \phi(y)\nabla \phi(z)^T ). $$
This will lead us to the following representation for the set $\mathcal{D}_{di}(\epsilon , y_0, z_0)$:
\begin{equation}\label{e:SP-W-1-1}
\begin{array}{c}
\displaystyle \mathcal{D}_{di}(\epsilon ,  y_0, z_0) \BYDEF  \{\gamma \in \mathcal{P}(U \times Y\times Z) \ :
\\[12pt]
\displaystyle \int_{U\times Y\times Z} [ \ \psi(z)\nabla \phi(y)^T f(u,y,z) + \epsilon \
[\ \phi (y)\nabla \psi(z)^T g(u,y,z)\
 \\[12pt]
+ \ C (\ \phi(y_0)\psi(z_0)-\phi(y)\psi(z) \ )\ ]\ ]\gamma(du,dy,dz) =  0 \ \ \ \ \ \forall \phi(\cdot)\in C^1 (\R^m), \ \ \forall \psi(\cdot)\in C^1 (\R^n) \}.
\end{array}
\end{equation}
Taking $\epsilon = 0 $ in the expression above, one arrives at the set
\begin{equation}\label{e:SP-W-3}
\begin{array}{c}
 \displaystyle \mathcal{D}\BYDEF \{\gamma \in \mathcal{P}(U \times Y\times Z) \ :
\\[12pt]
\displaystyle \int_{U\times Y\times Z} [ \ \psi(z)\nabla \phi(y)^T f(u,y,z)\ ]\gamma(du,dy,dz) =  0 \ \ \ \ \ \forall \phi(\cdot)\in C^1 (\R^m),
\ \ \forall \psi(\cdot)\in C^1 (\R^n) \}
\end{array}
\end{equation}
that does not depend on   the initial conditions $y_0, z_0 $ (as well as  on the value of the discount factor $C$).

It is easy to see that $\ \lim\sup_{\epsilon\rightarrow 0}\mathcal{D}_{di}(\epsilon , y_0, z_0)\subset \mathcal{D} $. In general case, however,
 $  \mathcal{D}_{di}(\epsilon ,  y_0, z_0)\nrightarrow \mathcal{D} $ when
$\epsilon\rightarrow 0 $, this being due to the fact that the equalities  defining the set $  \mathcal{D}_{di}(\epsilon ,  y_0, z_0)$ contain
some \lq\lq implicit" constraints that are getting lost with equating $\epsilon $ to  zero.
 In fact, by taking $\ \phi(y)=1  $,  one can see that, if $ \gamma$ satisfies
the equalities in (\ref{e:SP-W-1-1}), then it also satisfies the equality
\begin{equation}\label{e:SP-W-2-1-1}
 \int_{U\times Y\times Z}[\ \nabla \psi(z)^T g(u,y,z)\ + \ C (\ \psi(z_0)-\psi(z) \ )\ ]\gamma(du,dy,dz) =  0
 \ \ \forall \psi(\cdot)\in C^1 (\R^n)
\end{equation}
for any $\ \epsilon > 0$.
That is,
\begin{equation}\label{e:SP-W-2-1-1-0}
\mathcal{D}_{di}(\epsilon ,  y_0, z_0) = \mathcal{D}_{di}(\epsilon ,  y_0, z_0)\cap \mathcal{A}_{di}(z_0)\ \ \ \ \ \forall \epsilon > 0,
\end{equation}
where
\begin{equation}\label{e:SP-M}
\begin{array}{c}
 \displaystyle \mathcal{A}_{di}(z_0) \BYDEF\{\gamma \in \mathcal{P}(U \times Y\times Z) \ :
\\[10pt]
 \displaystyle \int_{U\times Y\times Z}[\ \nabla \psi(z)^T g(u,y,z)\ + \ C (\ \psi(z_0)-\psi(z) \ )\ ]\gamma(du,dy,dz) =  0
 \ \ \forall \psi(\cdot)\in C^1 (\R^n).
\end{array}
\end{equation}
Define the set $\mathcal{D}_{di}^{\mathcal{A}}(z_0)$ by the equation
\begin{equation}\label{e:SP-M-0}
\mathcal{D}_{di}^{\mathcal{A}}(z_0)\BYDEF \mathcal{D}\cap \mathcal{A}_{di}(z_0)
\end{equation}
and consider the IDLP problem
\begin{equation}\label{SP-IDLP-0}
  \min _{\gamma \in \mathcal{D}_{di}^{\mathcal{A}}(z_0) }\int_{U \times Y\times Z }G(u,
y,z)\gamma(du,dy,dz) \BYDEF G_{di}^{\mathcal{A}}(z_0).
\end{equation}
We will be referring to this problem as to {\it augmented reduced} IDLP problem (the term {\it  reduced} problem is commonly used for the problem
obtained from a perturbed family by equating the small parameter to zero).
\begin{Proposition}\label{Prop-SP-1}
The following relationships are valid:
\begin{equation}\label{SP-IDLP-convergence-1}
 \limsup_{\epsilon\rightarrow 0}\mathcal{D}_{di}(\epsilon , y_0, z_0)\subset \mathcal{D}_{di}^{\mathcal{A}}(z_0),
\end{equation}
\begin{equation}\label{SP-IDLP-convergence-1-1}
 \liminf_{\epsilon\rightarrow 0}G_{di}^*(\epsilon , y_0, z_0)\geq  G_{di}^{\mathcal{A}}(z_0).
\end{equation}
\end{Proposition}
\begin{proof}
The validity of (\ref{SP-IDLP-convergence-1}) is implied by (\ref{e:SP-W-2-1-1-0}), and the validity of (\ref{SP-IDLP-convergence-1-1})
 follows from (\ref{SP-IDLP-convergence-1}).
 \end{proof}

 Note that the set $\mathcal{D}_{di}^{\mathcal{A}}(z_0)$ allows another representation which makes use of the fact that
an arbitrary
 $\gamma\in\mathcal{P}(U \times Y\times Z)$ can be  \lq\lq disintegrated"
  as follows
\begin{equation}\label{e:SP-W-4-extra}
\gamma(du,dy,dz) = \mu(du,dy|z)\nu(dz),
\end{equation}
where  $ \ \nu(dz)\BYDEF \gamma(U\times Y,dz)\ $ and where $\ \mu(\cdot|z)$ is a probability measure on
Borel subsets of $U\times Y $ uniquely defined  for $\nu$-almost all $z\in Z$ (with $\mu(A|z) $ being
Borel measurable on $Z$ for any $A\subset U\times Y $).

\begin{Proposition}\label{Prop-present-2}
The  set $\mathcal{D}_{di}^{\mathcal{A}}(z_0)$ can be represented in the form:
\begin{equation}\label{tilde-W-1}
\begin{array}{c}
\mathcal{D}_{di}^{\mathcal{A}}(z_0)= \{\gamma = \mu(du,dy|z)\nu(dz) \ : \ \mu(\cdot  |z)\in W(z) \ \ for \ \  \nu - almost \ all \ z\in Z,
\\[12pt]
 \int_{Z}\ [\ \nabla \psi(z)^T \tilde g(\mu(\cdot |z),z) + C (\ \psi(z_0)-\psi(z) \ ) ]\nu(dz)
 = 0 \ \ \forall   \ \psi(\cdot)\in C^1 (\R^n)\},
\end{array}
\end{equation}
where $W(z)\subset \mathcal{P}(U \times Y) $ is defined by the equation
\begin{equation}\label{e:2.4}
\begin{array}{c}
\displaystyle
W(z) \BYDEF  \{\mu \in \mathcal{P}(U \times Y) \ : \
\int_{U\times Y} \nabla\phi(y)^{T} f(u,y,z)
\mu(du,dy) =  0 \ \ \forall \phi(\cdot) \in C^1(\R^m) \} \
\end{array}
\end{equation}
and where  $\  \tilde g(\mu(\cdot |z),z)= \int_{U\times Y}g(u,y,z)\mu(du,dy |z)$.
\end{Proposition}
\begin{proof}
Let $\gamma$ belong to the right hand side  of (\ref{tilde-W-1}).  Then, by (\ref{e:SP-W-4-extra}),
$$
\int_{U\times Y\times Z} \psi(z)\nabla \phi(y)^T f(u,y,z)\gamma(du,dy,dz)
$$
$$
=
\int_{Z}\ [\psi(z)\int_{U\times Y} \nabla \phi(y)^T f(u,y,z)\mu(du,dy|z)\ ]\nu(dz)\ = \ 0
$$
 for any $\phi(\cdot)\in C^1 (\R^m)$ and for any $ \psi(\cdot)\in C^1 (\R^n)$ (the equality to zero being due to the fact that $ \mu(\cdot , \cdot |z)\in W(z)$ for $\nu$- almost all $z\in Z$).
 Also,
 $$
 \int_{U\times Y\times Z}[\ \nabla \phi(z)^T g(u,y,z)\ + \ C (\ \psi(z_0)-\psi(z) \ )\ ]\gamma(du,dy,dz)
 $$
$$
= \int_{Z}\ [\ \nabla \phi(z)^T \int_{U\times Y}  g(u,y,z) \mu(du,dy|z)\ + \ C (\ \psi(z_0)-\psi(z) \ )\ ]\nu(dz) \ = \ 0.
$$
These imply that $ \ \gamma\in\mathcal{D}_{di}^{\mathcal{A}}(z_0)$.
Assume now that $\gamma\in \mathcal{D}_{di}^{\mathcal{A}}(z_0) $. That is, $\ \gamma\in \mathcal{D}$  and
$\gamma\in \mathcal{A}_{di}(z_0)  $ (see (\ref{e:SP-W-3}),  (\ref{e:SP-M}) and (\ref{e:SP-M-0})). Using the fact
that   $ \gamma\in \mathcal{D}$ (and taking into account the disintegration (\ref{e:SP-W-4-extra})), one can obtain that
\begin{equation}\label{e:SP-W-5}
\int_{z\in Z} [\psi(z)\ \int_{U\times Y} \nabla \phi(y)^T f(u,y,z)\mu(du,dy|z)\ ]  \nu(dz) =  0,
\end{equation}
the latter implying  that
$ \ \int_{U\times Y} \nabla \phi(y)^T f(u,y,z)\mu(du,dy|z) = 0 \ $ for $\nu$-almost all $z\in Z$ (due to the fact that $\psi(z)$ is an arbitrary continuously differentiable function). That is,
$\mu(\cdot |z)\in W(z) $ for $\nu$-almost all $z\in Z$. This, along with the inclusion $\gamma\in \mathcal{A}(z_0) $,
imply that
$\ \gamma$ belongs to the right hand side of (\ref{tilde-W-1}). This proves (\ref{tilde-W-1}).
\end{proof}

REMARK II.1. A possibility of the  presence of implicit constraints in families of finite-dimensional LP problems depending on
 a small parameter was noted in \cite{PG88}, where such families were called singularly perturbed. In \cite{PG88} it has been also shown that,
 under certain conditions,  \lq\lq the true limits" of the optimal value and of the optimal solutions set of such SP families of LP problems
 can be obtained by adding these implicit constraints to the set of constrains defining the feasible set with the
 zero value of the parameter (see Theorem 2.3, p. 149 in \cite{PG88} and also recent results in \cite{ABFG}).
  The IDLP problem (\ref{SP-IDLP-0}) is also constructed by augmenting  the set of constraints defining the set $\mathcal{D} $
  with the constraints defining the set $\mathcal{A}_{di}(z_0) $. A similar constraints augmentation was proposed in \cite{Bor-Gai-2} in
  dealing with the IDLP problem related to a problem of optimal control of SP stochastic differential equations, where a result similar
  to Proposition \ref{Prop-SP-1} has been established (Theorem 5.1 in \cite{Bor-Gai-2})  and where also it was shown, that under certain
  conditions (including the assumption that the matrices of coefficients near Brownian motions  are non-degenerate), the optimal value of the SP
  problem converges to the optimal value of the augmented reduced IDLP problem (Theorem 7.1 in \cite{Bor-Gai-2}). In the next section, we show that,
 in the purely deterministic setting  we are dealing with in this paper, the inclusion (\ref{SP-IDLP-convergence-1})
  and the inequality
  (\ref{SP-IDLP-convergence-1-1}) are replaced by equalities under the assumption that the averaged system
  approximates the SP system on finite time intervals (see Definition \ref{Def-Average-Approximation} and Proposition \ref{Prop-present-3}).


\section{Averaged optimal control problem; equivalence of the augmented reduced and the averaged IDLP problems }\label{Sec-Ave-Aug}
The system
 \begin{equation}\label{e:intro-0-3}
y'(\tau) = f(u(\tau),y(\tau), z)\ , \ \ \ \ \ \ z =
\mbox{const} \ ,
\end{equation}
is called {\it associated
system} (with respect to the SP system (\ref{e:intro-0-1})-(\ref{e:intro-0-1})).
Note that the associated system  looks similar to the \lq\lq fast" subsystem (\ref{e:intro-0-1}) but, in contrast
to (\ref{e:intro-0-1}), it is evolving in the \lq\lq stretched" time
scale $\tau=\frac{t}{\epsilon}$, with   $z$ being a vector of
fixed parameters. Everywhere in what follows, it is assumed that the associated system is  viable in $Y$ (see \cite{aubin0}).

 \begin{Definition}\label{Def-adm-associate}
 A pair $(u(\cdot), y(\cdot))$ will be called
{\it admissible for the associated system} if
(\ref{e:intro-0-3}) is satisfied for almost all $\tau $ ($u(\cdot) $ being measurable and $y(\cdot) $ being absolutely continuous functions) and if
 \begin{equation}\label{e:intro-0-3-1}
u(\tau)\in U, \ \ \ \ y(\tau)\in Y.
\end{equation}
\end{Definition}

Denote by $\mathcal{M}(z,S,y) $ the set of occupational measures generated on the interval $[0,S] $ by the admissible pairs of the
associated system that satisfy the initial
conditions $y(0)=y $  and denote by  $\mathcal{M}(z,S) $ the union of $\mathcal{M}(z,S,y) $ over all $y\in Y $. In \cite{Gai8} it has been established
 that, under mild conditions,
\begin{equation}\label{e:intro-0-3-2}
\lim_{S\rightarrow\infty}\rho_H(\bar{co}\mathcal{M}(z,S), W(z))=0,
\end{equation}
where $W(z)$ is defined in (\ref{e:2.4}) (see Theorem 2.1(i) in \cite{Gai8}), and also that, under some additional conditions (see Theorem 2.1(ii),(iii) and Proposition 4.1 in \cite{Gai8})),
\begin{equation}\label{e:intro-0-3-3}
\lim_{S\rightarrow\infty}\rho_H(\mathcal{M}(z,S,y), W(z))=0 \ \ \ \forall \ y\in Y,
\end{equation}
with the convergence being uniform with respect to $y\in Y $.

Define the function $\tilde g(\mu,z): \mathcal{P}(U \times Y)\times Z\rightarrow \R^n  $ by the equation
\begin{equation}\label{e:g-tilde}
\tilde{g}(\mu ,z) \BYDEF  \int_{U\times Y} g(u,y,z)
\mu(du,dy) \ \ \ \forall \mu\in \mathcal{P}(U \times Y)
 \end{equation}
and consider the system
\begin{equation}\label{e:intro-0-4}
z'(t)=\tilde{g}(\mu (t),z(t)) ,
\end{equation}
in which the role of controls  is played by measure valued functions $\mu(\cdot)$ that satisfy the inclusion
\begin{equation}\label{e:intro-0-5}
\mu (t) \in W(z(t)) ,
\end{equation}

The system (\ref{e:intro-0-4}) will be referred to as {\it the averaged
system}. In what follows, it is assumed that the averaged system is  viable in $Z$.
\begin{Definition}\label{Def-adm-averaged}
A pair $(\mu(\cdot), z(\cdot))$ will be referred to as
{\it admissible for the averaged system} if
(\ref{e:intro-0-4}) and (\ref{e:intro-0-5}) are satisfied for almost all $t$ ($\mu(\cdot) $ being measurable and $z(\cdot) $ being absolutely
continuous functions) and if
 \begin{equation}\label{e:intro-0-3-6}
 z(t)\in Z.
\end{equation}
\end{Definition}
From Theorem 2.6 of \cite{Gai-Ng} (see also Corollary 3.1 in  \cite{Gai8}) it follows that, under the assumption that (\ref{e:intro-0-3-3})
is satisfied (and under other technical assumptions including the Lipschitz continuity of the multi-valued map
$V(z)\BYDEF \cup_{\mu\in W(z)}\{\tilde g(\mu,z) \} $),
the averaged system approximates the SP dynamics  in the  sense that the following two statements are valid on any finite time interval $[0,T] $:

{\it (i) Let $u(\cdot) $ be a control and $(y_{\epsilon}(\cdot),z_{\epsilon}(\cdot))$ be the corresponding solution of the SP system  that
satisfies the initial condition (\ref{e-initial-SP}). There exists an admissible pair of the averaged system $(\mu(\cdot),z(\cdot)) $  satisfying the initial
condition
\begin{equation}\label{e:intro-0-3-6-1}
 z(0)=z_0
\end{equation}
such that
\begin{equation}\label{e:intro-0-3-7}
 \max_{t\in [0,T]}||z_{\epsilon}(t) -z(t)||\leq \alpha (\epsilon , T), \ \ \ \ \ {\rm where} \ \ \ \ \ \lim_{\epsilon\rightarrow 0}\alpha (\epsilon , T) = 0
\end{equation}
and such that, for any Lipschitz continuous functions $q(u,y,z,t) $,
\begin{equation}\label{e:intro-0-3-8}
 |\int_0^T q(u(t),y_{\epsilon}(t),z_{\epsilon}(t),t)dt - \int_0^T\tilde q(\mu(t),z(t),t)dt| \leq \alpha_q (\epsilon , T), \ \ \ \ \ {\rm where} \ \ \ \ \
 \lim_{\epsilon\rightarrow 0}\alpha_q (\epsilon , T) = 0
\end{equation}
and}
\begin{equation}\label{e:intro-0-3-9}
 \tilde{q}(\mu ,z,t) \BYDEF  \int_{U\times Y} q(u,y,z,t)
\mu(du,dy) \ \ \ \forall \mu\in \mathcal{P}(U \times Y).
\end{equation}
{\it (ii) Let $(\mu(\cdot),z(\cdot)) $ be an admissible pair of the averaged system satisfying the initial
condition (\ref{e:intro-0-3-6-1}). There exists a control $u(\cdot)  $ such that the solution  $(y_{\epsilon}(\cdot),z_{\epsilon}(\cdot))$
of the SP system obtained with this control and with the initial condition (\ref{e-initial-SP})   satisfies the relationships
 (\ref{e:intro-0-3-7}) and (\ref{e:intro-0-3-8}).}

 Note  that the validity of the statements (i) and (ii) was established in \cite{Gai-Ng} for the case when
 $q(u,y,z,t)=q(u,y,z) $. The  dependence on $t$ does not, however, affect the validity of the result (see
the proof of Theorem 2.6 in \cite{Gai-Ng}).
 Without going into technical details, let us  introduce the following definition that will be used in the sequel.

\begin{Definition}\label{Def-Average-Approximation} The averaged system will be said to approximate the SP system on  finite time intervals
 if the statements (i) and (ii) are valid on any interval $[0,T] $.
\end{Definition}

 Consider the problem
 \begin{equation}\label{Vy-ave-opt}
  \inf _{(\mu(\cdot), z(\cdot)) }\int_0 ^ { + \infty }e^{- C t}\tilde{G}(\mu(t), z(t))dt
 \stackrel{\rm def}{=} \tilde{V}^*_{di}(z_0),
\end{equation}
where
\begin{equation}\label{e:G-tilde}
\tilde{G}(\mu ,z) \BYDEF  \int_{U\times Y} G(u,y,z)
\mu(du,dy)
 \end{equation}
and $inf $ is sought over all admissible pairs of the averaged system that satisfy the initial
condition (\ref{e:intro-0-3-6-1}). This will be referred to as {\it the averaged problem}.

Denote by $\tilde{\Gamma}_{di}(z_0) $ the set of discounted occupational measures generated by the admissible pairs
of the averaged system satisfying the initial
condition (\ref{e:intro-0-3-6-1}). That is,
$$
\tilde{\Gamma}_{di}(z_0) \BYDEF \bigcup_{(\mu(\cdot), z(\cdot))}\{ p^{\mu(\cdot), z(\cdot)}_{di}\}
\subset \mathcal{P}(F),
$$
where $p^{\mu(\cdot), z(\cdot)}_{di} $ is the discounted occupational measure generated by
$(\mu(\cdot), z(\cdot))$ and the union is over all admissible pairs of the averaged system, $F $ being the graph of $W(\cdot) $
(see (\ref{e:2.4})):
\begin{equation}\label{e:graph-w}
F\BYDEF \{(\mu , z) \ : \ \mu\in W(z), \ \ z\in Z \} \ \subset \mathcal{P}(U\times Y)\times Z \ .
 \end{equation}

Using (\ref{e:oms-0-2}), one can rewrite the averaged
problem (\ref{Vy-ave-opt}) in terms of minimization over measures from the set $\tilde{\Gamma}_{di}(z_0) $ as follows
\begin{equation}\label{e:oms-4-di-ave}
 \inf_{p\in \tilde{\Gamma}_{di}(z_0)}\int_{F} \tilde G(\mu,z)p(d\mu,dz) = C \tilde{V}^*_{di}(z_0).
\end{equation}

To establish the relationships between the SP   and the averaged optimal control problems, let us introduce the
map $\Phi(\cdot):\mathcal{P}(F) \rightarrow \mathcal{P}(U \times Y \times Z)$
defined as follows.
 For any $p \in \mathcal{P}(F) $, let   $\Phi(p)\in \mathcal{P}(U \times Y \times Z)$ be such that
 \begin{equation}\label{e:h&th-1}
  \int_{U\! \times Y \times Z} \!q(u,y,z) \Phi(p) (du,dy,dz) =
\int_{F} \!\tilde q (\mu ,z)
p(d\mu,dz)\ \ \ \ \ \forall q(\cdot)\in \ C(U\times Y\times Z),
\end{equation}
where $ \tilde q (\mu,z) $ is as in (\ref{e:intro-0-3-9}) (this definition is legitimate since the right-hand side of the
above expression defines a linear continuous functional on $C(U\times Y\times Z) $, the latter being associated with an element
of $\mathcal{P}(U \times Y \times Z) $ that makes the equality (\ref{e:h&th-1}) valid).
Note that the map $\Phi(\cdot):\mathcal{P}(F) \rightarrow \mathcal{P}(U \times Y \times Z)$ is linear and it is continuous in the sense that
\begin{equation}\label{e-continuity-Psi}
\lim_{p_l\rightarrow p}\Phi(p_l)=\Phi(p),
\end{equation}
with $p_l$ converging to $p$ in the weak$^*$ topology of $\mathcal{P}(F)$ and $\Phi(p_l)$ converging to $\Phi(p)$
in the weak$^*$ topology of $\mathcal{P}(U \times Y \times Z)$
 (see Lemma 4.3 in \cite{Gai-Ng}).

 \begin{Proposition}\label{Prop-ave-disc}
If the averaged system approximates the SP system on finite time intervals, then
\begin{equation}\label{e-OccupSet-Convergence-Dis}
\lim_{\epsilon\rightarrow 0}\rho_H(cl \Gamma_{di}(\epsilon , y_0,z_0), cl \Phi(\tilde{\Gamma}_{di}(z_0)) ) = 0,
\end{equation}
where $cl$ stands for the closure of the corresponding set and
 $\ \Phi(\tilde{\Gamma}_{di}(z_0))\BYDEF \{\gamma \ : \ \gamma =\Phi(p), \ \ \ p\in \tilde{\Gamma}_{di}(z_0) \} $. Also,
\begin{equation}\label{e-Objective-Convergence-Dis}
\lim_{\epsilon\rightarrow 0}V_{di}^*(\epsilon , y_0, z_0)  = \tilde{V}^*_{di}(z_0).
\end{equation}
 \end{Proposition}
 \begin{proof} The proof is given in Section \ref{Sec-Phi-Map}.
 \end{proof}

Define  the set $\tilde{\mathcal{D}}_{di}(z_0)$ by the equation
\begin{equation}\label{D-SP-new}
\tilde{\mathcal{D}}_{di}(z_0)\BYDEF \{ p \in {\cal P} (F): \;
 \int_{F }(\nabla\psi(z)^T \tilde{g}(\mu,z) + C (\psi (z_0) - \psi (z) )) p(d\mu,dz) = 0 \ \ \ \forall \psi(\cdot) \in C^1(R^n)\}
  \end{equation}
and consider the IDLP problem
\begin{equation}\label{e-ave-LP-opt-di}
\min_{p\in \tilde{\mathcal{D}}_{di}(z_0)}\int_{F}\tilde G(\mu,z)p(d\mu , dz)\BYDEF \tilde{G}^*_{di}(z_0).
\end{equation}
This problems plays an important role in our consideration  and, for convenience,
we will be referring to it as to {\it the averaged} IDLP problem.

From Lemma 2.1 of \cite{GQ-1} it follows that
$$
\bar{co}\tilde{\Gamma}_{di}(z_0) \subset \tilde{\mathcal{D}}_{di}(z_0) \ \ \ \ \ \ \Rightarrow \ \ \ \ \ \ C \tilde{V}^*_{di}(z_0)\geq \tilde{G}^*_{di}(z_0).
$$
Also from Theorem 2.2 of \cite{GQ-1} it follows that, under certain conditions, the equality is valid
\begin{equation}\label{e-ave-LP-sets-di}
\bar{co}\tilde{\Gamma}_{di}(z_0)= \tilde{\mathcal{D}}_{di}(z_0),
\end{equation}
the latter implying  the equality
\begin{equation}\label{C-result-di-ave}
 C \tilde{V}_{di}^*(z_0) = \tilde{G}^*_{di}( z_0) .
\end{equation}

\begin{Proposition}\label{Prop-ave-disc-1}
Let the averaged system approximate the SP system on finite time intervals and let (\ref{e-ave-LP-sets-di}) be valid. Then
\begin{equation}\label{e-OccupSet-Convergence-Dis-LP}
\lim_{\epsilon\rightarrow 0}\rho_H(\bar{co} \Gamma_{di}(\epsilon , y_0,z_0), \Phi( \tilde{\mathcal{D}}_{di}(z_0)) = 0
\end{equation}
and
\begin{equation}\label{e-Objective-Convergence-Dis-LP}
C \lim_{\epsilon\rightarrow 0}V_{di}^*(\epsilon , y_0, z_0)  = \tilde{G}^*_{di}( z_0) .
\end{equation}
\end{Proposition}
\begin{proof} From (\ref{e:intro-3-1}) and (\ref{e-OccupSet-Convergence-Dis}) it follows that
\begin{equation}\label{e-OccupSet-Convergence-Dis-LP-1}
\lim_{\epsilon\rightarrow 0}\rho_H(\bar{co} \Gamma_{di}(\epsilon , y_0,z_0), \bar{co} \Phi(\tilde{\Gamma}_{di}(z_0)) = 0
\end{equation}
Due to continuity and linearity of $\Phi(\cdot) $,
$$
\bar{co} \Phi(\tilde{\Gamma}_{di}(z_0)) =  \Phi(\bar{co}\tilde{\Gamma}_{di}(z_0)).
$$
Hence, (\ref{e-OccupSet-Convergence-Dis-LP}) is impled by (\ref{e-ave-LP-sets-di}) and (\ref{e-OccupSet-Convergence-Dis-LP-1}).
Also, (\ref{e-Objective-Convergence-Dis-LP}) is implied by (\ref{e-Objective-Convergence-Dis}) and (\ref{C-result-di-ave}).
\end{proof}

\medskip

The following result establishes  that
the averaged IDLP problem (\ref{e-ave-LP-opt-di}) is equivalent to the augmented reduced IDLP
problem (\ref{SP-IDLP-0}).

\begin{Proposition}\label{Prop-SP-2}
 {\it The averaged and the augmented reduced  IDLP problems are equivalent in the sense that }
\begin{equation}\label{equality-1-SP-new}
\mathcal{D}_{di}^{\mathcal{A}}(z_0) =\Phi(\tilde{\mathcal{D}}_{di}(z_0)),
\end{equation}
\vspace{-.2in}
\begin{equation}\label{e-ave-LP-opt-1}
 G_{di}^{\mathcal{A}}(z_0)=\tilde{G}^*_{di}(z_0).
\end{equation}
Also, $\gamma = \Phi(p) $ is an optimal solution of the augmented reduced IDLP problem (\ref{SP-IDLP-0}) if and only if
$p$ is an optimal solution of the averaged IDLP problem (\ref{e-ave-LP-opt-di}).
\end{Proposition}
\begin{proof} The proof is given in  Section \ref{Sec-Phi-Map}.
\end{proof}

\begin{Corollary}\label{Cor-Important-Inequality}
The  inequality
\begin{equation}\label{e-imp-inequality-1}
\liminf_{\epsilon\rightarrow 0}V_{di}^*(\epsilon, y_0, z_0)\geq \frac{1}{C}\tilde{G}^*_{di}(z_0)
\end{equation}
is valid.
\end{Corollary}
\begin{proof} The proof follows from (\ref{e:oms-6-0}), (\ref{SP-IDLP-convergence-1-1}) and (\ref{e-ave-LP-opt-1}).
\end{proof}

Note that  Proposition \ref{Prop-SP-2} does not assume that the averaged system approximates  the
SP system on finite time intervals. If this assumption is made, then Proposition \ref{Prop-ave-disc-1}
in combination with Proposition \ref{Prop-SP-2} imply that the augmented problem (\ref{SP-IDLP-0})
defines the \lq\lq true limit" for the perturbed IDLP problem (\ref{SP-IDLP-di}) in the sense that the
 following statement strengthening Proposition
\ref{Prop-SP-1} is valid

\begin{Proposition}\label{Prop-present-3}
Let the averaged system approximate the SP system on finite time intervals and let the equalities (\ref{e:oms-6}), (\ref{e-ave-LP-sets-di}) be valid. Then
\begin{equation}\label{SP-IDLP-convergence-1-100}
 \lim_{\epsilon\rightarrow 0}\rho_H(\mathcal{D}_{di}(\epsilon , y_0, z_0),\mathcal{D}_{di}^{\mathcal{A}}(z_0))=0,
\end{equation}
\vspace{-.2in}
\begin{equation}\label{SP-IDLP-convergence-1-1-101}
 \lim_{\epsilon\rightarrow 0}G_{di}^*(\epsilon , y_0, z_0)=  G_{di}^{\mathcal{A}}(z_0).
\end{equation}
\end{Proposition}
\begin{proof}
The proof follows from Propositions \ref{Prop-ave-disc-1} and \ref{Prop-SP-2}.
\end{proof}

REMARK II.2. Results of this and the previous sections have their counterparts in dealing with the periodic
optimization problem
\begin{equation}\label{Vy-perturbed-per-new-1}
\displaystyle{
  \inf _{T, u(\cdot) }\frac{1}{T}\int_0 ^ {T}G(u(t), y_{\epsilon}(t),
z_{\epsilon}(t))dt} \stackrel{\rm def}{=}V_{per}^*(\epsilon) ,
\end{equation}
where $inf$ is sought over the length of the time interval $T$ and over the (defined on this interval ) controls $u(\cdot)$
 such that the corresponding solutions of the SP system (\ref{e:intro-0-1})-(\ref{e:intro-0-2})  satisfy the periodicity condition
$\ (y_{\epsilon}(0),
z_{\epsilon}(0)) =(y_{\epsilon}(T),
z_{\epsilon}(T))$. It is known (see Corollaries 3, 4 in \cite{GR} and Lemma 3.5 in \cite{FinGaiLeb}) that in the general case
\begin{equation}\label{Vy-perturbed-per-new-2}
G^*(\epsilon)\leq V_{per}^*(\epsilon)
\end{equation}
and, under certain additional assumptions,
\begin{equation}\label{Vy-perturbed-per-new-3}
G^*(\epsilon)= V_{per}^*(\epsilon),
\end{equation}
where $G^*(\epsilon) $ is the optimal value of the following IDLP problem
\begin{equation}\label{Vy-perturbed-per-new-4}
G^*(\epsilon )\BYDEF  \min _{\gamma \in \mathcal{D}(\epsilon)}\int_{U \times Y\times Z }G(u,
y,z)\gamma(du,dy,dz),
\end{equation}
in which
\vspace{-.2in}
\begin{equation}\label{Vy-perturbed-per-new-5}
\begin{array}{c}
\displaystyle \mathcal{D}(\epsilon ) \BYDEF  \{\gamma \in \mathcal{P}(U \times Y\times Z) \ :
\\[12pt]
\displaystyle \int_{U\times Y\times Z} \nabla (\phi(y)\psi(z) )^T \chi_{\epsilon}(u,y,z)
\gamma(du,dy,dz) =  0 \ \ \ \forall \phi(\cdot)\in C^1 (\R^m), \ \ \forall \psi(\cdot)\in C^1 (\R^n) \},
\end{array}
\end{equation}
with $ \ \chi_{\epsilon}(u,y,z) $ being as in (\ref{e:SP-W-1}) (that is, $\chi_{\epsilon}(u,y,z)^T\BYDEF ( \frac{1}{\epsilon}\ f(u,y,z)^T,
\ g(u,y,z)^T) \ $). Note that $\ \mathcal{D}(\epsilon) $ can be formally obtained from  $\ \mathcal{D}_{di}(\epsilon ,  y_0, z_0)$
by  taking $C=0$ in the expression for the latter (see (\ref{e:SP-W-1}); note the disappearance of the dependence on $z_0$ and $y_0$ in (\ref{Vy-perturbed-per-new-5})).

Similarly to Proposition \ref{Prop-SP-1}, one can come to the conclusion that
\begin{equation}\label{Vy-perturbed-per-new-6}
 \liminf_{\epsilon\rightarrow 0}G^*(\epsilon )\geq  G^{\mathcal{A}},
\end{equation}
where $ G^{\mathcal{A}}$ is the optimal value of the augmented reduced problem
\begin{equation}\label{Vy-perturbed-per-new-7}
 G^{\mathcal{A}}\BYDEF \min _{\gamma \in \mathcal{D}\cap \mathcal{A} }\int_{U \times Y\times Z }G(u,
y,z)\gamma(du,dy,dz),
\end{equation}
the set $\mathcal{A}$ being defined by the right-hand side of (\ref{e:SP-M}) taken with $C=0$ (note that the dependence on  $z_0$ disappears here as well).
Also, similarly to Proposition \ref{Prop-SP-2}, it can be established that the problem (\ref{Vy-perturbed-per-new-7}) is equivalent
to the IDLP problem
\begin{equation}\label{Vy-perturbed-per-new-8}
\min_{p\in \tilde{\mathcal{D}}}\int_{F}\tilde G(\mu,z)p(d\mu , dz)\BYDEF \tilde{G}^*,
\end{equation}
where $\tilde{\mathcal{D}} $ is defined by the right-hand side of (\ref{D-SP-new}) taken with $C=0$.
The equivalence between these two  problems  includes, in particular, the equality
of the optimal values (see Proposition \ref{Prop-SP-2}),
$\
 G^{\mathcal{A}}= \tilde{G}^*,
$
  implying (by (\ref{Vy-perturbed-per-new-2}) and (\ref{Vy-perturbed-per-new-6})) that
  \vspace{-.1in}
\begin{equation}\label{Vy-perturbed-per-new-9}
  \liminf_{\epsilon\rightarrow 0}V_{per}^*(\epsilon )\geq  \tilde{G}^*.
\end{equation}
Note that the problem (\ref{Vy-perturbed-per-new-8}) is the IDLP problem related to the  periodic optimization problem
\begin{equation}\label{Vy-perturbed-per-new-10}
\inf_{T,(\mu(\cdot),z(\cdot))}\frac{1}{T}\int_0^T \tilde G(\mu(t),z(t))dt\BYDEF \tilde{V}_{per}^*,
\end{equation}
where $inf $ is sought over the length of the time interval $T$ and over the admissible pairs of the averaged system (\ref{e:intro-0-4}) that satisfy the periodicity condition $z(0)=z(T) $. In particular, under certain conditions,
\begin{equation}\label{Vy-perturbed-per-new-10-1-1}
\tilde{V}_{per}^* = \tilde{G}^*.
\end{equation}
The latter, under the assumption that the averaged problem (\ref{Vy-perturbed-per-new-10})
approximates the SP problem (\ref{Vy-perturbed-per-new-1}) in the sense that $\ \lim_{\epsilon\rightarrow 0}V_{per}^*(\epsilon)= \tilde{V}_{per}^* \ $
(sufficient conditions for this  can be found in \cite{Gai-Ng}) leads to the equality
\begin{equation}\label{Vy-perturbed-per-new-11}
  \liminf_{\epsilon\rightarrow 0}V_{per}^*(\epsilon )= \tilde{G}^*.
\end{equation}

{\bf III. Average control generating (ACG) families.}

\section{ACG families; averaged and associated dual problems; necessary  optimality condition}\label{Sec-ACG-nec-opt}
The validity of the representation (\ref{tilde-W-1}) for the set $\mathcal{D}_{di}^{\mathcal{A}}(z_0)$
 motivates the definition of {\it average control generating family} given below.
For any $z\in Z$, let  $(u_z(\cdot),y_z(\cdot))$ be an admissible pair
of  the associated system (\ref{e:intro-0-3})
and  $\mu(du,dy|z)$ be the
occupational measure generated by this pair on $[0,\infty) $ (see (\ref{e:oms-0-1-infy})),
 with the integral
$\int_{U\times Y}q(u,y,z)\mu(du,dy|z) $ being a measurable function of $z$ and
\begin{equation}\label{e-opt-OM-1-0}
 | S^{-1}\int_0^S q(u_{z}(\tau),y_{z}(\tau),z)d\tau - \int_{U\times Y}q(u,y,z)\mu(du,dy|z))|\leq \phi_q(S)\ \ \forall z\in Z, \ \  \ \ \ \ \ \ \
 \lim_{S\rightarrow\infty}\phi_q(S)
\end{equation}
 for any
continuous $q(u,y,z) $.
Note that the estimate (\ref{e-opt-OM-1-0}) is valid if $(u_z(\cdot),y_z(\cdot)$ is $T_z$-periodic,
with $T_z $ being uniformly bounded on $ Z$.

\begin{Definition}\label{Def-ACG}
The family $(u_z(\cdot),y_z(\cdot)$ will be called {\it average control generating} (ACG) if the system
\begin{equation}\label{e-opt-OM-1}
z'(t)=\tilde g_{\mu}(z(t)), \ \ \  \  z(0) = z_0 ,
\end{equation}
where
\begin{equation}\label{e-opt-OM-1-101}
\tilde g_{\mu}(z)\BYDEF \tilde g(\mu(\cdot |z),z)= \int_{U\times Y}g(u,y,z)\mu(du,dy |z),
\end{equation}
has a unique solution  $z(t)\in  Z\ \forall t\in [0,\infty)$.
 \end{Definition}

Note that, according to this definition, if $(u_z(\cdot),y_z(\cdot))$ is an ACG family,
with  $\mu(du,dy|z) $ being the family of occupational measures
 generated by this family, and if $ z(\cdot)$ is the corresponding solution of (\ref{e-opt-OM-1}),
 then the  pair $(\mu(\cdot), z(\cdot)) $ (with $\mu(t)\BYDEF \mu(du,dy|z(t)) $) is an admissible pair of the averaged
 system. For convenience, this admissible pair will also be referred to as one generated by the ACG family.

\begin{Proposition}\label{Prop-clarification-1}
Let  $(u_z(\cdot),y_z(\cdot))$ be an ACG family and let $\mu(du,dy|z) $ and $(\mu(\cdot), z(\cdot)) $ be, respectively, the family of occupational measures
 and the admissible pair of the averaged system generated by this family.  Let $\nu_{di}(dz) $ be the discounted occupational measure
generated by $ z(\cdot)$,
\begin{equation}\label{e-opt-OM-1-extra-101}
\nu_{di}(Q)= C \int_0^{\infty} e^{-Ct} 1_Q(z(t))dt \ \ \ \ \forall \ \ {\rm Borel} \ \ Q\subset Z,
\end{equation}
and let
\begin{equation}\label{e-opt-OM-2}
\gamma(du,dy,dz)\BYDEF \mu(du,dy|z)\nu_{di}(dz).
\end{equation}
Then
\begin{equation}\label{e-opt-OM-2-1}
\gamma(du,dy,dz) =\Phi (p_{di}^{\mu(\cdot), z(\cdot)}).
\end{equation}
where $\ p_{di}^{\mu(\cdot), z(\cdot)} $ is the discounted occupational measure generated by $(\mu(\cdot), z(\cdot)) $.
\end{Proposition}
\begin{proof}
For an arbitrary continuous function $q(u,y,z) $  and  $\tilde q(\mu,z) $  defined as in (\ref{e:intro-0-3-9}), one can write down
$$
\int_{F}\tilde q(\mu,z)p_{di}^{\mu(\cdot), z(\cdot)}(d\mu , dz)= C \int_0^{\infty}e^{-Ct}\tilde q (\mu(t),z(t))dt
= C \int_0^{\infty}e^{-Ct} \left(\int_{U\times Y}q(u,y,z)\mu(du,dy |z(t)) \right) dt
$$
\vspace{-.2in}
$$
= \int_Z \left(\int_{U\times Y}q(u,y,z)\mu(du,dy |z) \right) \nu_{di}(dz) = \int_{U\times Y\times Z} q(u,y,z)\gamma(du,dy,dz).
$$
By the definition of $\Phi(\cdot) $ (see (\ref{e:h&th-1})), the latter implies (\ref{e-opt-OM-2-1}).
\end{proof}

 \begin{Definition}\label{Def-ACG-opt}
 An ACG family  $(u_z(\cdot),y_z(\cdot))$   will be called  optimal  if the
admissible pair $(\mu(\cdot), z(\cdot)) $ generated by this family is optimal in the averaged problem (\ref{Vy-ave-opt}). That is,
\begin{equation}\label{e-opt-OM-2-101}
\int_0 ^ { + \infty }e^{- C t}\tilde{G}(\mu(t),z(t))dt
 = \tilde{V}^*_{di}(z_0).
 \end{equation}
An ACG family  $(u_z(\cdot),y_z(\cdot))$  will be called   {\it $\alpha $-near optimal} ($\alpha >0$)
  if
 \begin{equation}\label{e-opt-OM-2-101-1}
 \int_0 ^ { + \infty }e^{- C t}\tilde{G}(\mu(t),z(t))dt
 \leq \tilde{V}^*_{di}(z_0) + \alpha.
 \end{equation}
 \end{Definition}

 Note that, provided that the equality (\ref{C-result-di-ave}) is valid, an ACG family $(u_z(\cdot),y_z(\cdot))$ will be optimal (near optimal) if and only if
 the discounted occupational measure $\ p_{di}^{\mu(\cdot), z(\cdot)} $ generated by $(\mu(\cdot), z(\cdot)) $ is an optimal (near optimal) solution
of the averaged IDLP problem (\ref{e-ave-LP-opt-di}). Also, from (\ref{e-opt-OM-2-1}) it follows that
\begin{equation}\label{e-opt-OM-2-1000-1}
\int_{U\times Y\times Z} G(u,y,z)\gamma(du,dy,dz)= C\tilde{V}^*_{di}(z_0)
\end{equation}
if the ACG family is optimal  and that
\begin{equation}\label{e-opt-OM-2-1000-2}
\int_{U\times Y\times Z} G(u,y,z)\gamma(du,dy,dz)\leq C(\tilde{V}^*_{di}(z_0) + \alpha) .
\end{equation}
if the ACG family is $\alpha $-near optimal. Thus,
under the assumption that the equality (\ref{C-result-di-ave}) is valid, an ACG family
$(u_z(\cdot),y_z(\cdot))$
will be optimal (near optimal)
if and only if   $\gamma= \Phi(p_{di}^{\mu(\cdot), z(\cdot)})$  is
an optimal (near optimal) solution of the reduced augmented problem (\ref{SP-IDLP-0}).

Let $\tilde H(p,z)$ be the
 Hamiltonian of the averaged
system
\begin{equation}\label{e:H-tilde}
\tilde H(p,z)\BYDEF \min_{\mu\in W(z)} \{\tilde G(\mu, z)+p^T \tilde g (\mu , z)\},
\end{equation}
where $\tilde g(\mu, z) $ and $\tilde G(\mu , z) $ are defined by (\ref{e:g-tilde}) and (\ref{e:G-tilde}).
Consider the problem
\begin{equation}\label{e:DUAL-AVE}
\sup_{\zeta(\cdot)\in C^1} \{\theta : \theta\leq \tilde H(\nabla \zeta(z), z) + C(\zeta(z_0) - \zeta(z))\ \forall
z\in Z \}= \tilde{G}^*_{di}(z_0)  ,
\end{equation}
where $sup $ is sought over all continuously differentiable functions $\zeta(\cdot):\R^n\rightarrow \R^1 $. Note that
 the optimal value of the  problem (\ref{e:DUAL-AVE}) is equal to the optimal value of the averaged IDLP problem (\ref{e-ave-LP-opt-di}).
 The former is in fact dual with respect to the later, the equality of the optimal values
 being one of the duality relationships between the two (see Theorem 3.1 in \cite{GQ}).
 For brevity, (\ref{e:DUAL-AVE}) will be referred to as just {\it averaged dual problem}.
 Note that the averaged dual problem can be equivalently rewritten in the form
\begin{equation}\label{e:DUAL-AVE-0}
\sup_{\zeta(\cdot)\in C^1(\R^n)} \{\theta : \theta\leq  \tilde G(\mu , z)+ \nabla \zeta(z)^T \tilde g (\mu , z) + C (\zeta(z_0) - \zeta(z))\ \ \forall
(\mu ,z) \in F \}=\tilde{G}^*_{di}(z_0) , \ \ \ \ \
\end{equation}
where  $F$ is the graph of $W(\cdot) $ (see (\ref{e:graph-w})).
A function $\zeta^*(\cdot)\in C^1 $ will be called a solution of the averaged dual problem  if
\begin{equation}\label{e:DUAL-AVE-sol-1}
\tilde{G}^*_{di}(z_0) \leq \tilde H(\nabla \zeta^* (z), z) + C(\zeta^*(z_0) - \zeta^*(z))\ \ \forall
z\in Z \ ,
\end{equation}
or, equivalently, if
\begin{equation}\label{e:DUAL-AVE-sol-1-1001}
\tilde{G}^*_{di}(z_0) \leq \tilde G(\mu, z)+\nabla \zeta^*(z)^T \tilde g (\mu , z) + C(\zeta^*(z_0) - \zeta^*(z))\ \ \forall
(\mu ,z) \in F \ .
\end{equation}
Note that, if $\zeta^*(\cdot)\in C^1 $ satisfies (\ref{e:DUAL-AVE-sol-1}), then $\zeta^*(\cdot)+ const $ satisfies
(\ref{e:DUAL-AVE-sol-1}) as well.

Assume that a solution of the averaged dual problem  (that is, a functions $\zeta^*(\cdot) $ satisfying
(\ref{e:DUAL-AVE-sol-1})) exists and consider the problem in the right hand side of (\ref{e:H-tilde}) with $p=\nabla \zeta^* (z)$, rewriting it in the form
\begin{equation}\label{e:H-tilde-10-1}
 \min_{\mu\in W(z)} \{\int_{U\times Y}[G(u,y,z)+\nabla \zeta^* (z)^T  g (u,y , z)]\mu(du,dz)\} = \tilde H (\nabla \zeta^* (z),z).
\end{equation}
The latter is an IDLP problem, with the dual of it having
 the form
\begin{equation}\label{e:dec-fast-4}
\ \ \ \ \ \ \sup_{\eta(\cdot)\in C^1(\R^n)} \{\theta : \theta \leq  G(u,y,z)+\nabla \zeta^* (z)^T  g (u,y , z) + \nabla \eta (y)^T  f(u,y,z)
\ \ \forall
(u,y) \in U\times Y \}
\end{equation}
\vspace{-.3in}
$$
=\tilde H (\nabla \zeta^* (z),z),
$$
where $sup $ is sought over all continuously differentiable functions $\eta(\cdot): \R^m\rightarrow \R^1 $.
 The  optimal values of the problems (\ref{e:H-tilde-10-1}) and (\ref{e:dec-fast-4}) are equal,
  this being one of the duality relationships between these two problems
 (see Theorem 4.1 in \cite{FinGaiLeb}). The problem (\ref{e:dec-fast-4}) will be referred to as {\it associated dual problem}.
 A function $\eta^*_z(\cdot)\in C^1(\R^m) $ will be called a solution of the  problem (\ref{e:dec-fast-4}) if
\begin{equation}\label{e:H-tilde-10-3}
   \tilde H (\nabla \zeta^* (z),z)\leq G(u,y,z)+\nabla \zeta^* (z)^T  g (u,y , z) + \nabla \eta^*_z (y)^T  f(u,y,z) \ \ \forall (u,y)\in U\times Y.
\end{equation}
The following result gives a necessary  condition for an ACG family $(u_z(\cdot),y_z(\cdot))$ to be optimal provided that the latter is periodic, that is,
\begin{equation}\label{e:H-tilde-10-3-1001}
(u_z(\tau + T_z),y_z(\tau + T_z)) = (u_z(\tau ),y_z(\tau )) \ \ \ \ \forall \ \tau \geq 0
\end{equation}
for some $\ T_z > 0 $ (in fact, for the result to be valid, the periodicity is required only for
$z=z(t)$, where $z(t) $ is  the  solution of (\ref{e-opt-OM-1})).
\begin{Proposition}\label{Prop-necessary-opt-cond}
Assume that the equality (\ref{C-result-di-ave}) is valid. Assume also that a solution $\zeta^* (z) $ of the averaged dual problem exists and  a solution $\eta^*_z (y) $ of the associated dual problem
exists for any $z\in Z$. Then, for an ACG family $(u_z(\cdot),y_z(\cdot))$ satisfying (\ref{e:H-tilde-10-3-1001})  to be optimal, it is necessary
that
$$
u_{z(t)}(\tau)\in {\rm Argmin}_{u\in U}\{  G(u,y_{z(t)}(\tau),z(t))+\nabla \zeta^* (z(t))^T  g (u,y_{z(t)}(\tau) , z(t))
$$
\vspace{-.2in}
\begin{equation}\label{e:H-tilde-10-3-1002}
+ \nabla \eta^*_{z(t)} (y_{z(t)}(\tau))^T  f(u,y_{z(t)}(\tau) , z(t))  \}
\end{equation}
 for almost all $t\in [0,\infty)$ and for almost all $ \tau \in [0,T_{z(t)}]$, where $z(t) $ is  the  solution of (\ref{e-opt-OM-1}).
 \end{Proposition}

 \begin{proof}
Let $(u_z(\cdot),y_z(\cdot))$ be an optimal ACG family. By definition, this means that the admissible pair $(\mu(\cdot), z(\cdot)) $ generated by this family is optimal in the
averaged problem (\ref{Vy-ave-opt}). Due to the assumed validity of (\ref{C-result-di-ave}), it follows (see Proposition 2.1 in \cite{GRT})
that
\begin{equation}\label{e:opt-cond-AVE-1}
\mu(t)\in {\rm Argmin}_{\mu\in W(z(t))} \{\tilde G(\mu , z(t))+\nabla \zeta^*(z(t))^T \tilde g (\mu , z(t))\ \}
\end{equation}
for almost all $t\in [0,\infty)$. That is,
\begin{equation}\label{e:opt-cond-AVE-1-proof-1}
\tilde G(\mu(t) , z(t))+\nabla \zeta^*(z(t))^T \tilde g (\mu(t) , z(t)) \ = \min_{\mu\in W(z(t))} \{\tilde G(\mu , z(t))+\nabla \zeta^*(z(t))^T \tilde g (\mu , z(t))\ \}
\end{equation}
 for almost all $t\in [0,\infty)$. By (\ref{e:H-tilde-10-3-1001}), the latter can be rewritten as
 $$
 \frac{1}{T_{z(t)}}\int_0^{T_{z(t)}}[G(u_{z(t)}(\tau),y_{z(t)}(\tau),z(t)) + \nabla \zeta^*(z(t))^T g(u_{z(t)}(\tau),y_{z(t)}(\tau),z(t))]d\tau
 $$
 \vspace{-.2in}
 \begin{equation}\label{e:opt-cond-AVE-1-proof-2}
 = \min_{\mu\in W(z(t))} \{\tilde G(\mu , z(t))+\nabla \zeta^*(z(t))^T \tilde g (\mu , z(t))\ \}.
 \end{equation}
 Since, for any admissible $T$-periodic pair $(u(\cdot),y(\cdot)) $ of the associated system (\ref{e:intro-0-3}) considered with $z=z(t)$,
 $$
\frac{1}{T} \int_0^{T}[G(u(\tau),y(\tau),z(t)) + \nabla \zeta^*(z(t))^T g(u(\tau),y(\tau),z(t))]d\tau
 $$
 \vspace{-.2in}
$$
\geq  \min_{\mu\in W(z(t))} \{\tilde G(\mu , z(t))+\nabla \zeta^*(z(t))^T \tilde g (\mu , z(t))\ \},
 $$
 (see, e.g., Corollary 3 in \cite{GR}) from (\ref{e:opt-cond-AVE-1-proof-2}) it follows that $(u_{z(t)}(\cdot),y_{z(t)}(\cdot)) $ is an optimal
 solution of the periodic optimization problem
 \begin{equation}\label{e:dec-fast-6}
\min_{T, (u(\cdot), y(\cdot))}\{\frac{1}{T}\int_0^T [G(u(\tau), y(\tau),z(t))+ \nabla \zeta^* (z(t))^T g(u(\tau), y(\tau),z(t))] d\tau  \} ,
\end{equation}
where  $min$ is taken over the length of the time interval $[0,T]$ and over the admissible
 pairs $(u(\cdot), y(\cdot))$ of the associated system (\ref{e:intro-0-3}) (considered with $z=z(t)$)
 that satisfy the periodicity conditions $y(0)=y(T) $ (with
 the optimal value in (\ref{e:dec-fast-6})
being equal to the right hand side in (\ref{e:opt-cond-AVE-1-proof-2})).
By Corollary 4.5 in \cite{FinGaiLeb}, this implies
 the statement of the proposition.
 \end{proof}

REMARK III.1.  It can be readily  verified that the concept of a solution of the averaged dual problem (see (\ref{e:DUAL-AVE-sol-1}))
is  equivalent to that of  a  smooth viscosity subsolution of the Hamilton-Jacobi-Bellman ({\bf HJB}) equation related to
 the averaged optimal control
problem (\ref{Vy-ave-opt}) (provided that  (\ref{C-result-di-ave}) is valid). It can be also understood
 that the concept of a solution of the associated dual problem (see (\ref{e:H-tilde-10-3})) is essentially
equivalent
to that of  a  smooth viscosity subsolution of the HJB equation related to the
periodic optimization problem (\ref{e:dec-fast-6}).
Note that the convergence of the optimal value function of a SP optimal control
problem  to the viscosity solution (not necessarily smooth) of the corresponding averaged HJB equation have been  studied
in \cite{Alv}, \cite{Alv-1}, \cite{AG} and \cite{G-DG} (see also references therein). In the consideration above,  we are
using solutions of the  averaged and associated dual problems (which can be interpreted as the inequality forms of the corresponding HJB equations)
to state a necessary  condition for an  ACG family to be optimal.
 The price for a possibility of doing it is, however, the
 assumption that
solutions of these problems, that is $C^1$ functions satisfying  (\ref{e:DUAL-AVE-sol-1}) and (\ref{e:H-tilde-10-3}), exist.
This is a   restrictive assumption, and we are not going to use it in the sequel. Instead, we will be considering simplified (\lq\lq approximating") versions
of the averaged and associate dual problems,  solutions of which
exist under natural controllability type conditions. We will  use those solutions instead of $\zeta^* (z) $
and $\eta^*_z (y) $  in (\ref{e:H-tilde-10-3-1002}) for the construction of near optimal ACG families, which, in turn, will be used  for the construction of asymptotically near optimal controls of  the SP problem. Note, in conclusion, that we also will not be assuming the periodicity of optimal or near optimal ACG families although in the numerical examples that we are considering in Sections \ref{Sec-Two-examples} and \ref{Sec-construction-SP-examples},  the constructed near optimal ACG families  appear to be periodic (since the system describing the fast dynamics in the examples is two-dimensional, it is consistent with the recent result of \cite{Art-Bright} establishing the sufficiency  of periodic regimes in dealing with log run average optimal control problems with two-dimensional dynamics; see also earlier developments in \cite{Col1} and \cite{Col2}).

\section{Approximating averaged IDLP problem and approximating averaged/associated  dual problems}\label{N-approx-dual-opt}
Let $\psi_i(\cdot)\in C^1(\R^n) \ , \ \ i=1,2,...,$ be  a sequence of functions such that any $\zeta(\cdot)\in C^1(\R^n)$ and its gradient are simultaneously approximated by a linear
combination of $\psi_i(\cdot)$ and their gradients. Also, let $\phi_i(\cdot)\in C^1(\R^m) \ , \ \ i=1,2,...,$ be  a sequence of functions such that any $\eta(\cdot)\in C^1(\R^m)$ and its gradient are simultaneously approximated by a linear
combination of $\phi_i(\cdot)$ and their gradients. Examples of such sequences are monomials
 $z_1^{i_1}...z_n^{i_n}$, $i_1,...,i_n =0,1,...$ and, respectively, $y_1^{i_1}...y_m^{i_m}$, $i_1,...,i_m =0,1,...$, with $z_k$, and
 $y_l $ standing for the components of $z$ and $y$ (see, e.g., \cite{Llavona}).

Let us introduce the following notations:
\begin{equation}\label{e:2.4-M}
W_M(z) \BYDEF  \{\mu \in \mathcal{P}(U \times Y) \ : \
\int_{U\times Y} \nabla\phi_i(y)^{T} f(u,y,z)
\mu(du,dy) =  0, \ \ \ i=1,...,M \} ,
\end{equation}
 \begin{equation}\label{e:graph-w-M}
F_M\BYDEF \{(\mu , z) \ : \ \mu\in W_M(z), \ \ z\in Z \} \ \subset \mathcal{P}(U\times Y)\times Z \ ,
 \end{equation}
and
 \begin{equation}\label{D-SP-new-MN}
\tilde{\mathcal{D}}_{di}^{N,M}(z_0)\BYDEF \{ p \in {\cal P} (F_M): \;
 \int_{F_M }(\nabla\psi_i(z)^T \tilde{g}(\mu,z) + C (\psi_i (z_0) - \psi_i (z) )) p(d\mu,dz) = 0, \ \ \ i=1,...,N\},
  \end{equation}
(compare with (\ref{e:2.4}), (\ref{e:graph-w}) and (\ref{D-SP-new}), respectively) and let us consider the following semi-infinite  LP problem
(compare with (\ref{e-ave-LP-opt-di}))
\begin{equation}\label{e-ave-LP-opt-di-MN}
\min_{p\in \tilde{\mathcal{D}}_{di}^{N,M}(z_0)}\int_{F_M}\tilde G(\mu,z)p(d\mu , dz)\BYDEF \tilde{G}^{N,M}_{di}(z_0).
\end{equation}
This problem will be referred to as {\it $(N,M)$-approximating averaged problem}.

It is obvious that
\begin{equation}\label{e:graph-w-M-103-0}
W_1(z)\supset W_2(z) \supset ...\supset W_M(z)\supset ...\supset W(z) \ \ \ \Rightarrow \ \ \ \ F_1\supset F_2 \supset ... \supset F_M\supset ... \supset F.
\end{equation}
Defining the set $\tilde{\mathcal{D}}_{di}^{N}(z_0)$ by the equation
 \begin{equation}\label{D-SP-new-N}
\tilde{\mathcal{D}}_{di}^{N}(z_0)\BYDEF \{ p \in {\cal P} (F): \;
 \int_{F }(\nabla\psi_i(z)^T \tilde{g}(\mu,z) + C (\psi_i (z_0) - \psi_i (z) )) p(d\mu,dz) = 0, \ \ \ i=1,...,N\},
  \end{equation}
one can also see that
\begin{equation}\label{e:graph-w-M-103-1}
\tilde{\mathcal{D}}_{di}^{N,M}(z_0)\supset
\tilde{\mathcal{D}}_{di}^{N}(z_0)\supset \tilde{\mathcal{D}}_{di}(z_0)\ \ \ \ \ \ \forall \ N, M =1,2,...
\end{equation}
(with $\mathcal{D}_{di}^{N,M}(z_0)$, $\tilde{\mathcal{D}}_{di}^{N}(z_0)$ and $\tilde{\mathcal{D}}_{di}(z_0)$ being considered as  subsets of
 $\mathcal{P}(\mathcal{P}(U \times Y)\times Z) $), the latter implying, in particular, that
 \begin{equation}\label{e:graph-w-M-103-101}
\tilde{G}^{N,M}_{di}(z_0)\leq \tilde{G}^*_{di}(z_0) \ \ \ \ \forall \ N, M =1,2,...\ .
\end{equation}

It can be readily verified that (see, e.g.,  the proof of Proposition 7 in \cite{GR}) that
\begin{equation}\label{e:graph-w-M-100}
\lim_{M\rightarrow\infty}W_M(z)= W(z),  \ \ \ \ \ \ \lim_{M\rightarrow\infty}F_M= F,
 \end{equation}
where, in the first case, the convergence is in the Housdorff metric generated by the weak convergence in $\mathcal{P}(U \times Y)$ and,
in the second, it is in the Housdorff metric generated by the weak$^*$ convergence in $\mathcal{P}(U \times Y)$  and the convergence in $Z$.

\begin{Proposition}\label{Prop-LM-convergence}
The following relationships are valid:
\begin{equation}\label{e:graph-w-M-101}
\lim_{M\rightarrow\infty}\tilde{\mathcal{D}}_{di}^{N,M}(z_0)= \tilde{\mathcal{D}}_{di}^{N}(z_0),
 \end{equation}
 \vspace{-.2in}
 \begin{equation}\label{e:graph-w-M-101-101}
 \lim_{N\rightarrow\infty}\tilde{\mathcal{D}}_{di}^{N}(z_0) = \tilde{\mathcal{D}}_{di}(z_0),
 \end{equation}
where the convergence in both cases is in Housdorff metric generated by the  weak$^*$ convergence  in $\mathcal{P}(\mathcal{P}(U \times Y)\times Z)$.
Also,
\begin{equation}\label{e:graph-w-M-102}
\lim_{N\rightarrow\infty} \lim_{M\rightarrow\infty}\tilde{G}^{N,M}_{di}(z_0)= \tilde{G}^{*}_{di}(z_0).
 \end{equation}
If the optimal solution $p^* $ of the averaged IDLP problem (\ref{e-ave-LP-opt-di}) is unique, then, for an
 an arbitrary optimal solution $p^{N,M} $ of the $(N,M) $-approximating problem (\ref{e-ave-LP-opt-di-MN}),
 \begin{equation}\label{e:graph-w-M-102-101}
\lim_{N\rightarrow\infty }\limsup_{M\rightarrow\infty}\rho(p^{N,M},p^*) = 0.
 \end{equation}
 \end{Proposition}
\begin{proof}
The proofs of (\ref{e:graph-w-M-101}) and (\ref{e:graph-w-M-101-101}) follow a standard argument and are omitted.
From  (\ref{e:graph-w-M-101})  it follows that
 \begin{equation}\label{e:graph-w-M-102-101-1}
 \lim_{M\rightarrow\infty}  \tilde{G}^{N,M}_{di}(z_0) =
 \min_{p\in \tilde{\mathcal{D}}_{di}^{N}(z_0)}\int_{F}\tilde G(\mu,z)p(d\mu , dz)\BYDEF \tilde{G}^{N}_{di}(z_0),
 \end{equation}
and from (\ref{e:graph-w-M-101-101}) it follows that
\begin{equation}\label{e:graph-w-M-102-101-2}
 \lim_{N\rightarrow\infty}\tilde{G}^{N}_{di}(z_0) =\tilde{G}^{*}_{di}(z_0).
 \end{equation}
The  above two relationships imply (\ref{e:graph-w-M-102}).
If the optimal solution $p^* $ of the averaged IDLP problem (\ref{e-ave-LP-opt-di}) is unique, then, by (\ref{e:graph-w-M-102-101-2}),
for any solution $p^N $ of the problem in the right-hand side of (\ref{e:graph-w-M-102-101-1}) there
exists the limit
\begin{equation}\label{e:graph-w-M-102-101-3}
 \lim_{N\rightarrow\infty}p^N =p^*.
 \end{equation}
Also,  if for an arbitrary optimal solution $p^{N,M} $ of the $(N,M) $ approximating problem
(\ref{e-ave-LP-opt-di-MN}) and for some $M'\rightarrow\infty $, there exists  $ \lim_{M'\rightarrow\infty}p^{N,M'} $,
then this limit is an optimal solution of the problem in the right-hand side of (\ref{e:graph-w-M-102-101-1}). This proves
(\ref{e:graph-w-M-102-101}).
\end{proof}

Define the finite dimensional space $\mathcal{Q}_N\subset C^1(\R^n) $ by the equation
\begin{equation}\label{e:HJB-8}
\mathcal{Q}_N\BYDEF \{\zeta(\cdot)\in C^1(\R^n): \zeta(z)= \sum_{i=1}^N \lambda_i \psi_i(z), \ \ \lambda=
(\lambda_i)\in \mathbb{R}^N\}
\end{equation}
 and consider the following problem
\begin{equation}\label{e:DUAL-AVE-0-approx-MN}
\sup_{\zeta(\cdot)\in \mathcal{Q}_N} \{\theta : \theta\leq  \tilde G(\mu , z)+ \nabla \zeta(z)^T \tilde g (\mu , z) + C (\zeta(z_0) - \zeta(z))\ \ \forall
(\mu ,z) \in F_M \}=\tilde{G}^{N,M}_{di}(z_0)
\end{equation}
This problem is dual with respect to the problem (\ref{e-ave-LP-opt-di-MN}),
the equality of the optimal values of these two problems being a part of the duality relationships. Note that the problem (\ref{e:DUAL-AVE-0-approx-MN})
looks similar to the averaged dual problem (\ref{e:DUAL-AVE-0}). However, in contrast to the latter, the sup is sought over the finite dimensional subspace $\ \mathcal{Q}_N$ of $\ C^1(\R^n)$ and $F_M $ is used instead of $F$.
 The problem (\ref{e:DUAL-AVE-0-approx-MN}) will be referred to as {\it $(N,M) $-approximating averaged dual problem}.
A function $\zeta^{N,M}(\cdot)\in  \mathcal{Q}_N$,
\begin{equation}\label{e:DUAL-AVE-0-approx-2}
\zeta^{N,M}(z) = \sum_{i=1}^N \lambda_i^{N,M} \psi_i(z),
\end{equation}
will be called a solution of the $(N,M)$-approximating  averaged dual problem if
\begin{equation}\label{e:DUAL-AVE-0-approx-2-1}
\tilde{G}^{N,M}_{di}(z_0)\leq  \tilde G(\mu , z)+ \nabla \zeta^{N,M}(z)^T \tilde g (\mu , z) + C (\zeta^{N,M}(z_0) - \zeta^{N,M}(z))\ \ \forall
(\mu ,z) \in F_M .
\end{equation}

Define the finite dimensional space $\mathcal{V}_M\subset C^1(\R^n) $ by the equation
\begin{equation}\label{e:HJB-8-Associated}
\mathcal{V}_M\BYDEF \{\eta(\cdot)\in C^1(\R^m): \eta(y)= \sum_{i=1}^M \omega_i \phi_i(y), \ \ \omega=
(\omega_i)\in \mathbb{R}^M\}
\end{equation}
 and, assuming that a solution $\zeta^{N,M}(z)$ of the $(N,M) $-approximating averaged dual problem exists,  consider the following problem
\begin{equation}\label{e:dec-fast-4-Associated}
\sup_{\eta(\cdot)\in \mathcal{V}_M} \{\theta : \theta \leq  G(u,y,z)+\nabla \zeta^{N,M} (z)^T  g (u,y , z) + \nabla \eta (y)^T  f(u,y,z) \ \ \forall
(u,y) \in U\times Y \}\BYDEF \sigma_{N,M}^*(z).
\end{equation}
While the problem (\ref{e:dec-fast-4-Associated}) looks similar to the associated dual problem
 (\ref{e:dec-fast-4}), it
differs from the latter, firstly, by that
$sup $ is sought over the finite dimensional subspace  $\mathcal{V}_M $ of   $C^1(\R^n) $ and, secondly, by that
a solution
$\zeta^{N,M} (z) $ of  (\ref{e:DUAL-AVE-0-approx-MN}) is used instead of a solution $\zeta^* (z) $ of (\ref{e:DUAL-AVE}) (the later
may not exist).  The problem (\ref{e:dec-fast-4-Associated}) will be referred to as
{\it $(N,M)$-approximating  associated dual problem}.
It can be shown that
it is, indeed, dual with respect to the semi-infinite LP problem
\begin{equation}\label{e:dec-fast-4-Associated-1}
\min_{\mu\in W_M(z)}\{\int_{U\times Y}[G(u,y,z)+\nabla \zeta^{N,M} (z)^T  g (u,y , z)]\mu(du,dy)= \sigma_{N,M}^*(z),
\end{equation}
the duality relationships including the equality of the optimal values (see Theorem 5.2(ii) in \cite{FinGaiLeb}).
A function $\eta^{N,M}_z(\cdot)\in \mathcal{V}_M$,
\begin{equation}\label{e:DUAL-AVE-0-approx-1-Associate-1}
\eta^{N,M}_z(y)= \sum_{i=1}^M \omega_{z,i}^{N,M} \phi_i(y),
\end{equation}
will be called a solution of the $(N,M)$-approximating  associated dual problem if
\begin{equation}\label{e:DUAL-AVE-0-approx-1-Associate-2}
\sigma_{N,M}^*(z) \leq G(u,y,z)+\nabla \zeta^{N,M} (z)^T  g (u,y , z) + \nabla \eta_z^{N,M} (y)^T  f(u,y,z) \ \ \forall (u,y)\in U\times Y.
\end{equation}
In the next section, we show that solutions of the $(N,M)$-approximating averaged and associated dual problems exist under natural controllability conditions.

\section{Controllability conditions sufficient for the existence of solutions of approximating averaged and associated dual problems}\label{Sec-Existence-Controllability}
In what follows it is assumed that, for any $N=1,2,...,$ and $M=1,2,..., $ the gradients $\nabla \psi_i(z), \ i =1,2,... N,$ and $\nabla \phi_i(y), \ i =1,2,... M, $  are linearly independent on
any open subset of $\R^N $ and, respectively, $\R^M $. That is, if $Q $ is an open subset of $\R^N$, then  the equality
$$
\ \sum_{i=1}^N v_i \nabla \psi_i(z) = 0 \ \ \forall z\in Q
$$
is valid only if $v_i=0, \ i=1,...,N $. Similarly, if $D $ is an open subset of $\R^M$, then the equality
$$
\ \sum_{i=1}^M w_i \nabla \phi_i(y) = 0 \ \forall y\in D
$$
is valid only if $w_i=0, \ i=1,...,M $.

Let $ \mathcal{R}_{z_0}$ stand
for the set of  points that are reachable  by the state components $z(t)$ of
the admissible pairs $(\mu(\cdot),z(\cdot)) $ of the averaged system (\ref{e:intro-0-4}). That is,
\begin{equation}\label{e:HJB-9-reachability-AVE}
\mathcal{R}_{z_0}\BYDEF \{z : z = z(t),  \ \ t\in [0,\infty), \ \ \ (\mu(\cdot),z(\cdot))-{\rm admissible\ in} \ (\ref{e:intro-0-4})\ \}
\end{equation}

The existence of a solution of the approximating averaged dual problem can be guaranteed
under the following controllability type assumption about the averaged system.

\begin{Assumption}\label{Ass-ave-disc-controllability}
The closure of the set $\mathcal{R}_{z_0}$ has a nonempty interior. That is,
\begin{equation}\label{e:HJB-10}
int (cl \mathcal{R}_{z_0}) \neq\emptyset .
\end{equation}
\end{Assumption}

\begin{Proposition}\label{Prop-existence-disc}
If Assumption \ref{Ass-ave-disc-controllability} is satisfied, then a solution of the $(N,M) $-approximating averaged
dual problem exists for any $N$ and $M$.
\end{Proposition}
\begin{proof}
The proof is given at the end of this section (its idea being similar to that of the proof of Proposition 3.2 in \cite{GRT}).
\end{proof}

The existence of a solution of the approximating associated dual problem is guaranteed
by the following assumption about  controllability  properties  of the associated system.

\begin{Assumption}\label{Ass-associated-local-controllability}
There exists $\ Y^0(z)\subset Y $ such that the closure of $Y^0(z) $ has a nonempty interior
and
such that any two points in  $Y^0 (z)$ can be connected by an admissible trajectory of the associated system (that is, for
any $ y', y''\in Y^0(z)$ , there exists an admissible pair $(u(\cdot ), y(\cdot))$ of the associated system defined on some interval
$[0, S]$ such that $y(0) = y'$ and $y(S) = y''$).
\end{Assumption}

\begin{Proposition}\label{dual-existence-average}
If Assumption \ref{Ass-associated-local-controllability} is satisfied, then
 a solution  of the
$(N,M)$-approximating  associated dual problem exists.
\end{Proposition}
\begin{proof}
The proof is at the end of this section. Its idea is similar to that used in the proof of Proposition 8(ii) in \cite{GR}.
\end{proof}

The proofs of Propositions \ref{Prop-existence-disc} and \ref{dual-existence-average} are based on the following lemma

\begin{Lemma}\label{Lemma-Important-simple-duality}
Let $X$ be a compact metric space and let $\Psi_i(\cdot): X\rightarrow \R^1, \ i=0,1,...,K, $ be continuous functional on $X $. Let
\begin{equation}\label{e:Important-simple-duality-Lemma-1}
\sigma^*\BYDEF \sup_{\{\lambda_i\}}\{\theta \ : \ \theta \leq \Psi_0(x) + \sum_{i=1}^{K}\lambda_i \Psi_i(x) \ \forall x\in X\} ,
\end{equation}
where $sup$ is sought over $\lambda\BYDEF \{\lambda_i\}\in \R^K $. A solution of the problem (\ref{e:Important-simple-duality-Lemma-1}), that is
$\lambda^*\BYDEF \{\lambda_i^*\}\in \R^K $ such that
\begin{equation}\label{e:Important-simple-duality-Lemma-2}
\sigma^* \leq \Psi_0(x) + \sum_{i=1}^{K}\lambda_i^* \Psi_i(x) \ \ \forall x\in X
\end{equation}
exists if the inequality
\begin{equation}\label{e:Important-simple-duality-Lemma-3}
0 \leq  \sum_{i=1}^{K}v_i \Psi_i(x) \ \forall x\in X
\end{equation}
is valid only with  $v_i = 0 , \ i=1,...,K $.
\end{Lemma}
\begin{proof}
Assume that the inequality (\ref{e:Important-simple-duality-Lemma-3}) implies that $v_i = 0 , \ i=1,...,K $.
Note that from this assumption it immediately follows that $\sigma^*$ is bounded (since, otherwise, (\ref{e:Important-simple-duality-Lemma-1})
would imply that there exist $\{\lambda_i\}$ such that $\ \sum_{i=1}^{K}\lambda_i \Psi_i(x) > 0\  \forall x\in X  $).
For any $k=1,2,...$, let $\lambda^k = (\lambda^k_i) \in \R^K$
 be such that
\begin{equation}\label{e-new-proof-3-1-2}
\sigma^*- \frac{1}{k} \leq \Psi_0(x) + \sum_{i=1}^{K}\lambda_i^k \Psi_i(x) \ \ \forall x\in X
\end{equation}
 Show that the sequence $\lambda^k \ , k=1,2,... ,\  $ is  bounded. That is, there exists $\alpha > 0$ such that
\begin{equation}\label{e:Th4.3-proof-2}
||\lambda^k|| \leq \alpha \ , \ \ \ \ \ \ k=1,2,... \ .
\end{equation}
Assume that the sequence $\lambda^k \ , k=1,2,... ,\  $  is not bounded. Then there  exists a
subsequence $\lambda^{k'}$ such that
\begin{equation}\label{e:Th4.3-proof-3}
\lim_{k'\rightarrow\infty}||\lambda^{k'}|| = \infty \ , \ \ \ \ \ \
\lim_{k'\rightarrow\infty}\frac{\lambda^{k'} }{||\lambda^{k'}||}\BYDEF
v, \ \ \ \ \ \ ||v ||=1 \ .
\end{equation}
Dividing (\ref{e-new-proof-3-1-2}) by $||\lambda^k||$ and passing to the
limit along the subsequence $\{k'\}$, one can obtain  that
$$
0 \leq  \sum_{i=1}^{K}v_i^k \Psi_i(x) \ \ \forall x\in X,
$$
which,
by our assumption, implies that
$v=(v_i)=0$. The latter contradicts  (\ref{e:Th4.3-proof-3}). Thus, the validity of
(\ref{e:Th4.3-proof-2}) is established.
Due to (\ref{e:Th4.3-proof-2}), there exists a subsequence
$\{k'\}$ such that there exists a limit
\begin{equation}\label{e:Th4.3-proof-5}
\lim_{k'\rightarrow\infty}\lambda^{k'} \BYDEF \lambda^* \ .
\end{equation}
Passing  to the limit in
(\ref{e-new-proof-3-1-2}) along  this subsequence, one proves (\ref{e:Important-simple-duality-Lemma-2}).

\end{proof}

{\it Proof of Proposition \ref{Prop-existence-disc}}.
By Lemma \ref{Lemma-Important-simple-duality}, to prove the proposition, it is sufficient to show that,
under Assumption \ref{Ass-ave-disc-controllability}, the inequality
\begin{equation}\label{e:existence-NM-average-dual-1}
0\leq \sum_{i=1}^N v_i [\nabla\psi_i(z)^T\tilde g(\mu,z) +C (\psi_i(z_0) - \psi_i(z))] \ \ \forall (\mu, z)\in F_M
\end{equation}
can be valid only with $\ v_i=0, \ i=1,...,N $. Let us assume that (\ref{e:existence-NM-average-dual-1}) is valid
and let us rewrite it in the form
\begin{equation}\label{e:existence-NM-average-dual-1-1}
0\leq \nabla\psi(z)^T\tilde g(\mu,z) +C (\psi(z_0) - \psi(z)) \ \ \forall (\mu, z)\in F_M , \ \ \ \ \ \ {\rm where} \ \ \ \ \ \  \psi(z)\BYDEF \sum_{i=1}^N v_i \psi_i(z).
\end{equation}
It can be readily verified (using integration by part) that, for
an arbitrary admissible pair $(\mu(\cdot), z(\cdot)) $ of the averaged system,
\begin{equation}\label{e:existence-NM-average-dual-2}
\int_0^{\infty} e^{-Ct}[\nabla\psi(z(t))^T\tilde g(\mu(t),z(t)) +C (\psi(z_0) - \psi(z(t))]dt = 0.
\end{equation}
Also, by the definition of an admissible pair of the averaged system  (see Definition \ref{Def-adm-averaged}) and by (\ref{e:graph-w-M-103-0}),
$$\ (\mu(t),z(t))\in F\subset F_M \ \ \ \ {\rm a.e.} \ \ \ \  t\in [0,\infty). $$
Hence, from (\ref{e:existence-NM-average-dual-1-1}) it
follows that
$$
0\leq \nabla\psi(z(t))^T\tilde g(\mu(t),z(t)) +C (\psi(z_0) - \psi(z(t))  \ \ \ \ {\rm a.e.} \ \ \ \  t\in [0,\infty).
$$
The latter and (\ref{e:existence-NM-average-dual-2}) imply that
$$
0= \nabla\psi(z(t))^T\tilde g(\mu(t),z(t)) +C (\psi(z_0) - \psi(z(t))  \ \ \ \ {\rm a.e.} \ \ \ \  t\in [0,\infty),
$$
$$
\Rightarrow \ \ \ \ \ \frac{d}{dt}[ \psi(z(t)) -  \psi(z_0)]= C [ \psi(z(t)) - \psi(z_0)]
\ \ \ \ \ \ \Rightarrow \ \ \ \ \ \  \psi(z(t))=\psi(z_0) \ \ \forall \ t\in [0,\infty).
$$
Consequently, by the definition of $\mathcal{R}_{z_0} $ (see (\ref{e:HJB-10}))
$$
\psi(z)=\psi(z_0)  \ \ \forall z\in \mathcal{R}_{z_0}\ \ \ \ \ \ \Rightarrow \ \ \ \ \ \
\psi(z)=\psi(z_0) \ \ \ \forall z\in cl\mathcal{R}_{z_0}
$$
The latter implies that
$$
\nabla \psi(z) = \sum_{i=1}^N v_i \nabla \psi_i(z) = 0  \ \ \ \ \forall z\in int(cl\mathcal{R}_{z_0}),
$$
which, in turn, implies that $v_i=0, \ i=1,...,N $ (due to  linear independence of $\nabla \psi_i(\cdot)$).
\qquad
\endproof

{\it Proof of Proposition \ref{dual-existence-average}}.
By Lemma \ref{Lemma-Important-simple-duality}, to prove the proposition, it is sufficient to show that,
under Assumption \ref{Ass-associated-local-controllability}, the inequality
\begin{equation}\label{e:existence-NM-associated-dual-1}
0\leq \sum_{i=1}^M v_i [\nabla\phi_i(y)^Tf(u,y,z)]  \ \ \forall (u,y)\in U\times Y
\end{equation}
can be valid only with $\ v_i=0, \ i=1,...,M $ (remind that $z = {\rm constant}$). Let us assume that (\ref{e:existence-NM-associated-dual-1}) is valid
and let us rewrite it in the form
\begin{equation}\label{e:existence-NM-associated-dual-1-1}
0\leq \nabla\phi(y)^T f(u,y,z)  \ \ \forall (u,y)\in U\times Y , \ \ \ \ \ \ {\rm where} \ \ \ \ \ \  \phi(y)\BYDEF \sum_{i=1}^M v_i \phi_i(y).
\end{equation}
Let $ y', y''\in Y^0(z)$  and let an admissible pair $(u(\cdot ), y(\cdot))$ of the associated system be such that
 $y(0) = y'$ and $y(S) = y''$ for some $S>0$. Then, by (\ref{e:existence-NM-associated-dual-1-1}),
 $$
 \phi(y'') - \phi(y')= \int_0^S f(u(\tau),y(\tau),z)d\tau \geq 0 \ \ \ \ \ \Rightarrow \ \ \ \ \ \phi(y'') \geq \phi(y').
 $$
Since $ y', y''$ can be arbitrary points in $Y^0(z) $, it follows that
$$
\phi(y) = {\rm const} \ \ \forall y\in Y^0(z) \ \ \ \ \ \Rightarrow \ \ \ \ \ \phi(y) = {\rm const} \ \ \forall y\in cl Y^0(z).
$$
The latter implies that
$$
\nabla \phi(y) = \sum_{i=1}^M v_i \nabla \phi_i(y) = 0  \ \ \ \ \forall y\in int(clY^0(z)),
$$
which, in turn, implies that $v_i= 0, \ i=1,...,M $  (due to linear independence of $\nabla \phi_i(\cdot) $).
\qquad
\endproof

\section{Construction of near optimal ACG families}\label{Sec-ACG-construction}
 Let us assume that, for any $N$ and $ M$, a solution $\zeta^{N,M} (z) $ of the $(N,M)$-approximating averaged dual problem exists  and a solution
  $ \eta^{N,M}_z (y)$ of the
 $(N,M)$-approximating associated problem exists for any $z\in Z$ (as follows from Propositions \ref{Prop-existence-disc} and \ref{dual-existence-average}
 these  exist if Assumptions \ref{Ass-ave-disc-controllability}
  and \ref{Ass-associated-local-controllability} are satisfied)

 Define a control $u^{N,M}(y,z) $  as an optimal solution of the problem
\begin{equation}\label{e-NM-minimizer-0}
\min_{u\in U}\{G(u,y,z)+\nabla \zeta^{N,M} (z)^T  g (u,y , z) + \nabla \eta^{N,M}_z (y)^T  f(u,y,z)\}.
\end{equation}
That is,
\begin{equation}\label{e-NM-minimizer-1}
u^{N,M}(y,z)=argmin_{u\in U}\{G(u,y,z)+\nabla \zeta^{N,M} (z)^T  g (u,y , z) + \nabla \eta^{N,M}_z (y)^T  f(u,y,z)\}
\end{equation}
Assume that the system
\begin{equation}\label{e:opt-cond-AVE-1-fast-100-2-NM}
y'_z(\tau) = f(u^{N,M}(y_z(\tau),z),y_z(\tau), z), \ \ \ \ \ y_z(0) = y \in Y,
\end{equation}
has a unique  solution $y_z^{N,M}(\tau)\in Y$. Below, we  introduce
assumptions under which it will be established that
 $(u_z^{N,M}(\cdot),y_z^{N,M}(\cdot)) $, where  $u_z^{N,M}(\tau)  \BYDEF u_z^{N,M}(y_z^{N,M}(\tau),z) $, is
a near optimal ACG family (see Theorem \ref{Main-SP-Nemeric}).

\begin{Assumption}\label{SET-1} The following conditions are satisfied:

(i) The optimal solution $p^*$ of the  IDLP problem (\ref{e-ave-LP-opt-di}) is unique, and  the equality (\ref{C-result-di-ave}) is valid.

(ii)  The optimal solution of the averaged problem (\ref{Vy-ave-opt}) (that is, an admissible pair $(\mu^*(\cdot), z^*(\cdot))$ that delivers minimum in (\ref{Vy-ave-opt})) exists,  the optimal control $\mu^*(\cdot) $ is piecewise continuous   and, at every discontinuity point,
$\mu^*(\cdot) $ is either continuous from the left or it is continuous from the right.

(iii) For almost all $t\in [0,\infty) $, there exists
an  admissible pair $(u_t^{*}(\tau),y_t^{*}(\tau)) $ of the associated system (considered with $z=z^*(t)$) such that
 $\mu^*(t) $ is the occupational measure generated by this pair on the interval $[0,\infty) $. That is,  for any
continuous $q(u,y) $,
\begin{equation}\label{e-opt-OM-1-0-NM-star}
\lim_{S\rightarrow\infty} S^{-1}\int_0^S q(u_t^{*}(\tau),y_t^{*}(\tau))d\tau = \int_{U\times Y}q(u,y)\mu^*(t)(du,dy).
 \end{equation}
 Also, for almost all $\tau\in [0,\infty) $
 and for any $r>0$, the $\mu^*(t) $-measure
of the set
$$
\ B_r(u_t^{*}(\tau),y_t^{*}(\tau))\BYDEF \{(u,y)\ : \ ||u-u_t^{*}(\tau)|| + ||y - y_t^{*}(\tau)||< r\}
$$
 is not zero. That is,
\begin{equation}\label{e:convergence-important-7}
\mu^*(t)(B_r(u_t^{*}(\tau),y_t^{*}(\tau))) > 0\ \ \ \ \forall \ r > 0.
\end{equation}

(IV) The pair $(u_z^{N,M}(\tau),y_z^{N,M}(\tau)) $, where $y_z^{N,M}(\tau)$ is the solution of (\ref{e:opt-cond-AVE-1-fast-100-2-NM})
and $u_z^{N,M}(\tau) \\ = u_z^{N,M}(y_z^{N,M}(\tau),z) $,
 generates  the occupational measure
$\mu^{N,M}(du,dy|z)$ on the interval $[0,\infty) $, the latter being independent of the initial conditions $y_z(0)=y$ for $y$ in a neighbourhood of $y_t^{*}(\cdot) $. Also, for any
continuous $q(u,y) $, the integral
$\int_{U\times Y}q(u,y)\mu^{N,M}(du,dy|z) $ is a measurable function of $z$ and
\begin{equation}\label{e-opt-OM-1-0-NM}
 | S^{-1}\int_0^S q(u_{z}^{N,M}(\tau),y_{z}^{N,M}(\tau))d\tau - \int_{U\times Y}q(u,y)\mu^{N,M}(du,dy|z))|\leq \phi_q(S)\ \ \forall z\in Z, \ \  \ \ \ \ \ \ \
 \lim_{S\rightarrow\infty}\phi_q(S).
\end{equation}

\end{Assumption}

Note that, as can be readily verified, both (\ref{e-opt-OM-1-0-NM-star}) and (\ref{e:convergence-important-7}) are satisfied if 
$(u_t^{*}(\cdot), y_t^{*}(\cdot)) $ is periodic,
$$
(u_t^{*}(\tau), y_t^{*}(\tau)) = (u_t^{*}(\tau + T_t), y_t^{*}(\tau + T_t)),
$$ 
with the period $T_t > 0 $ being uniformly bounded on $[0,\infty) $ and if $u_t^{*}(\cdot)$ is
 piecewise continuous on $[0,T_t] $.

To state our next assumption, let us re-denote the occupational measure $\mu^{N,M}(du,dy|z)$ (introduced in Assumption \ref{SET-1}(IV) above)
as $\mu^{N,M}(z) $ (that is, $\ \mu^{N,M}(du,dy|z)=\mu^{N,M}(z) $).

\begin{Assumption}\label{SET-2} For almost all $t\in [0,\infty)$, there exists
  an open ball  $Q_t\subset \mathbb{R}^n $ centered at $z^*(t) $ such that:

(i) The occupational measure $\mu^{N,M}(z) $   is continuous on  $ Q_t$.
Namely, for any $z', z''\in  Q_t$,
\begin{equation}\label{e:HJB-1-17-1-per-cont-mu-NM-g-1}
  \rho(\mu^{N,M}(z'), \mu^{N,M}(z''))\leq \kappa(||z'-z''||),
  \end{equation}
   where $\kappa(\theta) $ is a function tending to zero when $\theta $ tends to zero ($\lim_{\theta\rightarrow 0}\kappa   (\theta)=0$). Also,
 for any $z', z''\in  Q_t$,
       \begin{equation}\label{e:HJB-1-17-1-per-Lip-g}
||\int g(u,y,z')\mu^{N,M}(z')(du,dy) -\int g(u,y,z'')\mu^{N,M}(z'')(du,dy) ||\leq L ||z'-z''||,
\end{equation}
where $L$ is a constant.

(ii) The system
     \begin{equation}\label{e-opt-OM-1-MN}
z'(t)=\tilde{g}(\mu^{N,M}(z(t)), z(t)) \ , \ \ \  \  z(0) = z_0 .
\end{equation}
has a unique solution $z^{N,M}(t)\in Z$  and, for any $t > 0 $,
 \begin{equation}\label{e:HJB-1-17-1-N}
\lim_{N\rightarrow\infty}\limsup_{M\rightarrow\infty }meas \{A_t(N,M)\}=0.
\end{equation}
 where
 \begin{equation}\label{e:intro-0-4-N-A-t}
 A_t(N,M)\BYDEF \{t'\in [0,t] \ : \ z^{N,M}(t')\notin Q_{t'}\}.
 \end{equation}
and $meas\{\cdot\} $ stands for the Lebesgue measure of the corresponding set.
\end{Assumption}

In addition  to Assumptions \ref{SET-1} and \ref{SET-2}, let us also introduce

\begin{Assumption}\label{SET-3}
For each $t\in [0,\infty)$ such that
  $Q_t\neq \emptyset $, the following conditions are satisfied:

(i) For almost all $\tau\in [0,\infty ) $, there exists an open ball $B_{t,\tau}\subset \mathbb{R}^m $ centered at $y_t^{*}(\tau) $
such  $u^{N,M}(y,z)  $ is uniquely defined (the problem (\ref{e-NM-minimizer-0}) has a unique solution) for $(y,z)\in B_{t,\tau}\times Q_t  $.

(ii) The function $u^{N,M}(y,z)  $
 satisfies Lipschitz conditions on $B_{t,\tau}\times Q_t$.  That is,
 \begin{equation}\label{e:HJB-1-17-0-per-Lip-u-NM}
||u^{N,M}(y',z')- u^{N,M}(y'',z'')||\leq L ( ||y'-y''|| + ||z'-z''||) \ \ \ \ \ \  \forall (y',z'), (y'',z'')\in B_{t,\tau}\times Q_t,
\end{equation}
where $L$ is a constant.

(iii) Let  $ y^{N,M}_{t} (\tau)\BYDEF y^{N,M}_{z^*(t)} (\tau)$ be the solution of the system  (\ref{e:opt-cond-AVE-1-fast-100-2-NM})
considered with $z=z^*(t) $ and with the initial condition
$y_z(0)=y_t^*(0) $. We assume that, for any $\tau >0 $,
\begin{equation}\label{e:HJB-1-17-1-N-z-const}
\lim_{N\rightarrow\infty}\limsup_{M\rightarrow\infty }meas \{P_{t,\tau}(N,M)\}=0,
\end{equation}
where
\begin{equation}\label{e:HJB-1-17-1-N-z-const-def}
P_{t,\tau}(N,M)\BYDEF \{\tau'\in [0,\tau] \ : \  y^{N,M}_{t}(\tau')\notin B_{t, \tau'}\}.
\end{equation}

\end{Assumption}

\begin{Theorem}\label{Main-SP-Nemeric}
Let Assumptions \ref{SET-1}, \ref{SET-2} and \ref{SET-3} be valid.
Then the family $(u^{N,M}_z(\cdot),  y^{N,M}_z(\cdot)) $ $($with $\ y^{N,M}_z(\tau)$ being the solution of (\ref{e:opt-cond-AVE-1-fast-100-2-NM})
and  $\ u_z^{N,M}(\tau) = u_z^{N,M}(y_z^{N,M}(\tau),z)\ )$ is
a $\beta(N,M) $-
near optimal ACG family, where
\begin{equation}\label{e:HJB-1-17-1-N-z-const-def-near-opt}
\lim_{N\rightarrow\infty}\limsup_{M\rightarrow\infty}\beta(N,M) = 0.
\end{equation}
\end{Theorem}

\begin{proof} The proof is given in Section \ref{Sec-Main-Main}. It is based on Lemma \ref{fast-convergence} stated at the end
of this section.    Note  that in
the process of the proof of the theorem it is established that
\begin{equation}\label{e:HJB-1-19-2-NM}
\lim_{N\rightarrow\infty}\limsup_{M\rightarrow\infty}\max_{t'\in [0,t]} ||z^{N,M}(t')- z^*(t')|| = 0 \ \ \ \forall t\in [0,\infty),
\end{equation}
where $z^{N,M}(\cdot) $ is the solution of (\ref{e-opt-OM-1-MN}).
Also, it is shown that
  \begin{equation}\label{e:HJB-1-19-1-NM}\  \lim_{N\rightarrow\infty}\limsup_{M\rightarrow\infty}\rho (\mu^{N,M}(z^{N,M}(t)),\mu^*(t)) = 0   \end{equation}
  for almost all  $t\in [0,\infty) $,
and
 \begin{equation}\label{e:HJB-19-NM}
\lim_{N\rightarrow\infty}\limsup_{M\rightarrow\infty}|\tilde V_{di}^{N,M}(z_0)- \tilde V^*_{di}(z_0)| = 0,
\end{equation}
where
\begin{equation}\label{e:HJB-16-NM}
\tilde V_{di}^{N,M}(z_0) \BYDEF \int_0^{ + \infty }e^{-Ct}\tilde G(\mu^{N,M}(z^{N,M}(t)),
z^{N,M}(t))dt .
\end{equation}
The relationship  (\ref{e:HJB-19-NM}) implies the statement of the theorem
 with
$$
\beta(N,M)\BYDEF | \tilde V^*_{di}(z_0) - V_{di}^{N,M}(z_0)|
$$
(see Definition \ref{Def-ACG-opt}).
 \end{proof}

\begin{Lemma}\label{fast-convergence}
Let the assumptions of Theorem \ref{Main-SP-Nemeric} be satisfied and
let $t\in [0,\infty) $ be such that $Q_t\neq\emptyset $. Then
\begin{equation}\label{e:HJB-1-17-1-N-z-const-fast-1}
\lim_{N\rightarrow\infty}\limsup_{M\rightarrow\infty }\max_{\tau'\in [0,\tau]}| y^{N,M}_{t} (\tau') - y_t^{*}(\tau') ||=0
\ \ \ \forall \tau \in [0,\infty)
\end{equation}
Also,
\begin{equation}\label{e:HJB-1-17-1-N-z-const-fast-2}
\lim_{N\rightarrow\infty}\limsup_{M\rightarrow\infty }||u^{N,M}(y^{N,M}_{t} (\tau), z^*(t) ) - u_t^{*}(\tau) ||=0
\end{equation}
for almost all $\tau\in [0,\infty) $.
\end{Lemma}
\begin{proof}
The proof is given in Section  \ref{Sec-Main-Main}.
\end{proof}

REMARK III.2. Results of Sections \ref{Sec-ACG-nec-opt} - \ref{Sec-ACG-construction} can be extended to the case when the periodic
optimization problem (\ref{Vy-perturbed-per-new-1}) is under consideration. In particular, one can introduce the $(N,M)$-approximating averaged
 problem
 \begin{equation}\label{e-ave-LP-opt-di-MN-ave}
\min_{p\in \tilde{\mathcal{D}}^{N,M}}\int_{F_M}\tilde G(\mu,z)p(d\mu , dz)\BYDEF \tilde{G}^{N,M},
\end{equation}
where $\ \tilde{\mathcal{D}}^{N,M}$ is defined by the right-hand side of (\ref{D-SP-new-MN}) taken with $C=0$. The optimal
value $\tilde{G}^{N,M}$  of this problem  is related to the optimal value $\tilde{G}^{*}$ of the problem (\ref{Vy-perturbed-per-new-8}) by the inequalities
 \begin{equation}\label{e:graph-w-M-103-101-ave }
\tilde{G}^{N,M}\leq \tilde{G}^* \ \ \ \ \forall \ N, M =1,2,...\ , \ \ \ \ \ \ \
\end{equation}
and, in addition,
\begin{equation}\label{e:graph-w-M-102-ave}
\lim_{N\rightarrow\infty} \lim_{M\rightarrow\infty}\tilde{G}^{N,M}= \tilde{G}^{*}
 \end{equation}
 (compare with (\ref{e:graph-w-M-103-101}) and (\ref{e:graph-w-M-102})).
One can also introduce the $(N,M)$-approximating averaged dual problem as in  (\ref{e:DUAL-AVE-0-approx-MN}) (with $C=0$) and introduce the $(N,M)$-approximating associated dual problem as in (\ref{e:dec-fast-4-Associated}). Assuming that solutions of these problems exist,
one can define a control $u^{N,M}(y,z)$ as a minimizer   (\ref{e-NM-minimizer-1}). It can be shown that, under certain assumptions (some of which are similar to those used in Theorem \ref{Main-SP-Nemeric} and some are  specific for the periodic optimization case), the control  $u^{N,M}(y,z)$ allows one to construct an ACG family that generates a near optimal solution of the averaged periodic optimization problem (\ref{Vy-perturbed-per-new-10}). While we do not give a precise result justifying such a procedure in the present paper,
we demonstrate that it can be efficient in dealing with SP periodic optimization problems by considering a numerical example (see Example 2 in Sections \ref{Sec-Two-examples}
and \ref{Sec-construction-SP-examples}).

\medskip

{\bf IV. Asymptotically near optimal controls of SP problems.}

\section{Asymptotically optimal/near optimal controls based on optimal/near optimal ACG families}\label{Sec-SP-ACG-theorem}
In this section we describe  a way how an asymptotically optimal (near optimal)
control of the SP optimal control problem (\ref{Vy-perturbed}) can be constructed
given that an asymptotically optimal (near optimal)
ACG family is known (a way of construction of the latter has been discussed in Section \ref{Sec-ACG-construction} ).

\begin{Definition}\label{Def-asympt-opt}
A control $u_{\epsilon}(\cdot) $  will be called asymptotically optimal in the SP problem  (\ref{Vy-perturbed}) if
\begin{equation}\label{e-opt-convergence-1}
\lim_{\epsilon\rightarrow 0}\int_0 ^ { + \infty }e^{- C t}G(u_{\epsilon}(t), y_{\epsilon}(t),z_{\epsilon}(t)) dt=
\lim_{\epsilon\rightarrow 0}V_{di}^*(\epsilon, y_0, z_0),
\end{equation}
where $(y_{\epsilon}(\cdot),z_{\epsilon}(\cdot)) $ is the solution of the system (\ref{e:intro-0-1})-(\ref{e:intro-0-2}) obtained with the control
$u_{\epsilon}(\cdot) $ and with the initial condition (\ref{e-initial-SP}).
A control $u_{\epsilon}(\cdot) $  will be called asymptotically $\alpha$-near optimal ($\alpha>0$)
in the SP problem  (\ref{Vy-perturbed}) if
\begin{equation}\label{e-opt-convergence-2}
\lim_{\epsilon\rightarrow 0}\int_0 ^ { + \infty }e^{- C t}G(u_{\epsilon}(t), y_{\epsilon}(t),z_{\epsilon}(t)) dt\leq
\liminf_{\epsilon\rightarrow 0}V_{di}^*(\epsilon, y_0, z_0)+ \alpha .
\end{equation}
\end{Definition}

  We will need a couple of more definitions and assumptions.

\begin{Definition}\label{Def-steering-1}
Let $\mu\in W(z)$   and $y\in Y$. We shall say that $\mu$ is {\it attainable} by  the associated system  from the initial conditions $y$
 if there exists an admissible pair $(u_{y,z}(\cdot),y_{y,z}(\cdot))$  of the associated system (see Definition \ref{Def-adm-associate})
 satisfying the initial condition $y(0)=y $
  such that the occupational measure $\mu^{u_{y,z}(\cdot), y_{y,z}(\cdot), S}$
  generated by the pair $(u_{y,z}(\cdot), y_{y,z}(\cdot))$ on the interval $[0,S]$ converges to $\mu$ as $S\rightarrow\infty$. That is,
\begin{equation}\label{e-opt-OM-2-103}
\lim_{S\rightarrow\infty}\rho(\mu^{u_{y,z}(\cdot), y_{y,z}(\cdot), S},\mu) =0.
 \end{equation}
The control $u_{y,z}(\cdot) $ will be referred to as one {\it steering the associated system to $\mu $}.
\end{Definition}

\begin{Definition}\label{Def-steering-uniform}
Let $(u_z(\cdot),y_z(\cdot))$  be  an ACG   family  and let
$\mu(du,dy|z)$  be the  family of occupational measures generated by the former. We shall say that the family $\ \mu(du,dy|\cdot)$ is
{\it uniformly attainable}
 if the measure $\mu(du,dy|z)$ is  attainable by the associated system from any $y\in Y$ for any  $z\in Z$.
Moreover, the convergence in (\ref{e-opt-OM-2-103}) considered with $\mu = \mu(du,dy|z) $ is uniform with respect to $z\in Z$ and $y\in Y $. That is,
\begin{equation}\label{e-opt-OM-2-105}
\rho(\mu^{u_{y,z}(\cdot), y_{y,z}(\cdot), S},\mu(du,dy|z)) \leq \phi(S), \ \ \ \ \forall z\in Z  , \  \forall y\in Y,
 \end{equation}
 with $\ \lim_{S\rightarrow\infty}\phi(S) = 0 $.
\end{Definition}

Note that the family of  measures  $\ \mu(du,dy|z)$ generated by an arbitrary
ACG family $(u_z(\cdot),y_z(\cdot))$
is uniformly attainable by the associated system
if the convergence in (\ref{e:intro-0-3-3}) is uniform with respect to $z\in Z $ and $y\in Y $. Sufficient
conditions for this  can be found in \cite{Dont}, \cite{Gai1}, \cite{Gai8}, \cite{Gai-Leiz} and \cite{Gra} (see also
 Remark IV.1 at the end of this section).

Let $(u_z(\cdot),y_z(\cdot))$ be an ACG family and  let the  family of measures $\ \mu(du,dy|z)$ generated by this ACG
family be uniformly attainable by the associated system.
Define
 a control $u_{\epsilon}(\cdot) $ as follows.

 Let
 \begin{equation}\label{e:contr-rev-100-0-1}
 \Delta(\epsilon)\BYDEF \frac{\epsilon}{2L_f}\ln \frac{1}{\epsilon},
\end{equation}
where $L_f $ is a Lipschitz constant of $f(u,y,z)$ on  $Y\times Z$.
Partition the interval $[0,\infty) $ by the points
\begin{equation}\label{e:contr-rev-100-1}
t_l = l\Delta(\epsilon), \ l=0,1,....
\end{equation}
Note that
 \begin{equation}\label{e:contr-rev-100-2}
\lim_{\epsilon\rightarrow 0}\Delta(\epsilon)=0, \ \ \ \ \  \ \ \ \lim_{\epsilon\rightarrow 0}\frac{\Delta(\epsilon)}{\epsilon}=\infty .
\end{equation}
Define $u_{\epsilon}(t) $ on the interval $[0,t_1)  $ by the equation
\begin{equation}\label{e:contr-rev-100-3}
u_{\epsilon}(t) = u_{y_0,z_0}(\frac{t}{\epsilon}) \ \ \ \forall \ t\in [0,t_1),
\end{equation}
where $u_{y_0,z_0}(\tau)$ is the control steering the associated system (\ref{e:intro-0-3})
to
$\mu(du,dy|z_0)$ from the initial condition $y(0)= y_0$, the associated system being considered with
$z=z_0 $, where $(y_0,z_0)$ are as in  (\ref{e-initial-SP}).

Let us assume that the control $u_{\epsilon}(t)$ has been defined on the interval $t\in[0,t_l) $ and let $(y_{\epsilon}(t),z_{\epsilon}(t)) $ be
the corresponding solution of the SP system (\ref{e:intro-0-1})-(\ref{e:intro-0-2}) on this interval.
Extend  the definition of $u_{\epsilon}(t) $ to the interval $[0,t_{l+1}] $ by taking
\begin{equation}\label{e:contr-rev-100-4}
u_{\epsilon}(t) = u_{y_{\epsilon}(t_l),z_{\epsilon}(t_l)}(\frac{t-t_l}{\epsilon}) \ \ \ \forall \ t\in [t_l,t_{l+1}), \ \ \ l=0,1,... \ ,
\end{equation}
where $u_{y_{\epsilon}(t_l),z_{\epsilon}(t_l)}(\tau)$ is the control steering the associated system (\ref{e:intro-0-3}) to
$\mu(du,dy|z_{\epsilon}(t_l))$ from the initial condition $y(0)= y_{\epsilon}(t_l)$. Note that, in this instance, the associated system is considered
 with $z=z_{\epsilon}(t_l) $.

Below we establish that, under an additional technical assumption, the control $u_{\epsilon}(\cdot)$
constructed above is asymptotically optimal (near optimal) if the ACG family is optimal (near optimal). The needed assumption
 is introduced with the help of the following definition.

  \begin{Definition}\label{ASS-locally-Lipschitz}
 We will say that an ACG family  $(u_z(\cdot),y_z(\cdot))$ is weakly piecewise Lipschitz continuous  in a neighborhood of $z(\cdot)$
 if, for any Lipschitz continuous function $q(u,y,z)$, the function
$$
 \tilde q_{\mu}(z)  \BYDEF \int_{U\times Y}q(u,y,z)\mu(du,dy|z)
$$
 is piecewise Lipschitz continuous in a neighborhood of $z(\cdot)$, where $\mu(du,dy|z)$ is the family of measures generated by
 $(u_z(\cdot),y_z(\cdot))$ and
$z(\cdot) $ is the solution of (\ref{e-opt-OM-1}). The piecewise Lipschitz continuity of $ \tilde q_{\mu}(\cdot) $
  in a neighborhood of $z(\cdot)$ is understood in the following sense.
 For any $T>0$, there exists no more than a finite number of points $\bar t_i\in [0,T], \ i=1,...,k$ ($\bar t_i<\bar t_{i+1} $),
 such that $\tilde{g}_{\mu}(\cdot)$ satisfies Lipschitz condition on $ z(t) + a_t B$ for $t\neq\bar t_i  $,
 where $a_t > 0 $ and $B $ is the open unit ball in $\ R^n $.
 Moreover, $r_{\delta}$ defined by the equation
 $$
 r_{\delta} \BYDEF \inf \{ a_t \ : \ t\notin\cup_{i=1}^k ( \bar t_i - \delta , \bar t_i+ \delta )\}
 $$
is a positive
right-continuous function of $\delta $ (which
  may tend to zero when  $\delta $ tends to zero).
 \end{Definition}

\begin{Theorem}\label{Prop-convergence-measures-discounted}
Let the family of measures $\mu(du,dy|z) $
generated by an ACG family $(u_z(\cdot),y_z(\cdot))$ be uniformly attainable
 by the associated system and let this ACG family be
weakly piecewise Lipschitz continuous in a neighborhood of $z(\cdot)$
($z(\cdot)$ being the solution of (\ref{e-opt-OM-1})). Let also, for any $T>0 $, there exists
 $\ \bar \delta_T(\epsilon) > 0$, $\ \lim_{\epsilon\rightarrow 0}\bar \delta_T(\epsilon) = 0 $, such that the solution
$(y_{\epsilon}(\cdot),z_{\epsilon}(\cdot)) $ of the SP system (\ref{e:intro-0-1})-(\ref{e:intro-0-2}) obtained
with the control $u_{\epsilon}(\cdot)$ defined by (\ref{e:contr-rev-100-2}), (\ref{e:contr-rev-100-3})  satisfies the condition
 \begin{equation}\label{e:convergence-to-gamma-extra-condition}
z_{\epsilon}(t)\in z(t)+ a_tB \ \ \ \forall t\in [0,T] \setminus \cup_{i=1}^k ( \bar t_i - \bar \delta_T(\epsilon),
\bar t_i+ \bar \delta_T(\epsilon) ),
\end{equation}
where $a_t \ $ and $\ \bar t_i $ are as in Definition \ref{ASS-locally-Lipschitz}.
Then, for any $T>0 $,
 \begin{equation}\label{e:convergence-to-gamma-z-estimate}
max_{t\in [0,T]}||z_{\epsilon}(t)-z(t)||\leq \beta_T(\epsilon), \ \ \ \ \ \ \lim_{\epsilon\rightarrow 0}\beta_T(\epsilon) = 0,
\end{equation}
and
 the discounted occupational measure $\gamma^{\epsilon}_{di} $ generated by
  $(u_{\epsilon}(\cdot), y_{\epsilon}(\cdot),z_{\epsilon}(\cdot))$,  converges to $\gamma$ defined
 by the ACG family
 according to (\ref{e-opt-OM-2}). That is,
 \begin{equation}\label{e:convergence-to-gamma}
\lim_{\epsilon\rightarrow 0}\rho(\gamma^{\epsilon}_{di},\gamma)= 0.
\end{equation}
 \end{Theorem}

\begin{proof}
The proof is given in Section \ref{Sec-Main-AVE}. Note only here that it exploits ideas similar to those used in establishing
 earlier results on averaging of SP control systems (see, e.g., \cite{Art1}, \cite{Dont}, \cite{Gai1}, \cite{Gra}, \cite{QW}).
\end{proof}

Note that from (\ref{e:convergence-to-gamma}) (and from (\ref{e-opt-OM-2-1000-1}) and (\ref{e-opt-OM-2-1000-2})) it follows that
\begin{equation}\label{e:convergence-to-gamma-11}
C\lim_{\epsilon\rightarrow 0}\int_0^{\infty}e^{-Ct}G(u_{\epsilon}(t), y_{\epsilon}(t),z_{\epsilon}(t))dt =
\int_{U\times Y\times Z}G(u,y,z)\gamma(du,dy,dz)=C\tilde{V}^*_{di}(z_0)
\end{equation}
if the ACG family is optimal and that
\begin{equation}\label{e:convergence-to-gamma-12}
C\lim_{\epsilon\rightarrow 0} \int_0^{\infty}e^{-Ct}G(u_{\epsilon}(t), y_{\epsilon}(t),z_{\epsilon}(t))dt =
\int_{U\times Y\times Z}G(u,y,z)\gamma(du,dy,dz)\leq C(\tilde{V}^*_{di}(z_0) + \alpha)
\end{equation}
if the ACG family is $\alpha$-near optimal. These lead to the following corollary.

\begin{Corollary}\label{Cor-asym-near-opt}
Let an ACG family $(u_z(\cdot),y_z(\cdot))$ be optimal ($\alpha $-near optimal) and let
the assumptions of Theorem \ref{Prop-convergence-measures-discounted} be satisfied. Let also the equality (\ref{e-Objective-Convergence-Dis})
or  the equality (\ref{C-result-di-ave}) (or both of them) be valid. Then the control defined by (\ref{e:contr-rev-100-2}) and (\ref{e:contr-rev-100-3})
is asymptotically optimal ($\alpha $-near optimal) in the SP problem  (\ref{Vy-perturbed}).
\end{Corollary}
\begin{proof}
If (\ref{e-Objective-Convergence-Dis}) is valid, then the relationships (\ref{e-opt-convergence-1}) and (\ref{e-opt-convergence-2}) follow
from (\ref{e:convergence-to-gamma-11}) and (\ref{e:convergence-to-gamma-12}). If (\ref{C-result-di-ave}) is true, then these relationships
are implied by (\ref{e:oms-6-0}), (\ref{SP-IDLP-convergence-1-1}) and (\ref{e-ave-LP-opt-1}).
\end{proof}

REMARK IV.1.
As can be readily seen, the validity of
(\ref{e:convergence-to-gamma-extra-condition}) is  a necessary condition for the estimate (\ref{e:convergence-to-gamma-z-estimate}) to be valid
(that is, (\ref{e:convergence-to-gamma-extra-condition}) is satisfied if (\ref{e:convergence-to-gamma-z-estimate}) is true).
The validity
of (\ref{e:convergence-to-gamma-extra-condition})  can be directly verified to be true in case the associated system has the following property:
for any $z\in Z$, any $S>0$, any absolutely continuous function $\bar z(\tau): [0,S]\rightarrow Z $ and any control
 $u(\tau): [0,S]\rightarrow U $,
 \begin{equation}\label{e:Ass 2-Lip}
 \max_{\tau\in [0,S] } ||y_z(\tau) -\bar y(\tau)||\leq L \max_{\tau\in [0,S]}||z -\bar z(\tau)||, \ \ \ \ \ \ \ L={\rm const},
 \end{equation}
where $y_z(\cdot) $ is the solution of the system (\ref{e:intro-0-3}) obtained with the given control and  initial conditions
$y_z(0)\in Y $, and $\bar y(\tau)$ is the solution of the same system obtained with the same control, the same initial conditions
but with the replacement of $z$ by $\bar z(\tau) $. The associated system can be shown to have this property
  if the following stability type condition is satisfied(see Lemma 4.1 in \cite{Gai1}):
 for any $z\in Z$, the solutions $y_z(\tau, u(\cdot), y_1)$ and $y_z(\tau, u(\cdot), y_2)$ of
the  system (\ref{e:intro-0-3})  obtained with an arbitrary control $u(\cdot)$
and with  initial values $y_z(0)= y_1 $ and  $ y_z(0)= y_2$ ($ y_1 $ and  $ y_2 $ being arbitrary vectors in $Y$) satisfy the inequality
\begin{equation}\label{e-opt-OM-2-106}
||y_z(\tau, u(\cdot), y_1) - y_z(\tau, u(\cdot), y_2)||\leq \xi (\tau)||y_1-y_2||, \ \ \ \ \ \lim_{\tau\rightarrow\infty}\xi (\tau)=0.
 \end{equation}
Note that the above conditions is also sufficient for the convergence in (\ref{e:intro-0-3-3}) to be uniform with respect to $z\in Z $ and $y\in Y $ (see
Theorem 4.1(i) in \cite{Gai1} and Proposition 4.1 in \cite{Gai8}).

\section{Construction of asymptotically near optimal controls; Examples 1 and 2 (continued)}\label{Sec-construction-SP-examples}\
Provided that Assumption \ref{SET-1}(IV) is satisfied, the control $u_z^{N,M}(y_z^{N,M}(\tau),z) $ (with $u^{N,M}(y,z) $ being the minimizer in (\ref{e-NM-minimizer-1}) and $y_z^{N,M}(\tau)$ being the solution of  (\ref{e:opt-cond-AVE-1-fast-100-2-NM}))
steers the associated system to  $ \mu^{N,M}(du,dy|z)$ (see Definition \ref{Def-steering-1}). Therefore, one can follow (\ref{e:contr-rev-100-3})-(\ref{e:contr-rev-100-4}) to
construct a control $u^{N,M}_{\epsilon}(\cdot) $ that is asymptotically near optimal in the SP problem (\ref{Vy-perturbed}).
Namely, one can define the control $u^{N,M}_{\epsilon}(\cdot) $  on the interval $t\in [0,t_1) $ by the equation
\begin{equation}\label{e:near-opt-SP-10-1}
u^{N,M}_{\epsilon}(t) \BYDEF u^{N,M}(y^{N,M}_{z_0}(\frac{t}{\epsilon}), z_0)\ \ \  \ \ \forall \ t\in [0,t_1),
\end{equation}
where $y^{N,M}_{z_0}(\tau) $ is the solution of (\ref{e:opt-cond-AVE-1-fast-100-2-NM})
 obtained with $z=z_0 $ and with the initial condition $y^{N,M}_{z_0}(0)=y_0 $.
    Then, assuming that the control $u^{N,M}_{\epsilon}(t)$ has been defined on the interval $t\in[0,t_l) $ and that \\   $(y^{N,M}_{\epsilon}(t),z^{N,M}_{\epsilon}(t)) $ is
the corresponding solution of the SP system (\ref{e:intro-0-1})-(\ref{e:intro-0-2}) on this interval, one can
extend  the definition of $u^{N,M}_{\epsilon}(t) $ to the interval $[0,t_{l+1}] $ by taking
 \begin{equation}\label{e:near-opt-SP-10-2}
u^{N,M}_{\epsilon}(t) \BYDEF u^{N,M}(y^{N,M}_{z^{N,M}_{\epsilon}(t_l)}(\frac{t-t_l}{\epsilon}), z^{N,M}_{\epsilon}(t_l))\ \ \  \ \ \forall \ t\in [t_l,t_{1+1}),  \ \ \ l=0,1,... \ ,
\end{equation}
  where  $y^{N,M}_{z^{N,M}_{\epsilon}(t_l)}(\tau) $ is the solution of (\ref{e:opt-cond-AVE-1-fast-100-2-NM})
 obtained with $z=z^{N,M}_{\epsilon}(t_l) $ and with the initial condition $y^{N,M}_{z^{N,M}_{\epsilon}(t_l)}(0)=y^{N,M}_{\epsilon}(t_l) $.
The control
 $u^{N,M}_{\epsilon}(\cdot)$ will be asymptotically  $\beta(N,M)$-near optimal in the SP problem (\ref{Vy-perturbed}), with $\beta(N,M)$ satisfying (\ref{e:HJB-1-17-1-N-z-const-def-near-opt}), if all assumptions of Theorems \ref{Main-SP-Nemeric} and \ref{Prop-convergence-measures-discounted}  are satisfied.

 Note that the assumptions of Theorems \ref{Main-SP-Nemeric} and \ref{Prop-convergence-measures-discounted}  do not need to be
verified for one to be able to construct the control $u^{N,M}_{\epsilon}(t)$ defined by (\ref{e:near-opt-SP-10-1})-(\ref{e:near-opt-SP-10-2}).
 The latter can be constructed  as soon as an optimal solution of the $(N,M)$-approximating averaged problem (\ref{e-ave-LP-opt-di-MN}), its optimal value $\tilde{G}^{N,M}_{di}(z_0) $, and  solutions
 $\zeta^{N,M}(z)$, $\ \eta^{N,M}_z(y)$ of the  $(N,M)$-approximating averaged and associated dual problems
 are found for some $N$ and $M$ (a LP based algorithm for finding the latter is described in Section \ref{Sec-LP-based-algorithm}). Once
 the control $u^{N,M}_{\epsilon}(\cdot)$ is constructed, one can integrate the system (\ref{e:intro-0-1})-(\ref{e:intro-0-2})
 and find the value of the objective function $V_{di}^{N,M}(\epsilon, y_0, z_0)$ obtained with this control,
 \begin{equation}\label{e:near-opt-SP-10-3}
V_{di}^{N,M}(\epsilon, y_0, z_0)\BYDEF \int_0^ { + \infty }e^{-C t}G(u^{N,M}_{\epsilon}(t), y^{N,M}_{\epsilon}(t),
z^{N,M}_{\epsilon}(t))dt.
\end{equation}
Since (by (\ref{e-imp-inequality-1}) and (\ref{e:graph-w-M-103-101})),
$$
\limsup_{\epsilon\rightarrow 0}[V_{di}^{N,M}(\epsilon, y_0, z_0)-V_{di}^{*}(\epsilon, y_0, z_0)]\ \leq \
 \limsup_{\epsilon\rightarrow 0}V_{di}^{N,M}(\epsilon, y_0, z_0) - \liminf_{\epsilon\rightarrow 0}V_{di}^{*}(\epsilon, y_0, z_0)
$$
\vspace{-.2in}
\begin{equation}\label{e:near-opt-SP-10-4}
\leq \ \limsup_{\epsilon\rightarrow 0}V_{di}^{N,M}(\epsilon, y_0, z_0) - \frac{1}{C}\tilde{G}^{*}_{di}(z_0)\ \leq \ \limsup_{\epsilon\rightarrow 0}V_{di}^{N,M}(\epsilon, y_0, z_0) - \frac{1}{C}\tilde{G}^{N,M}_{di}(z_0),
\end{equation}
the difference $\ |V_{di}^{N,M}(\epsilon, y_0, z_0) - \frac{1}{C}\tilde{G}^{N,M}_{di}(z_0)| $ can serve as
a measure \lq\lq asymptotic near optimality" of the control $u^{N,M}_{\epsilon}(\cdot) $.

Let us resume the above in the form of steps that  one may follow to find an asymptotically near optimal control.

{\it (1) Choose test functions $\ \psi_i(z), \ i=1,...N $, $\ \ \phi_j(y), \ j=1,...M$, and construct the $(N,M)$-approximating averaged problem (\ref{e-ave-LP-opt-di-MN}) for some $N$ and $M$.

(2) Use the LP based algorithm of Section \ref{Sec-LP-based-algorithm} to find an optimal solution and the optimal value $\tilde{G}^{N,M}_{di}(z_0) $ of the  problem (\ref{e-ave-LP-opt-di-MN}), as well as  a solution  $\zeta^{N,M}(z)$ of the  $(N,M)$-approximating averaged
dual problem (\ref{e:DUAL-AVE-0-approx-MN})
and a solution $\ \eta^{N,M}_z(y)$ of the  $(N,M)$-approximating  associated dual problem (\ref{e:dec-fast-4-Associated});

(3) Define $u^{N,M}(y,z) $ according to  (\ref{e-NM-minimizer-1}) and construct the control $u^{N,M}_{\epsilon}(\cdot)$
according to (\ref{e:near-opt-SP-10-1})-(\ref{e:near-opt-SP-10-2});

(4) Substitute the control $u^{N,M}_{\epsilon}(\cdot)$ into the system (\ref{e:intro-0-1})-(\ref{e:intro-0-2}) and integrate the obtained ODE with MATLAB.
Also, use MATLAB to evaluate the objective function $V_{di}^{N,M}(\epsilon, y_0, z_0)$;

(5) Assess  the proximity of the found solution to the optimal one by evaluating the difference\\ $\ |V_{di}^{N,M}(\epsilon, y_0, z_0) - \frac{1}{C}\tilde{G}^{N,M}_{di}(z_0)| $.}

\medskip

EXAMPLE 1 (Continued). Consider the   SP optimal control problem defined by the equations (\ref{e:EXAMPL2-4})-(\ref{e:ex-4-1-rep}).
The $(N,M)$-approximating averaged problem (\ref{e-ave-LP-opt-di-MN}) for this example
was constructed with the use of  powers $\ z^i$ as $\ \psi_i(z) $
and monomials $\ y_1^{i_1}y_2^{i_2}$ as $\ \phi_{i_1,i_2}(y)$ and with
$N=15$, $M=35$ (note the change in the indexation of the test functions and recall that $N$ stands for the number of constraints in (\ref{D-SP-new-MN})
and $ M$ stands for the number of constraints in (\ref{e:2.4-M})),
\begin{equation}\label{e:near-opt-SP-10-5}
\psi_i(z)\BYDEF z^i, \ \ i=1,..., 15, \ \ \ \ \ \ \ \ \ \ \phi_{i_1,i_2}(y)=y_1^{i_1}y_2^{i_2}, \ \  \ \  1 \leq i_1+ i_2\leq  5.
\end{equation}
This problem was solved with the algorithm of Section  \ref{Sec-LP-based-algorithm}. The optimal
value of the problem was obtained to be approximately equal to $0.853$:
\begin{equation}\label{e:near-opt-SP-10-6}
\tilde{G}^{15,35}_{di}(z_0)\approx 0.853.
\end{equation}
The expansions (\ref{e:DUAL-AVE-0-approx-2}) and (\ref{e:DUAL-AVE-0-approx-1-Associate-1}) defining
solutions of the  $(N,M)$-approximating averaged and dual problems take the form
\begin{equation}\label{e:near-opt-SP-10-7}
\zeta^{15,35}(z)=\sum_{i=1}^{15} \lambda_i^{15,35}z^i , \ \ \ \ \ \ \ \ \ \ \eta^{15,35}_z(y)=\sum_{1\leq i_1+i_2\leq 5} \omega_{z,i_1,i_2}^{15,35} \ y_1^{i_1}y_2^{i_2},
\end{equation}
where the coefficients $\ \{ \lambda_i^{15,35} \}$  and $\{  \omega_{z,i_1,i_2}^{15,35}\} $ are  obtained as a part of the solution of the problem
(\ref{e-ave-LP-opt-di-MN}) with the algorithm of Section  \ref{Sec-LP-based-algorithm}.

Using $\zeta^{15,35}(z)$ and $\eta^{15,35}_z(y) $, one can compose the problem (\ref{e-NM-minimizer-0}), which in this case takes  the form
\begin{equation}\label{e-NM-minimizer-0-example}
\min_{u_i\in [-1,1]}\{u_1^2+u_2^2+y_1^2+y_2^2+z^2+ \frac{d\zeta^{15,35}(z)}{dz}(-y_1u_2 + y_2u_1) + \frac{\partial\eta^{15,35}_z(y)}{\partial y_1}(-y_1 +u_1)
+\frac{\partial\eta^{15,35}_z(y)}{\partial y_2}(-y_2 +u_2)\}. \ \ \ \
\end{equation}
The solution of the problem (\ref{e-NM-minimizer-0-example}) is as follows
\begin{equation}\label{feedbackFinalVel-SP}
u^{15,35}_i(y,z)=\left\{
\begin{array}{rrrl}
- \frac{1}{2} a_i^{15,35}(y,z) \ & \ \ \ \ \ if&  \ \ \ \ \ \ \ \ \ \ \ \  |\frac{1}{2} a_i^{15,35}(y,z)|\leq 1,  
\\
-1  \ & \ \ \ \  \  if&  \ \ \ \ \ \ \ \ \ \ \ \  \ -\frac{1}{2} a_i^{15,35}(y,z) < -1,
\\
1  \ & \ \ \ \ \  if&  \ \ \ \ \ \ \ \ \ \ \ \  \
-\frac{1}{2} a_i^{15,35}(y,z) > 1,
\\
\end{array}
\right\},
\ \ \ i=1,2,
\end{equation}
where
$\ a_1^{15,35}(y,z)\BYDEF \frac{d\zeta^{15,35}(z)}{dz} y_2 + \frac{\partial\eta^{15,35}_z(y)}{\partial y_1} \ $
and $\ a_2^{15,35}(y,z)\BYDEF -\frac{d\zeta^{15,35}(z)}{dz} y_1 + \frac{\partial\eta^{15,35}_z(y)}{\partial y_2}\ $.

Construct the control $u_{\epsilon}^{15,35}(t) $ according to (\ref{e:near-opt-SP-10-1})-(\ref{e:near-opt-SP-10-2}). Note that, since, in the given example, the equations
describing the fast dynamics do not depend on the slow component $z(\cdot)$, the control $u_{\epsilon}^{15,35}(t)=(u^{15,35}_{1, \epsilon}(t),u^{15,35}_{2, \epsilon}(t )) $ can be presented in a more explicit feedback form
\begin{equation}\label{e:near-opt-SP-10-8}
 u^{15,35}_{i, \epsilon}(t) \BYDEF u^{15,35}_i( y^{15,35}_{\epsilon}(t),  z^{15,35}_{\epsilon}(t_l))\ \ \
 \ \ \forall \ t\in [t_l,t_{1+1}), \ \ \ l=0,1,... \ ,\ \ \ \ \ \ \ \ i=1,2,
\end{equation}
where $(y^{15,35}_{\epsilon}(\cdot),  z^{15,35}_{\epsilon}(\cdot)) $ is the solution of the system (\ref{e:ex-4-2-repeat})-(\ref{e:ex-4-1-rep})
obtained with the control $\ u^{15,35}_{ \epsilon}(\cdot)$
(for convenience,  we omit the superscripts below and write $u_{i, \epsilon}(t)$, $y_{i, \epsilon}(t)$ and $z_{\epsilon}(t)$ instead of
$u^{15,35}_{i, \epsilon}(t)$, $y^{15,35}_{i, \epsilon}(t)$ and $z^{15,35}_{\epsilon}(t)$).

The graphs of $u_{i, \epsilon}(t)$, $y_{i, \epsilon}(t)$ and $z_{\epsilon}(t)$ obtained via
the integration with MATLAB of the system (\ref{e:ex-4-2-repeat})-(\ref{e:ex-4-1-rep}) considered  with $\epsilon=0.1 $  and $\epsilon=0.01 $ are depicted in Figures 1,2,3 and 4,5,6 (respectively). Note that in the process of integration the lengths of the intervals $[t_l,t_{1+1}) $ were chosen experimentally
 to optimize the results (that is, they were not automatically taken to be equal to $\Delta(\epsilon) $ as in (\ref{e:contr-rev-100-1})). The values of the
objective function  were obtained to be approximately equal to $ 8.56$ and $ 8.54$ (respectively), both values being close to $\  \frac{1}{C}\tilde{G}^{15,35}_{di}(z_0)=8.53\ $ (see (\ref{e:near-opt-SP-10-6}); recall that
  $\ C=0.1\ $ in this case). Hence, the constructed control can be considered to be a \lq\lq good candidate" for being asymptotically near optimal.


\begin{center}
{\it Controls and state components as functions of time for $\epsilon=0.1 $}
\end{center}
\begin{center}
\vspace{-.1in}
\includegraphics[scale=0.30]{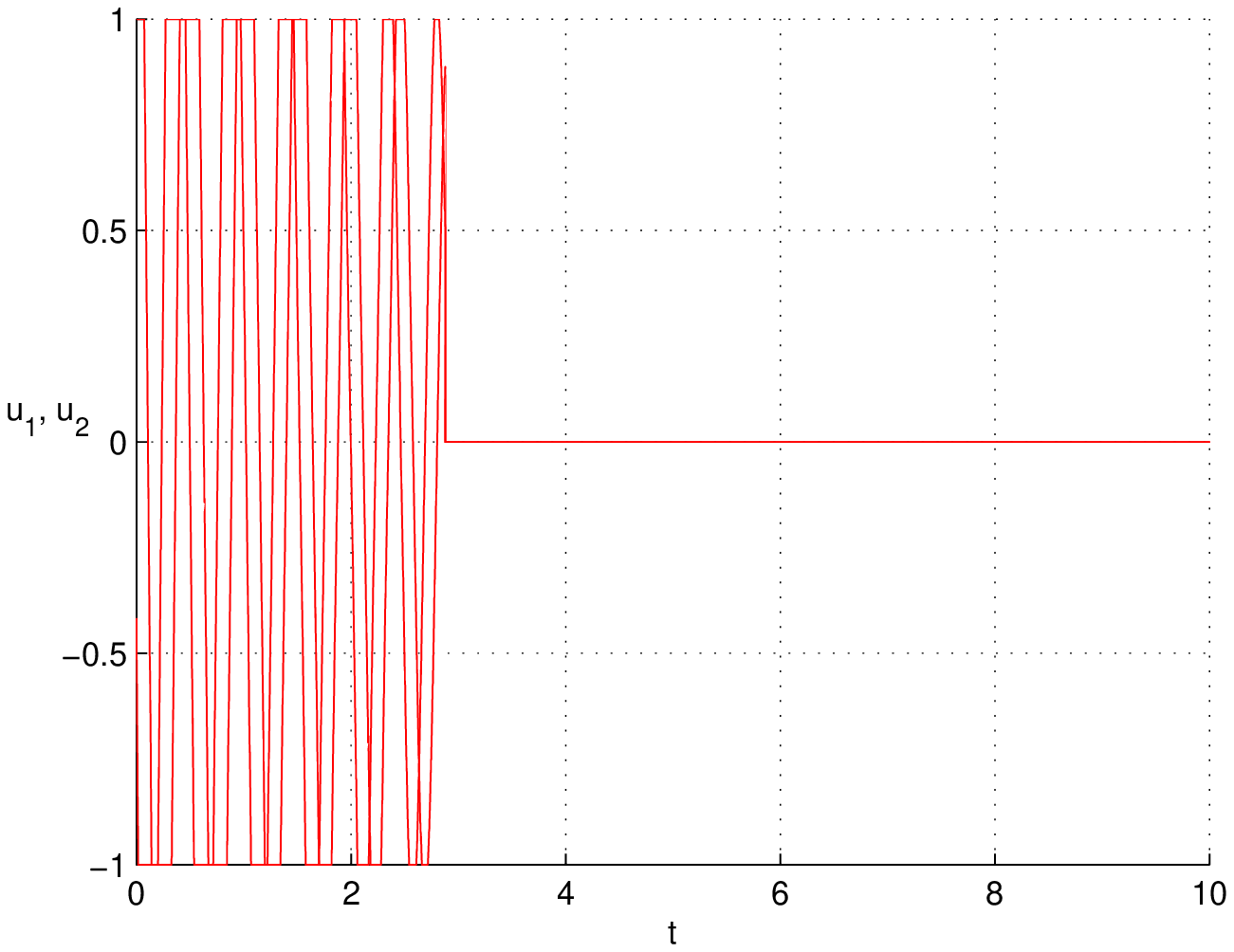}
\includegraphics[scale=0.30]{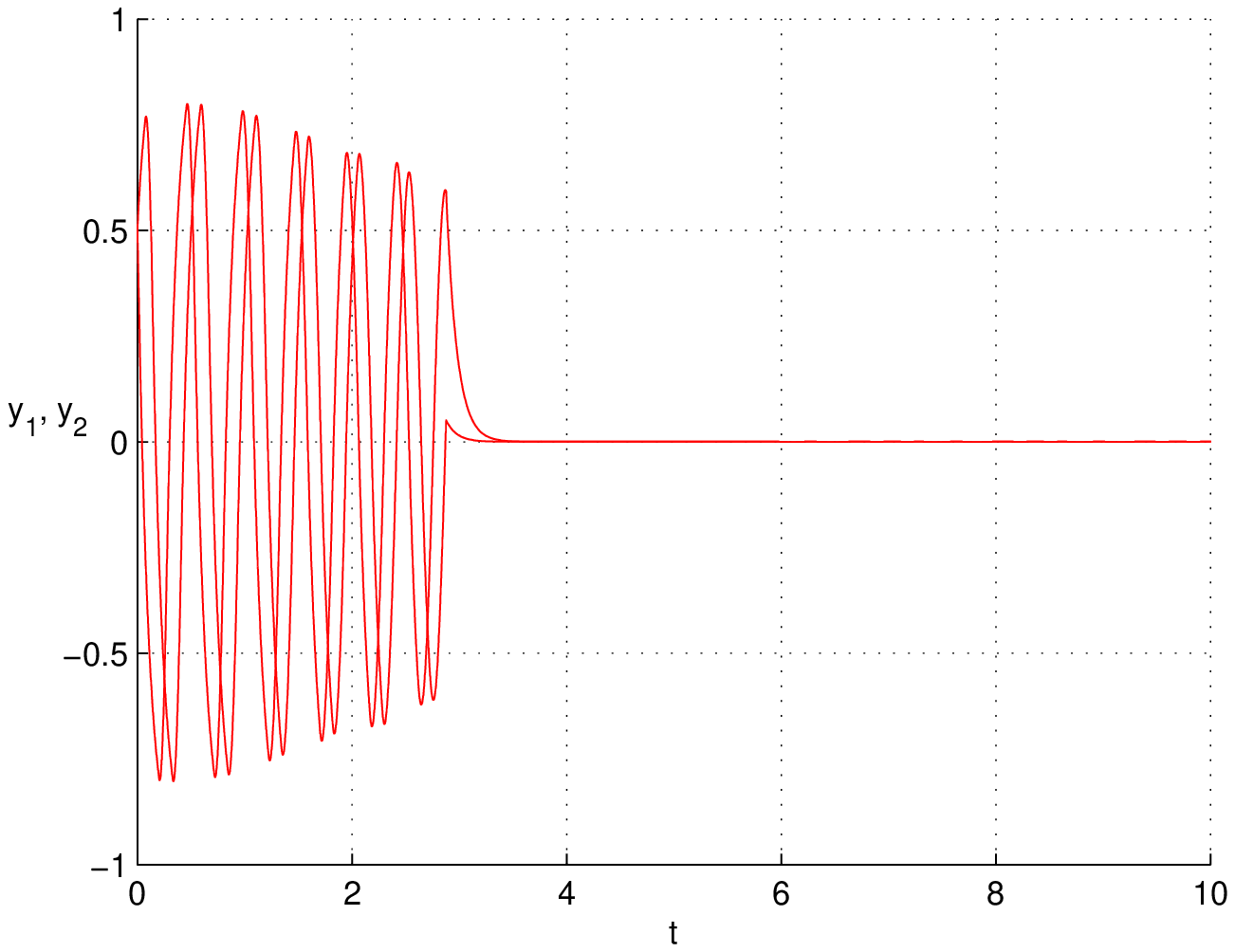}
\includegraphics[scale=0.30]{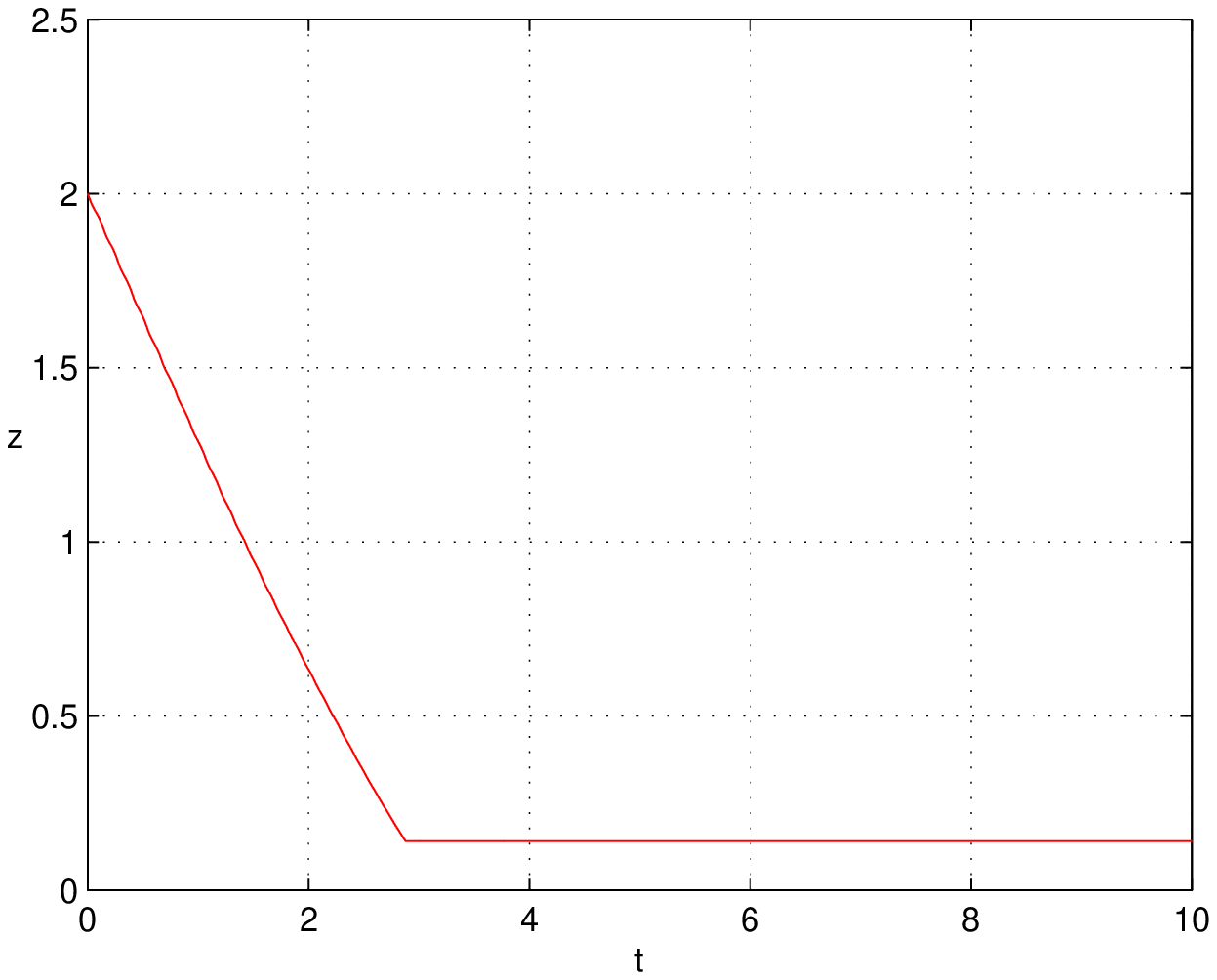}
\end{center}
\vspace{-.1in}
\ \ \ \ \ \ \ \ \ \ \ \ \ \ \ \ \ Fig. 1: \  $u_{1, \epsilon}(t),u_{2, \epsilon}(t )$
 \ \ \ \ \ \ \ \ \  Fig. 2: \ $y_{1, \epsilon}(t),y_{2, \epsilon}(t )$ \ \ \ \ \ \ \ \ \ \ \ \ \ \ \ \ \  Fig. 3: \ $z_{ \epsilon}(t)$

\bigskip

\begin{center}
{\it Controls and state components as functions of time for $\epsilon=0.01 $}
\end{center}
\begin{center}
\vspace{-.1in}
\includegraphics[scale=0.30]{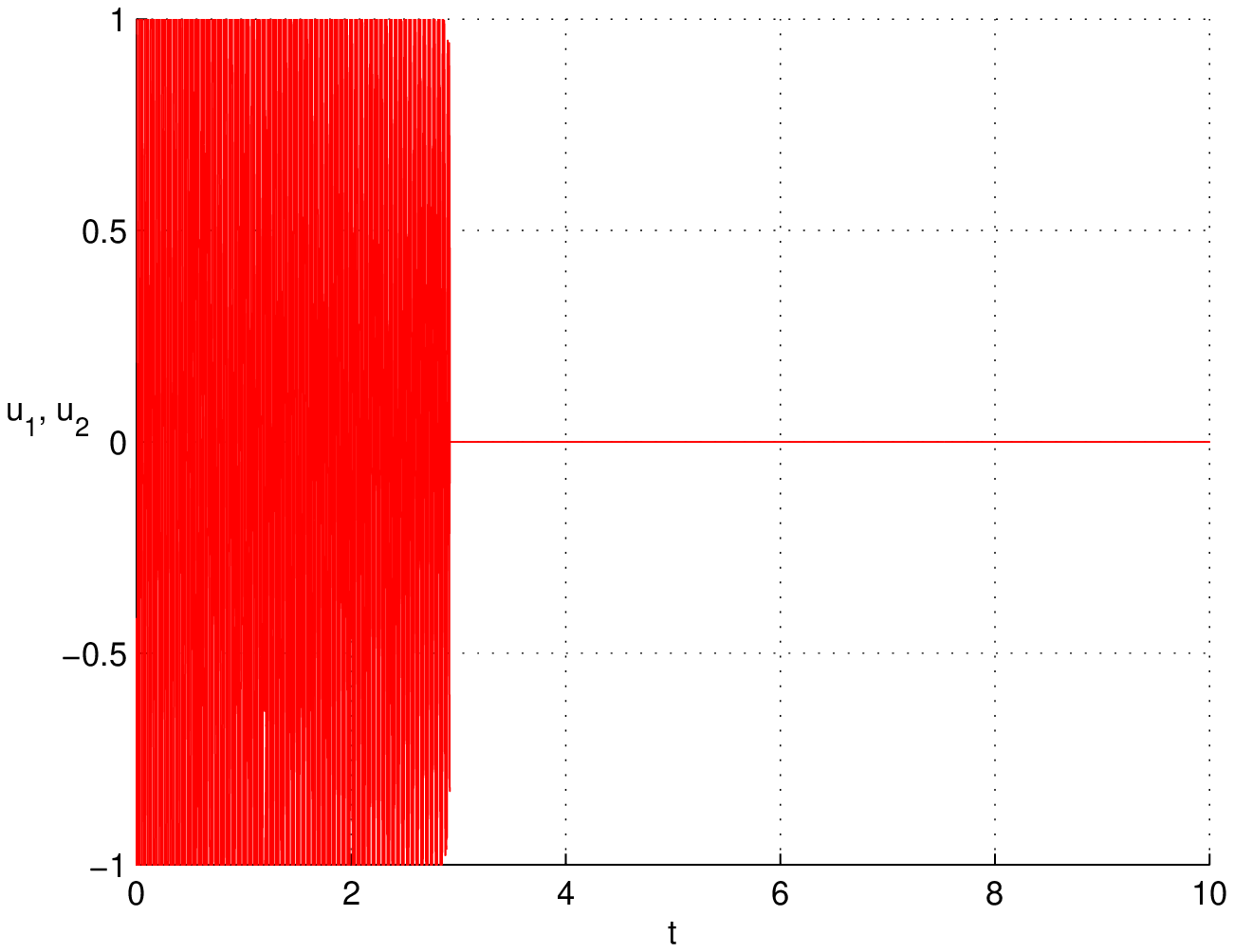}
\includegraphics[scale=0.30]{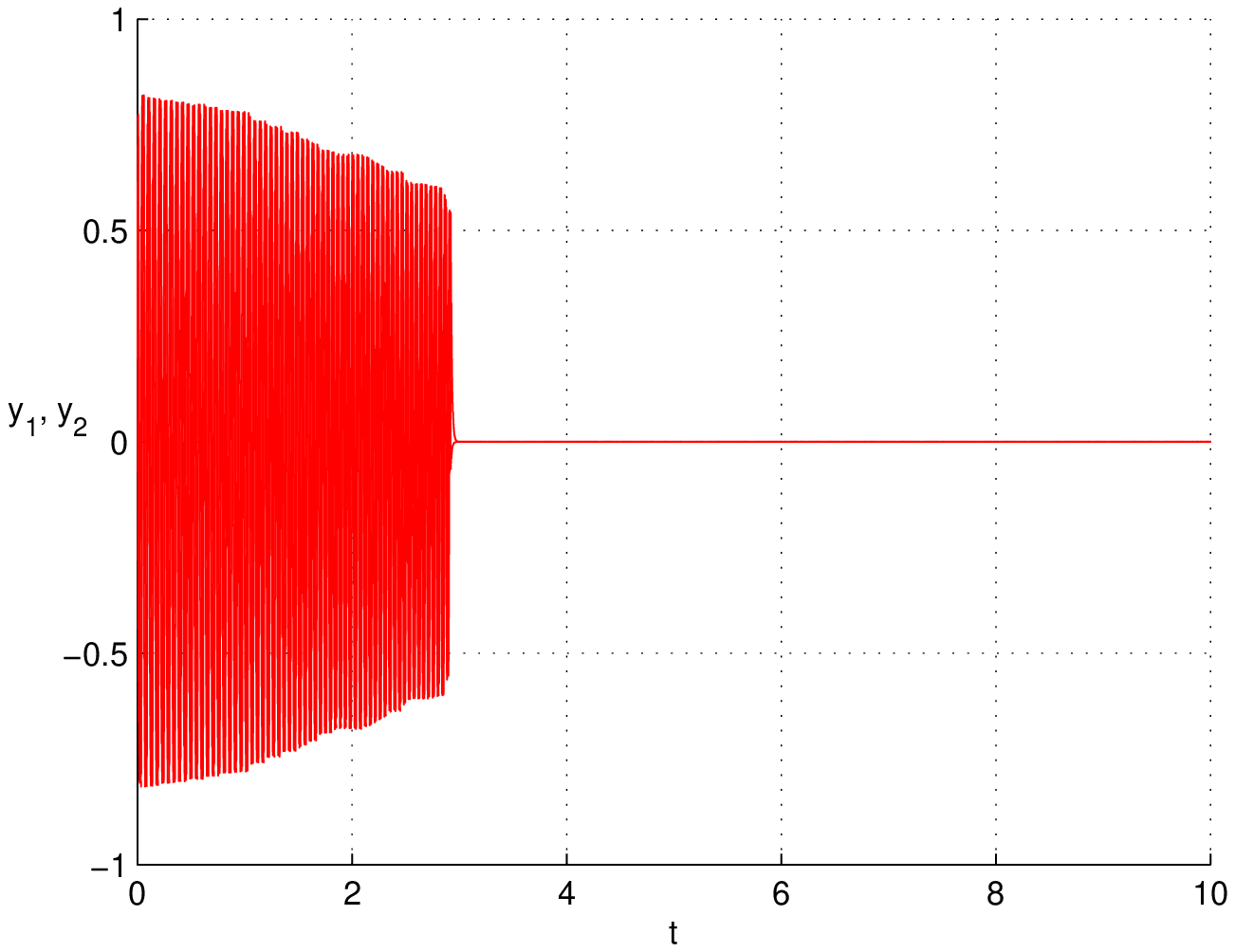}
\includegraphics[scale=0.30]{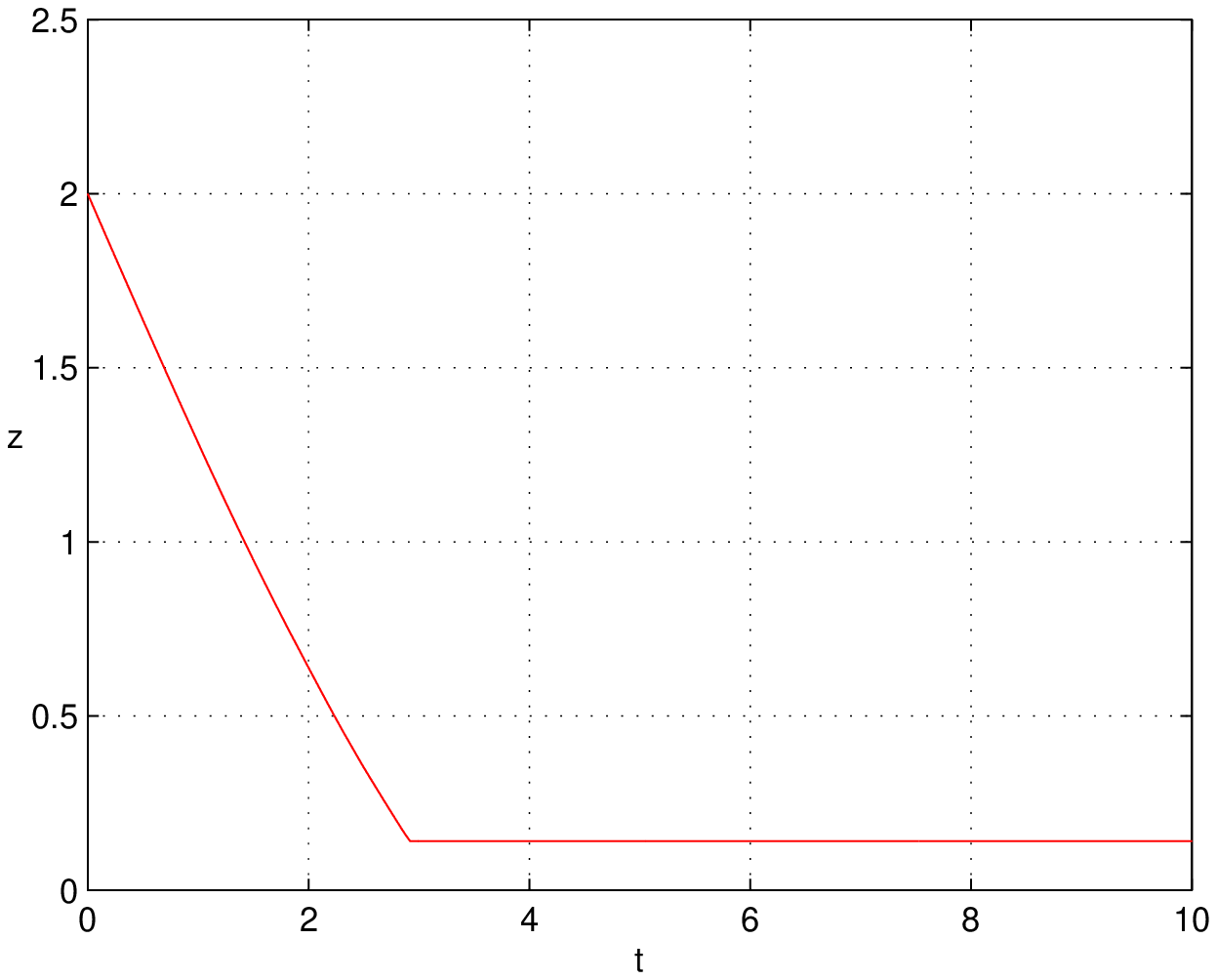}
\end{center}
\vspace{-.1in}
\ \ \ \ \ \ \ \ \ \ \ \ \ \ \ \ \ Fig. 4: \  $u_{1, \epsilon}(t),u_{2, \epsilon}(t )$
 \ \ \ \ \ \ \ \ \  Fig. 5: \ $y_{1, \epsilon}(t),y_{2, \epsilon}(t )$ \ \ \ \ \ \ \ \ \ \ \ \ \ \ \ \ \  Fig. 6: \ $z_{ \epsilon}(t)$

\bigskip

 The state trajectories corresponding to the two cases ($\epsilon=0.1 $  and $\epsilon=0.01 $) are depicted in
 Figures 7 and 8. As can be seen, the  fast state  variables $y_{i,\epsilon}(\cdot), \ i=1,2, $ move along a square like figure that  gradually changes its shape
 while the slow variable  $z_{\epsilon}(\cdot) $ is decreasing  from the initial level $z(0)=2 $
  to the level $z\approx 0.14$, which is reached at the moment $t\approx 2.9 $. After this moment, the zero controls
are applied and the fast variables are rapidly converging to zero, with the slow variable stabilizing and remaining  approximately equal to $0.14 $.

\begin{center}
{\it State trajectories  for $\ \epsilon = 0.1$  and $\ \epsilon = 0.01$  }
\end{center}
\begin{center}
\vspace{-.1in}
\includegraphics[scale=0.36]{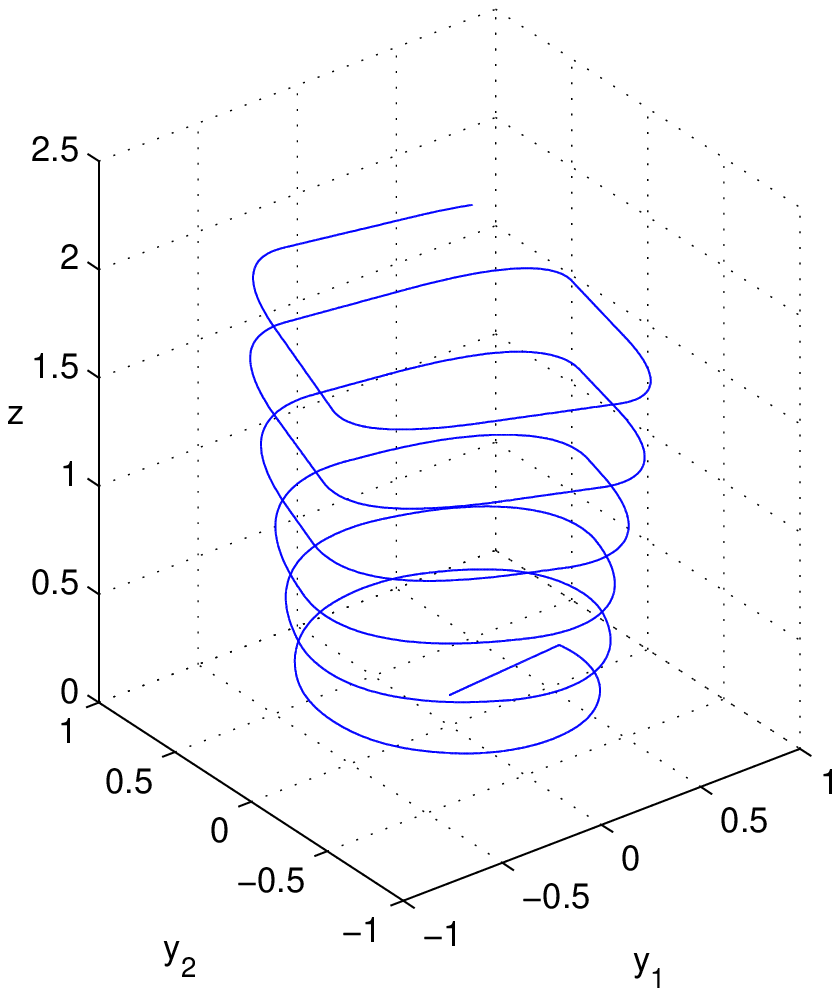}
\hspace{+.12in}
\includegraphics[scale=0.36]{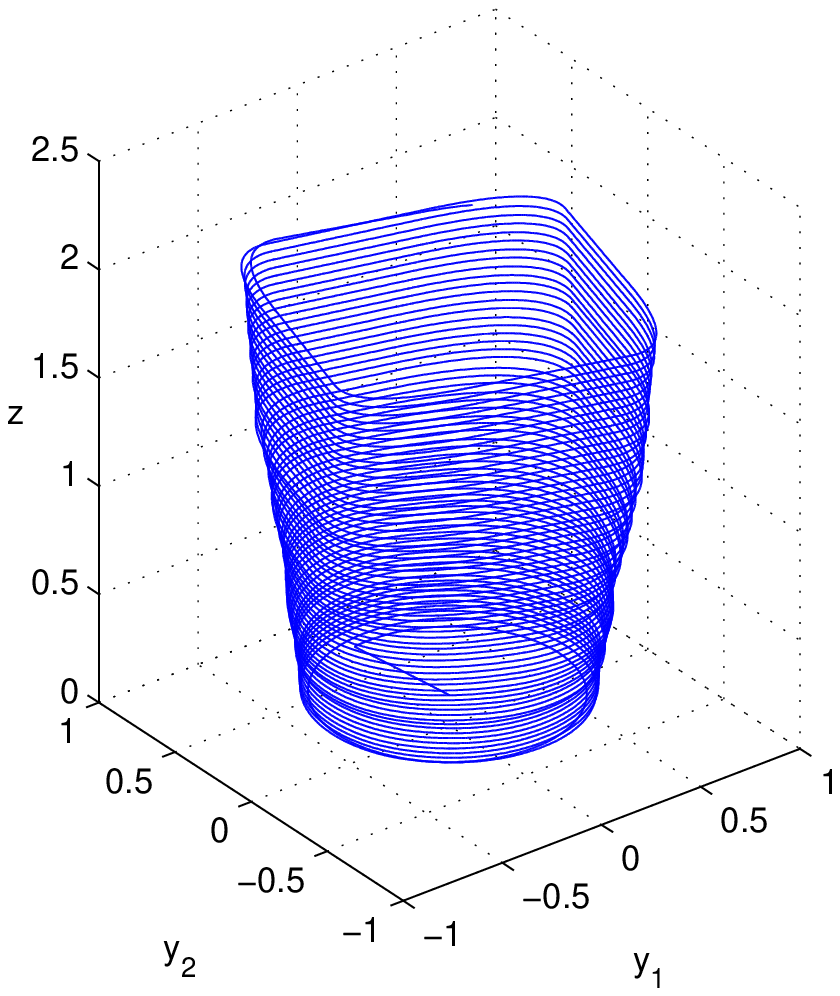}
\end{center}
 \vspace{-.1in}
\ \ \ \ \ \ \ \ \ \   \ \ \ \ \ \ \ \ \  \ \ \ \ \ Fig. 7:  $\ (y_{ \epsilon}(t),z_{\epsilon}(t))\  $  for $\epsilon = 0.1$ \ \ \ \ \ \ \  \ \ \ \ \ \
 Fig.8: \ $\ (y_{ \epsilon}(t),z_{\epsilon}(t))$ for $\epsilon = 0.01$

\bigskip

REMARK IV.2. As has been mentioned earlier (see Remarks II.2 and III.2), many results obtained for the SP optimal control
problem with time discounting have their counterparts in the periodic optimization setting. This remains valid also
for the consideration of this section (as well as  that of Section \ref{Sec-SP-ACG-theorem}). In particular, the process of construction of a control that is asymptotically
near optimal in the SP periodic optimization problem  (\ref{Vy-perturbed-per-new-1}) on the basis of an ACG family that generates a near optimal
solution the averaged problem (\ref{Vy-perturbed-per-new-10}) is similar to that outlined in steps $(1)$ to $(5)$ above,
with $u^{N,M}(y,z)$ being a minimizer  in (\ref{e-NM-minimizer-1}) (and with $\zeta^{N,M}(z)$ and $\eta^{N,M}_z (y)$ being solutions of the corresponding
$(N,M) $-approximating  averaged and associated dual problems; see Remark III.2).
Denote thus obtained control as $\ u^{N,M}_{\epsilon}(t)$ and  the corresponding periodic solution of the system  (\ref{e:intro-0-1})-(\ref{e:intro-0-2})
(assume that it exists) as $(y^{N,M}_{\epsilon}(t), z^{N,M}_{\epsilon}(t))$, the period of the latter being denoted as $T^{N,M}$. Denote also
by $V_{per}^{N,M}(\epsilon)$ the corresponding value of the objective function:
 \begin{equation}\label{e:near-opt-SP-10-3-ave}
V_{per}^{N,M}(\epsilon)\BYDEF
 \frac{1}{T^{N,M}}\int_0^{T^{N,M}} G(u^{N,M}_{\epsilon}(t), y^{N,M}_{\epsilon}(t),
z^{N,M}_{\epsilon}(t))dt.
\end{equation}
By (\ref{Vy-perturbed-per-new-9}) and by (\ref{e:graph-w-M-103-101-ave }),
$$
\limsup_{\epsilon\rightarrow 0}[V_{per}^{N,M}(\epsilon)-V_{per}^{*}(\epsilon)]\ \leq \
 \limsup_{\epsilon\rightarrow 0}V_{per}^{N,M}(\epsilon) - \liminf_{\epsilon\rightarrow 0}V_{per}^{*}(\epsilon)
$$
\vspace{-.2in}
\begin{equation}\label{e:near-opt-SP-10-4-ave}
\leq \ \limsup_{\epsilon\rightarrow 0}V_{per}^{N,M}(\epsilon) - \tilde{G}^{*}\ \leq \ \limsup_{\epsilon\rightarrow 0}V_{per}^{N,M}(\epsilon) - \tilde{G}^{N,M},
\end{equation}
where $\tilde{G}^{*} $ is the optimal of the IDLP problem (\ref{Vy-perturbed-per-new-8}) and $\tilde{G}^{N,M} $
is the optimal value of the $(N,M)$-approximating problem (\ref{e-ave-LP-opt-di-MN-ave}) (compare with (\ref{e:near-opt-SP-10-4})). That is,
the difference $\ |V_{per}^{N,M}(\epsilon) - \tilde{G}^{N,M}| $ can serve as
a measure of the asymptotic near optimality of the control $u^{N,M}_{\epsilon}(\cdot) $ in the  periodic optimization case.
Thus,  one may conclude that, to find an asymptotically near optimal control for the SP periodic
optimization problem (\ref{Vy-perturbed-per-new-1}), one may follow the steps similar to $(1)-(5)$, with the differences being as follows: {\it (i) A solutions of the $(N,M)$ approximating problem (\ref{e-ave-LP-opt-di-MN-ave}) and  the corresponding averaged and associated dual problems should be used  instead of solutions of the problem (\ref{e-ave-LP-opt-di-MN}) and its corresponding duals; (ii) A periodic solution of the SP system
(\ref{e:intro-0-1})-(\ref{e:intro-0-2})  should be sought instead of one satisfying the initial condition (\ref{e-initial-SP}); (iii) The difference $\ |V_{per}^{N,M}(\epsilon) - \tilde{G}^{N,M}| $ should be used as a measure
of asymptotic near optimality of the control $u^{N,M}_{\epsilon}(\cdot) $.}

EXAMPLE 2 (Continued).
Consider the   periodic optimization problem  (\ref{e:ex-4-1-rep-101-2}).
The $(N,M)$-approximating averaged problem (\ref{e-ave-LP-opt-di-MN-ave})
was constructed with the use of  monomials $\ z^{j_1}_1z^{j_2}_2$
as $\ \psi_{j_1,j_2}(z) $
and  monomials $\ y_1^{i_1}y_2^{i_2}$ as $\ \phi_{i_1,i_2}(y)$ and with $N,M=35$,
\begin{equation}\label{e:near-opt-SP-10-5-per}
\psi_{j_1,j_2}(z)\BYDEF z^{j_1}_1z^{j_2}_2, \ \ 1 \leq j_1+ j_2\leq  5, \ \ \ \ \ \ \ \ \ \ \phi_{i_1,i_2}(y)=y_1^{i_1}y_2^{i_2}, \ \  \ \  1 \leq i_1+ i_2\leq  5.
\end{equation}
 Solving this problem with the algorithm of Section  \ref{Sec-LP-based-algorithm}, one finds  its optimal
value
\begin{equation}\label{e:near-opt-SP-10-6-per}
\tilde{G}^{35,35}\approx -1.186.
\end{equation}
as well as the coefficients of the expansions
\begin{equation}\label{e:near-opt-SP-10-7-per}
\zeta^{35,35}(z)=\sum_{1\leq j_1+j_2\leq 5} \lambda_{j_1,j_2}^{35,35}z^{j_1}_1z^{j_2}_2 , \ \ \ \ \ \ \ \ \ \ \eta^{35,35}_z(y)=\sum_{1\leq i_1+i_2\leq 5} \omega_{z,i_1,i_2}^{35,35} \ y_1^{i_1}y_2^{i_2},
\end{equation}
that define
solutions of the  $(N,M)$-approximating averaged and dual problems (see (\ref{e:DUAL-AVE-0-approx-2}) and (\ref{e:DUAL-AVE-0-approx-1-Associate-1})).
Using $\zeta^{35,35}(z)$ and $\eta^{35,35}_z(y) $, one can compose the problem (\ref{e-NM-minimizer-0}):
$$
\min_{u_i\in [-1,1]}\{0.1u_1^2+ 0.1u_2^2- z_1^2+ \frac{\partial \zeta^{35,35}(z)}{\partial z_1}z_2 + \frac{\partial \zeta^{35,35}(z)}{\partial z_2}(-4z_1-0.3 z_2 -y_1u_2 + y_2u_1) + \frac{\partial\eta^{35,35}_z(y)}{\partial y_1}(-y_1 +u_1)
$$
\vspace{-.2in}
\begin{equation}\label{e-NM-minimizer-0-example-per}
\ \ \ \ \ \ \ \ \ \ \ \ \ \ \ \ \ \ \ \ \ \ \ +\ \frac{\partial\eta^{35,35}_z(y)}{\partial y_2}(-y_2 +u_2)\}.
\end{equation}
The solution of the problem (\ref{e-NM-minimizer-0-example-per}) is similar to that of (\ref{feedbackFinalVel-SP}) and is written in the form
\begin{equation}\label{feedbackFinalVel-SP-per}
u^{35,35}_i(y,z)=\left\{
\begin{array}{rrrl}
- 5 b_i^{35,35}(y,z) \ & \ \ \ \ \ if&  \ \ \ \ \ \ \ \ \ \ \ \  |5 b_i^{35,35}(y,z)|\leq 1,  
\\
-1  \ & \ \ \ \  \  if&  \ \ \ \ \ \ \ \ \ \ \ \  \ -5 b_i^{35,35}(y,z) < -1,
\\
1  \ & \ \ \ \ \  if&  \ \ \ \ \ \ \ \ \ \ \ \  \
-5 b_i^{35,35}(y,z) > 1,
\\
\end{array}
\right\},
\ \ \ i=1,2,
\end{equation}
where
$\ b_1^{35,35}(y,z)\BYDEF \frac{\partial \zeta^{35,35}(z)}{\partial z_2} y_2 + \frac{\partial\eta^{35,35}_z(y)}{\partial y_1} \ $
and $\ b_2^{35,35}(y,z)\BYDEF -\frac{\partial \zeta^{35,35}(z)}{\partial z_2} y_1 + \frac{\partial\eta^{35,35}_z(y)}{\partial y_2}\ $.

As in Example 1, the equations
describing the fast dynamics do not depend on the slow component $z(\cdot)$. Hence, the
the control $\ u_{\epsilon}^{35,35}(t)=(u^{15,35}_{1, \epsilon}(t),u^{15,35}_{2, \epsilon}(t ))\ $ defined by (\ref{e:near-opt-SP-10-1})-(\ref{e:near-opt-SP-10-2}) can be written in the feedback form:
\begin{equation}\label{e:near-opt-SP-10-8-per}
 u^{35,35}_{i, \epsilon}(t) \BYDEF u^{35,35}_i( y^{35,35}_{\epsilon}(t),  z^{35,35}_{\epsilon}(t_l))\ \ \
 \ \ \forall \ t\in [t_l,t_{1+1}), \ \ \ l=0,1,... \ ,\ \ \ \ \ \ \ \ i=1,2,
\end{equation}
where $(y^{35,35}_{\epsilon}(\cdot),  z^{35,35}_{\epsilon}(\cdot)) $ is the solution of the system (\ref{e:ex-4-2-repeat}),(\ref{e:ex-4-1-rep-101-1})
obtained with the control $\ u^{35,35}_{ \epsilon}(\cdot)$.

The periodic solution of the system (\ref{e:ex-4-2-repeat}),(\ref{e:ex-4-1-rep-101-1}) was found with MATLAB for $\epsilon =0.01 $
and $\epsilon = 0.001$. The images of the state trajectories obtained as the result of the integration are depicted in Figures 9 and 10
(where, again, the superscripts are omitted from the notations).
The slow $z $-components appear to be  moving periodically along a closed, ellipse like, figure on the plane $(z_1,z_2)$, with the period being approximately equal
 to $3.16 $. Note that this  figure  and the period  appear to be the same for $\epsilon =0.01 $
and $\epsilon = 0.001$. The fast $y $-components  are moving along  square like figures centered around the points on  the \lq\lq ellipse", with about $50$  rounds for the case  $\epsilon = 0.01$ and about $500$  rounds for the case  $\epsilon = 0.001$. The  values of the objective functions obtained for these two cases
are approximately the same and $\approx -1.177$, the latter being close to the value of $\tilde{G}^{35,35}$ (see (\ref{e:near-opt-SP-10-6-per})).

\bigskip

\begin{center}
{\it Images of state trajectories  for $\ \epsilon = 0.01$  and $\ \epsilon = 0.001$  }
\end{center}
\begin{center}
\vspace{-.1in}
\includegraphics[scale=0.36]{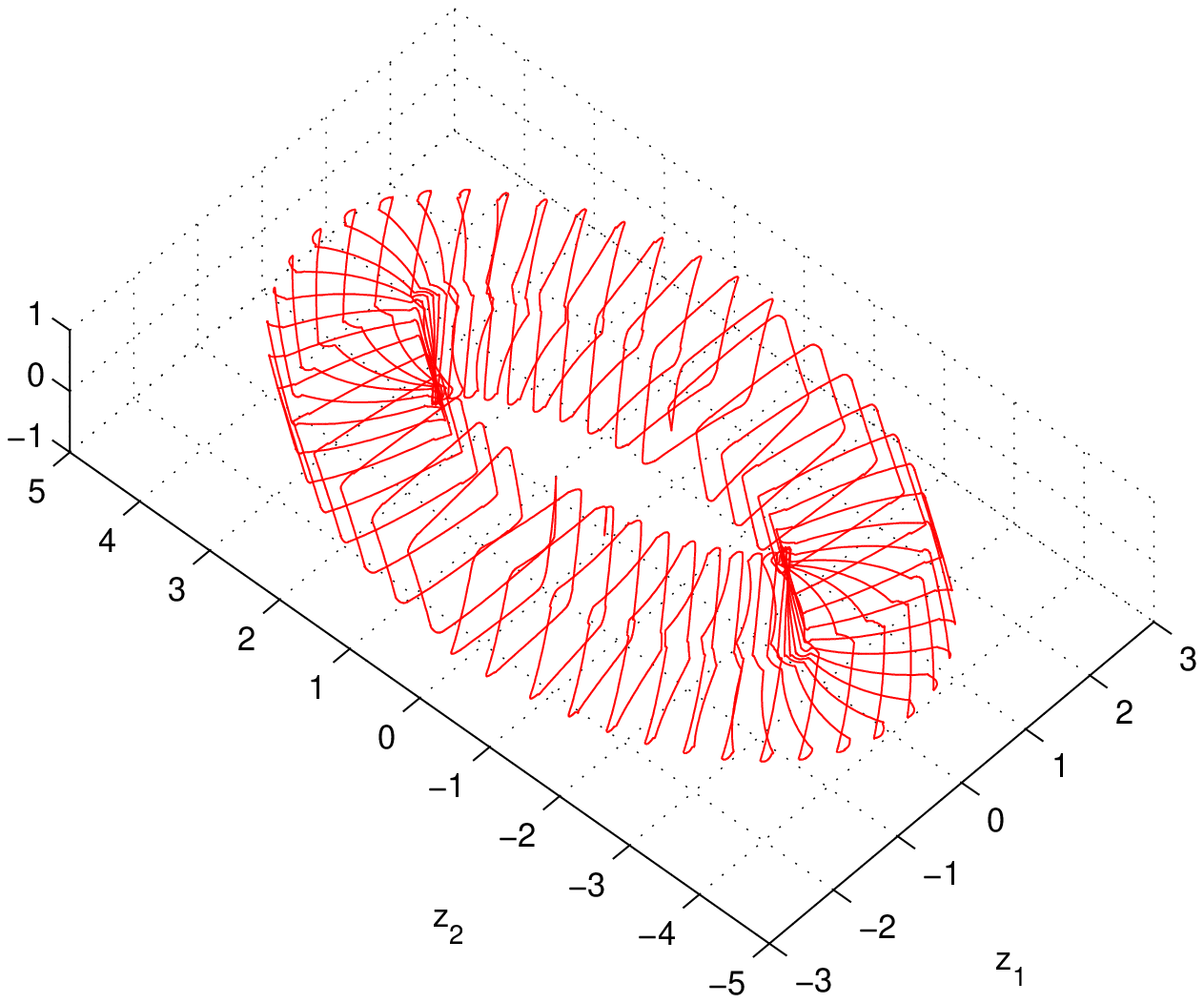}
\hspace{+.12in}
\includegraphics[scale=0.36]{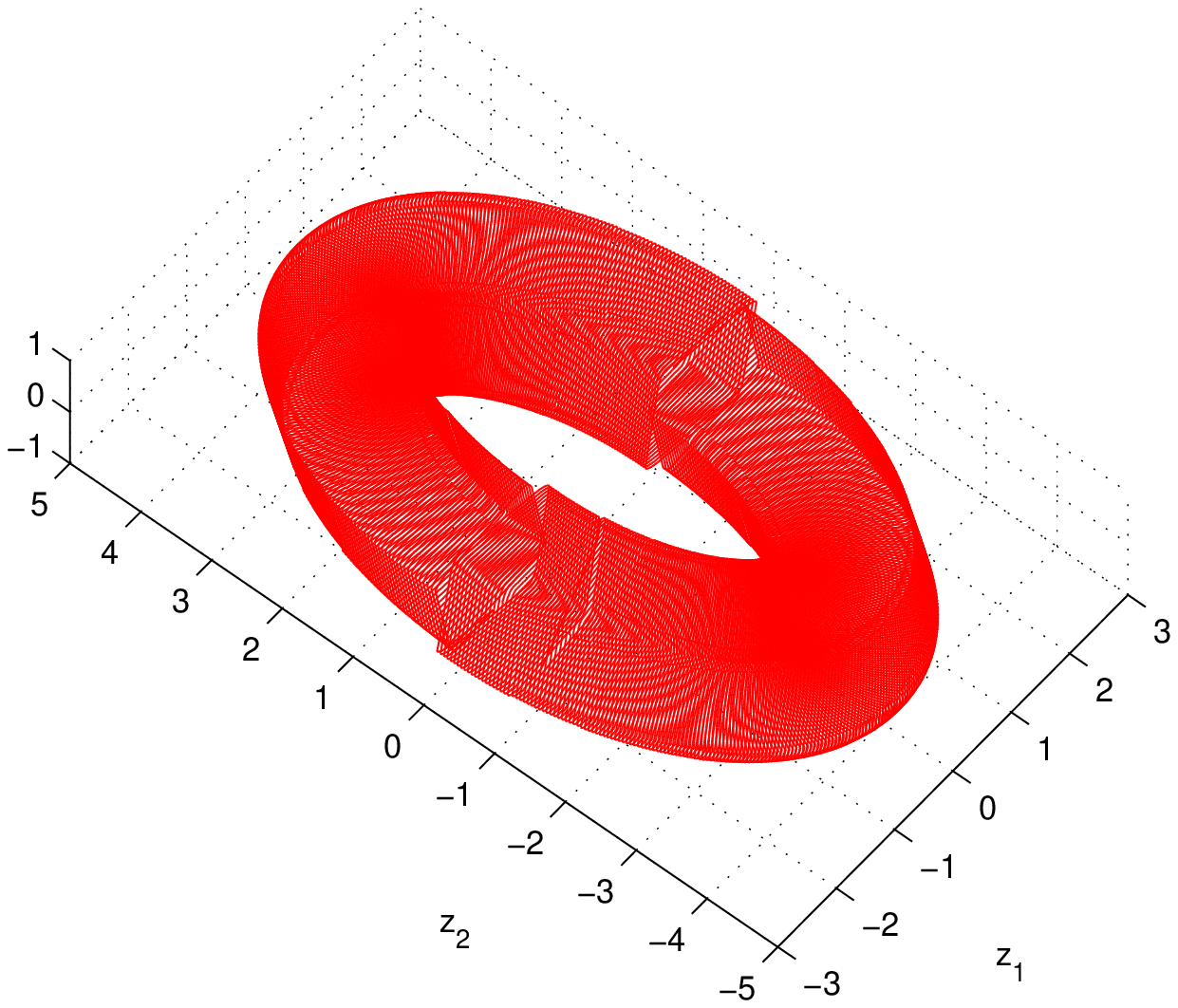}
\end{center}
 \vspace{-.1in}
\ \ \ \ \ \ \ \ \ \ \ \ \ \ \ \ \ \   \ \ \ \ Fig. 9:  $\ (y_{ \epsilon}(t), z_{\epsilon}(t))\  $  for $\epsilon = 0.01$ \ \ \ \ \ \ \  \ \ \ \ \ \
 Fig.10: \ $\ (y_{ \epsilon}(t), z_{\epsilon}(t))$ for $\epsilon = 0.001$

\section{LP based  algorithm for solving $(N,M) $-approximating problems}\label{Sec-LP-based-algorithm}

We will start with a consideration of an algorithm for finding an optimal solution of the following \lq\lq generic" semi-infinite LP problem
\begin{equation}\label{e:LP-Decomp-alg-1}
\min_{p\in \Omega}\{\int_{X}\Psi_0(x)p(dx)\}\BYDEF \sigma^*
\end{equation}
\vspace{-.1in}
where
\vspace{-.1in}
\begin{equation}\label{e:LP-Decomp-alg-2}
 \Omega\BYDEF \{p\in \mathcal{P}(X) \ : \ \int_{X}\Psi_i(x)p(dx)=0, \ \ i=1,...,K\},
\end{equation}
with $X$ being a non-empty compact metric space and  with $\Psi_i(\cdot): X\rightarrow \R^1, \ i=0,1,...,K, $ being continuous functional on $X $.
Note that the problem dual with respect to (\ref{e:LP-Decomp-alg-1}) is the problem (\ref{e:Important-simple-duality-Lemma-1}), and we assume
that the inequality (\ref{e:Important-simple-duality-Lemma-3}) is valid only with $v_i = 0 , \ i=1,...,K $ (which, by Lemma \ref{Lemma-Important-simple-duality}, ensures the existence of a solution of the problem (\ref{e:Important-simple-duality-Lemma-1})).

It is known (see, e.g., Theorems A.4 and A.5 in \cite{Rubio}) that among the optimal solutions of the problem (\ref{e:LP-Decomp-alg-1}) there exists
one that is presented in the form
$$
 p^{*}= \sum^{K+1}_{l=1}p_l^{*} \delta_{x_l^*}, \ \  \ \ \ \ \ {\rm where} \ \ \ \ \ \ p_l^{*}\geq 0, \ \ \ \ \sum^{K+1}_{l=1}p_l^{*}=1,
$$
where $\delta_{x_l^*}$ are Dirac measures concentrated at $\ x_l^*\in X$, $\ l=1,...,K+1$. Having in mind this presentation, let us consider the following algorithm for finding optimal concentration points $\{x_l^*\} $ and optimal weights $\{p_l^*\}$ (see \cite{GM} and \cite{GRT}).
Let points $ \ \{x_l\in X \ , \ l = 1,..., L\} \ $ ($L\geq K+1$) be chosen to define
an initial grid $\mathcal{X}_0$ on $X$
$$ \mathcal{X}_0= \{x_l\in X \ , \ l = 1,..., L\}. \ $$
At every iteration a new point is defined and added to this set.
Assume that after $J $ iterations the points $x_{L+1},...,x_{L+J}$ have been defined and the set $\mathcal{X}_J $ has been constructed:
$$
\mathcal{X}_J =\{x_l\in X \ , \ l = 1,..., L+J\}.
$$
The iteration $J+1$ ($J=0,1,... $) is described as follows:

(i) Find a basic optimal solution $\ p^J = \{p^J_l\} \ $  of the LP problem
\begin{equation}\label{e:NS-3}
\min_{p\in\Omega_J}\{\sum_{l=1}^{L+J}p_l\Psi_0(x_l)\} \BYDEF \sigma^J ,
\end{equation}
where
\begin{equation}\label{e:NS-4}
\Omega_J\BYDEF \{ p     \ : \ p= \{p_l\} \geq 0\ ,  \ \ \ \sum_{l=1}^{L+J}p_l = 1,\ \ \ \ \ \ \sum_{l=1}^{L+J}p_l\Psi_i(x_l)= 0 \ ,\ \ i = 1,...,K \}.
\end{equation}
Note that no more than $K+1 $ components of $ p^J $ are positive, these being called basic components.
Also, find an optimal solution  $\ \{\lambda_0^J,\ \lambda_i^J , i=1,...,K\}$ of the problem dual with respect to (\ref{e:NS-3}), the latter being
of the form
\begin{equation}\label{e:NS-4-dual}
 \max\{\lambda_0 \ : \  \lambda_0\leq \Psi_0(x_l)+ \sum_{i=1}^{K}\lambda_i \Psi_i(x_l)\ \ \forall \ l=1,...,K+J\}.
\end{equation}

(ii) Find an optimal solution $x_{L+J+1} $ of the problem
\begin{equation}\label{e:NS-5}
\min_{x\in X} \{\Psi_0(x) + \sum_{i=1}^K\lambda_i^J \Psi_i(x)\}\BYDEF a^J
\end{equation}

(iii) Define the set $\mathcal{X}_{J+1} $ by the equation
$$
\mathcal{X}_{J+1} =\mathcal{X}_{J}\cup \{x_{L+J+1} \}.
$$

Using an argument similar to one commonly used in a standard  (finite dimensional) linear programming,  one can show (see, e.g., \cite{GM} or \cite{GRT}) that if
$\ a^J \geq \lambda_0^J $, then $\ \sigma^J = \sigma^*$ and  the measure $\ \sum_{l\in \varrho^J}p_l^J\delta_{x_l} $ (where $\varrho^J $
stands for the index set of basic components of $p^J$) is an optimal solution of the problem (\ref{e:LP-Decomp-alg-1}), with $\lambda^J\BYDEF \{\lambda_i^J, \ i=1,...,K\}$ being an optimal
solution of the problem (\ref{e:Important-simple-duality-Lemma-1}). If $\ a^J < \lambda_0^J $, for $J=1,2,..., $ then, under some non-degeneracy assumptions,
it can be shown that $\ \lim_{J\rightarrow\infty}\sigma^J = \sigma^* $, and that any cluster (limit) point of the set of measures
$\{ \sum_{l\in \varrho^J}p_l^J\delta_{x_l} , \ J=1,2,...\}  $ is an optimal solution of the problem (\ref{e:LP-Decomp-alg-1}), while
any cluster (limit) point of the set $\ \{ \lambda^J, \ J=1,2,...\} $ is  an optimal solution of the problem (\ref{e:Important-simple-duality-Lemma-1})
 (main features of the proof of these convergence results can be found  in \cite{GM}, \cite{GRT}).

 The $(N,M)$ approximating problem (\ref{e-ave-LP-opt-di-MN}) is a special case of the problem (\ref{e:LP-Decomp-alg-1}) with an obvious correspondence
 between the notations:
 $$
 x=(\mu,z), \ \ \ \ \ \  X= F_M, \ \ \  \ \ \ \Omega = \tilde{\mathcal{D}}_{di}^{N,M}(z_0), \ \ \ \ \ \ K=N,
 $$
 \vspace{-.2in}
 $$
\Phi_0(x) = \tilde G(\mu , z), \ \ \ \ \ \ \ \  \ \ \ \Phi_i(x) = \nabla\psi_i(z)^T \tilde{g}(\mu,z) + C (\psi_i (z_0) - \psi_i (z) ), \ \ i=1,...,N.
 $$
 Assume that the set
 \begin{equation}\label{e:NS-5-SP-0-101}
\mathcal{X}_J =\{(\mu_l, z_l)\in F_M \ , \ l = 1,..., L+J\}
\end{equation}
 has been constructed. The LP problem (\ref{e:NS-3}) takes in this case the form
 \begin{equation}\label{e:NS-3-ave}
\min_{p\in\Omega_J}\{\sum_{l=1}^{L+J}p_l\tilde G(\mu_l , z_l)\}\BYDEF \tilde{G}^{N,M,J}_{di}(z_0)  ,
\end{equation}
where
$$
\Omega_J\BYDEF \{ p     \ : \ p= \{p_l\} \geq 0\ ,  \ \ \ \sum_{l=1}^{L+J}p_l = 1,\ \ \ \ \ \ \sum_{l=1}^{L+J}p_l [\nabla\psi_i(z)^T \tilde{g}(\mu_l,z_l) + C (\psi_i (z_0) - \psi_i (z_l) )]= 0 \ ,\ \ i = 1,...,N \},
$$
  with the corresponding dual being of the form
   \begin{equation}\label{e:NS-3-ave-dual}
 \max\{\lambda_0 \ : \  \lambda_0\leq \tilde G(\mu_l , z_l)+ \sum_{i=1}^{N}\lambda_i [\nabla\psi_i(z)^T \tilde{g}(\mu_l,z_l) + C (\psi_i (z_0) - \psi_i (z_l) )]\ \ \forall \ l=1,...,K+J\}.
\end{equation}
  Denote by $p^{N,M,J}=\{p^{N,M,J}_l\} $ an optimal basic solution of the problem (\ref{e:NS-3-ave})
 and by $\ \{\lambda_0^{N,M,J},\ \lambda_i^{N,M,J} , i=1,...,N\}$ an optimal solution of the  dual problem (\ref{e:NS-3-ave-dual}).
  The problem (\ref{e:NS-5}) identifying the point to be added to the set
 $\mathcal{X}_J$ takes the following form
\begin{equation}\label{e:NS-5-SP}
\min_{(\mu,z)\in F_M} \{\tilde G(\mu , z) + \sum_{i=1}^N\lambda_i^{N,M,J} [\nabla\psi_i(z)^T \tilde{g}(\mu,z) + C (\psi_i (z_0) - \psi_i (z) ) ]\}
=\min_{z\in Z}\{\tilde{\mathcal{G}}^{N,M,J}(z) + C (\psi_i (z_0) - \psi_i (z) )\},
\end{equation}
where
\vspace{-.2in}
\begin{equation}\label{e:NS-5-SP-1}
\tilde{\mathcal{G}}^{N,M,J}(z)\BYDEF \min_{\mu\in W_M(z)}\{\int_{U\times Y}[G(u,y,z) + \sum_{i=1}^N\lambda_i^{N,M,J} \nabla\psi_i(z)^Tg(u,y,z)]\mu(du,dy)\} .
\end{equation}
Note that the problem (\ref{e:NS-5-SP-1}) is also a special case of the problem (\ref{e:LP-Decomp-alg-1}) with
$$
  x= (u,y), \ \ \ \ \ \  X=U\times Y, \ \ \ \ \ \ \Omega= W_M(z), \ \ \ \ \ \ K=M,
$$
 \vspace{-.2in}
$$
\Phi_0(x) =  G(u,y,z) + \sum_{i=1}^N\lambda_i^{N,M,J} \nabla\psi_i(z)^Tg(u,y,z), \ \ \ \ \ \ \ \  \ \ \ \Phi_i(x) = \nabla\phi_i(y)^T f(u,y,z), \ \ i=1,...,M.
$$
 Its optimal solution as well as an optimal solution of the corresponding dual problem can be found with the help of the same approach. Denote the latter
 as $\mu^{N,M,J}_z$ and $\{\omega_{z,0}^{N,M,J}, \omega_{z,i}^{N,M,J}, \ i=1,...,M\} $, respectively. By adding the point $(\mu^{N,M,J}_{z^*}, z^*)$
 to the set $\mathcal{X}_J$ ($z^*$ being an optimal solution of the problem in the right-hand side of
(\ref{e:NS-5-SP})),  one can define the set $\mathcal{X}_{J+1} $ and then proceed to the next iteration.

Under the controllability conditions introduced in Section \ref{Sec-Existence-Controllability} (see Assumptions \ref{Ass-ave-disc-controllability} and \ref{Ass-associated-local-controllability}) and under additional (simplex method related) non-degeneracy  conditions,
it can be proved (although we do not do it in the present paper) that
the optimal value of the problem (\ref{e:NS-3-ave}) converges to the optimal value of the $(N,M)$-approximating averaged problem
\begin{equation}\label{e:NS-3-ave-convergence}
\lim_{J\rightarrow\infty} \tilde{G}^{N,M,J}_{di}(z_0)  = \tilde{G}^{N,M}_{di}(z_0)
\end{equation}
and  that, if $\ \lambda^{N,M}= \{\lambda_i^{N,M} , i=1,...,N\}$ is a cluster (limit) point of the set of optimal solutions
$\ \lambda^{N,M,J}= \{\lambda_i^{N,M,J} , i=1,...,N\}\ $ of the problem (\ref{e:NS-3-ave-dual}) considered with $J=1,2,... $, then
$$
\zeta^{N,M}(z) \BYDEF \sum_{i=1}^N \lambda_i^{N,M} \psi_i(z)
$$
is an optimal solution of the $(N,M) $-approximating averaged dual problem (\ref{e:DUAL-AVE-0-approx-MN}). In addition to this, it can be shown that,
if $\ \omega^{N,M}_z= \{\omega_{z,i}^{N,M} , i=1,...,M\}$ is a cluster (limit) point of the set of optimal solutions
$\ \omega^{N,M,J}_z= \{\omega_{z,i}^{N,M,J} , i=1,...,M \}\ $ of the problem
dual to (\ref{e:NS-5-SP-1}) considered with $J=1,2,... $, then
$$
\eta^{N,M}_z \BYDEF \sum_{i=1}^M \omega_{z,i}^{N,M} \phi_i(y)
$$
is an optimal solution  of the $(N,M) $-approximating associated dual problem (\ref{e:dec-fast-4-Associated}).

 A software that implements this algorithm on the basis of the IBM ILOG CPLEX LP solver  and   global nonlinear optimization routines  designed by A. Bagirov and M. Mammadov
has been developed (with the CPLEX solver being used for finding optimal solutions of the LP problems  involved and Bagirov's and Mammadov's routines being used for  finding optimizers in (\ref{e:NS-5-SP}) and in problems similar to (\ref{e:NS-5}) that arise
when solving (\ref{e:NS-5-SP-1})). The numerical solutions of  Examples 1 and 2 in  Section \ref{Sec-construction-SP-examples} were obtained with the help of
this software ($C$ was taken to be equal to zero in dealing with the periodic optimization problem of Example 2)

REMARK IV.3. The decomposition of the problem (\ref{e:NS-5}), an optimal solution of which identifies  the  point to be added to the set $\mathcal{X}_J$,
into problems (\ref{e:NS-5-SP}) and (\ref{e:NS-5-SP-1})
 resembles  the column generating technique of generalized linear programming
 (see \cite{Dantzig}).  Note that a similar decomposition was observed in dealing with LP problems related to singular perturbed Markov chains (see, e.g., \cite{AFH}, \cite{PG88} and \cite{Yin}). Finally, let us also note that, while in this paper we are using the $(N,M)$-approximating problems and their LP based solutions
 for finding near optimal ACG families, other methods for finding the latter can be applicable as well. For example, due to the fact that
 the averaged and associated dual problems (\ref{e:DUAL-AVE-0}) and  (\ref{e:dec-fast-4}) are inequality forms of certain HJB equations (see Remark III.1),
 it is plausible  to expect that an adaptation of methods of solution of HJB equations  developed in \cite{Falcone}, \cite{Fleming}, \cite{Kus3-D} can be of a particular use.

\bigskip

{\bf V. Selected proofs.}\label{Proofs for SP-LP}

\section{Proofs of Proposition \ref{Prop-ave-disc} and Proposition \ref{Prop-SP-2} }\label{Sec-Phi-Map}\

{\it Proof of Proposition \ref{Prop-ave-disc}.}
Let $h_i(u,y,z):U\times
Y \times Z \to \mathbb{R}^1 \ i=1,2,...$, be a sequence of Lipschitz continuous functions that is dense
in the unit ball of $C(U\times Y\times Z)$ and let
\begin{equation}\label{e:h-N}
 h^N(u,y,z) = (h_1(u,y,z), ..., h_N(u,y,z)), \ \ N=1,2,... \ ,
 \end{equation}
 \begin{equation}\label{e:h-N-tilde}
\tilde h^N(\mu,z) = (\tilde h_1(\mu,z), ..., \tilde h_N(\mu,z)), \ \ N=1,2,... \ ,
 \end{equation}
 where $\ \tilde h_i(\mu,z)=\int_{U\times Y}h_i(u,y,z)\mu(du,dy) $ with $\mu\in \mathcal{P}(U\times Y) $. From the fact
 that the averaged system approximates the SP system on finite time intervals it follows (see (\ref{e:intro-0-3-8}))
 that

 \begin{equation}\label{e:SP-aug-5-1}
d_H(\Theta_N(\epsilon , T), \Theta_N( T))\leq \beta_N(\epsilon, T) \ , \ \ \ {\rm where} \ \ \ \lim_{\epsilon\rightarrow 0} \beta_N(\epsilon, T) = 0 \ ,
 \end{equation}
$d_H(\cdot,\cdot)$ is the Housdorff metric generated by a norm in $\R^N $, and the sets $\Theta_N(\epsilon , T) $, $ \Theta_N( T) $
 are defined by the equations
 \begin{equation}\label{e:SP-aug-5-1-0-1}
 \Theta_N(\epsilon , T)\BYDEF \bigcup_{u(\cdot)}\{ C \int_0^{T}  e^{-C t} h^N(u(t),y_{\epsilon}(t),z_{\epsilon}(t) dt  \} \ ,
 \end{equation}
 \begin{equation}\label{e:SP-aug-5-1-0-2}
\Theta_N(T)\BYDEF \bigcup_{(\mu(\cdot), z(\cdot))}\{C \int_0^{T}  e^{-C t}\tilde{h}^N(\mu(t),z(t)) dt\},
 \end{equation}
 with the first union being over all  controls of the SP system and the second being over all admissible pairs of the averaged system.
Define the sets $\Theta_N(\epsilon)$ and $\Theta_N $ by the equations
 \begin{equation}\label{e:SP-aug-5-1-3}
 \Theta_N(\epsilon)\BYDEF \bigcup_{u(\cdot)}\{ C \int_0^{\infty}  e^{-C t} h^N(u(t),y_{\epsilon}(t),z_{\epsilon}(t))) dt  \} \ ,
 \end{equation}
 \begin{equation}\label{e:SP-aug-5-1-0-4}
\Theta_N\BYDEF \bigcup_{(\mu(\cdot), z(\cdot))}\{C \int_0^{\infty}  e^{-C t}\tilde{h}^N(\mu(t),z(t)) dt\},
 \end{equation}
 where, again, the first union is over all  controls of the SP system and the second is over all admissible pairs of the averaged system.
  It is easy to see that
 \begin{equation}\label{e:SP-aug-5-1-0-5}
d_H(\Theta_N(\epsilon , T), \Theta_N(\epsilon) )
\leq a_Ne^{-C T}  , \ \ \ \ \ \ \ \ \ \
d_H(\Theta_N(T), \Theta_N)\leq  a_N e^{-C T} ,
 \end{equation}
  where $\ a_N\BYDEF \max_{(u,y.z)\in U\times Y\times Z}||h^N(u,y,z)||$.

  Let us  use (\ref{e:SP-aug-5-1}) and (\ref{e:SP-aug-5-1-0-5}) to show that
 \begin{equation}\label{e:SP-aug-5-1-0-6}
d_H( \Theta_N(\epsilon),  \Theta_N )
\leq \hat \beta_N(\epsilon)
 \end{equation}
 for some $\hat \beta_N(\epsilon) $ such that $ \  \lim_{\epsilon\rightarrow 0} \hat \beta_N(\epsilon)= 0  $.
 Let us separately deal with two
different cases. First is the case when the estimate (\ref{e:SP-aug-5-1}) is uniform, that is,
\begin{equation}\label{e:SP-aug-5-1-1}
 \beta_N(\epsilon, T)\leq \beta_N(\epsilon) \ , \ \ \ \ \ \lim_{\epsilon\rightarrow 0} \beta_N(\epsilon) = 0 \ .
 \end{equation}
The second is the case when there exists a number $\alpha >0 $ and sequences $\{\epsilon_l\}$,
$\{T_l\}$ such that
\begin{equation}\label{e:SP-aug-5-1-2}
\beta_N(\epsilon_l, T_l) \geq \alpha \  , \ \ \ \ \ \lim_{l\rightarrow \infty} \epsilon_l = 0 \ ,
\ \ \ \ \ \lim_{l\rightarrow \infty}T_l = \infty \ .
 \end{equation}
 In case (\ref{e:SP-aug-5-1-1}) is valid, from (\ref{e:SP-aug-5-1}) and (\ref{e:SP-aug-5-1-0-5}) it follows that
$$
d_H( \Theta_N(\epsilon),  \Theta_N )\leq d_H( \Theta_N(\epsilon),  \Theta_N(\epsilon , T)) +
d_H(\Theta_N(\epsilon , T), \Theta_N (T)) + d_H(\Theta_N (T), \Theta_N )
$$
$$
\leq \ 2 a_N e^{-C T} + \beta_N(\epsilon) \ \ \ \ \forall T\in [0,\infty) \ .
$$
Hence, passing to the limit when $T\rightarrow\infty$, one obtains (\ref{e:SP-aug-5-1-0-6}) with $\bar \beta_N(\epsilon) = \beta_N(\epsilon) $.

To deal with the case when (\ref{e:SP-aug-5-1-2}) is true, choose $T(\epsilon)$ in such a way that
\begin{equation}\label{e:SP-aug-5-1-5-1-1}
\lim_{\epsilon\rightarrow 0}T(\epsilon) = \infty \ , \ \ \ \ \ \lim_{\epsilon\rightarrow 0}\beta_N(\epsilon , T(\epsilon)) = 0 ;
\end{equation}
see Lemma \ref{Auxiliary-Lemma} below.
Using (\ref{e:SP-aug-5-1}) and (\ref{e:SP-aug-5-1-0-5}), one can obtain that
$$
d_H( \Theta_N(\epsilon),  \Theta_N )\leq d_H( \Theta_N(\epsilon),  \Theta_N(\epsilon , T(\epsilon))) +
d_H(\Theta_N(\epsilon , T(\epsilon)), \Theta_N (T(\epsilon))) + d_H(\Theta_N (T(\epsilon)), \Theta_N )
$$
$$
\leq \ 2 a_N e^{-\alpha T(\epsilon)} + \beta_N(\epsilon , T(\epsilon))  \ .
$$
Denoting $ \ 2 a_N e^{-\alpha T(\epsilon)} + \beta_N(\epsilon , T(\epsilon)) \BYDEF  \bar \beta_N(\epsilon) \ $, one has
 $ \  \lim_{\epsilon\rightarrow 0} \bar \beta_N(\epsilon) = 0 \ $ (due to (\ref{e:SP-aug-5-1-5-1-1})). This
proves the validity of (\ref{e:SP-aug-5-1-0-6}).

From (\ref{e:oms-0-2}) it follows that the set $\Theta_N(\epsilon) $ can be rewritten in the form
it follows that
 \begin{equation}\label{e:Theta-comparison-1}
\Theta_N(\epsilon) = \bigcup_{\gamma\in \Gamma_{di}(\epsilon , y_0,z_0)}\int_{U\times Y\times Z} h^N(u,y,z)\gamma(du,dy,dz) \ .
\end{equation}
Also, from (\ref{e:oms-0-2}) and from the definition of the map $\Phi(\cdot) $ (see (\ref{e:h&th-1})) it follows that
\begin{equation}\label{e:Theta-comparison-2}
\Theta_N = \bigcup_{p\in \tilde \Gamma_{di}(z_0)} \int_{F} \tilde{h}^N(\mu,z)p(d\mu , dz) =
 \bigcup_{\gamma\in \Phi(\tilde \Gamma_{di}(z_0))}\int_{U\times Y\times Z} h^N(u,y,z)\gamma(du,dy,dz) \ .
\end{equation}
Having in mind the representations (\ref{e:Theta-comparison-1}) and (\ref{e:Theta-comparison-2})
and using Corollary 3.6 of \cite{Gai-Ng},
one can come to the conclusion that
the validity of (\ref{e:SP-aug-5-1-0-6}) for any $N=1,2,...$, implies (\ref{e-OccupSet-Convergence-Dis}).

The validity of (\ref{e-Objective-Convergence-Dis}) follows from (\ref{e-OccupSet-Convergence-Dis}) (due to the presentations (\ref{e:oms-4}), (\ref{e:oms-4-di-ave}) and the definition of the map $\Phi(\cdot)$).

\endproof

\begin{Lemma}\label{Auxiliary-Lemma} If (\ref{e:SP-aug-5-1-2}) is valid, then there exists a monotone decreasing function $T(\epsilon)$ defined
on an interval $(0,c)$
 (c is some positive number) such that (\ref{e:SP-aug-5-1-5-1-1}) is valid.
\end{Lemma}

{\it Proof of Lemma \ref{Auxiliary-Lemma} }. Let us assume (without loss of generality) that $\beta_N(\epsilon, T)$ is
decreasing if $\epsilon$ is decreasing (with fixed $T$) and is increasing if  $T$ is increasing (with fixed $\epsilon $). Let us
define the sequence $\{\bar{\epsilon}_k \}$ by the equations
$$
\bar{\epsilon}_k = sup_{\epsilon\in [0,1]}\{\epsilon \ : \ \beta_N(\epsilon , k)\leq \frac{1}{2^{k}}\ \} \ , \ \ k=1,2,... \ .
$$
Note that, due to monotonicity of  $\beta_N(\epsilon, T)$ in $\epsilon$,
$$
\beta_N(\epsilon , k)\leq \frac{1}{2^{k}} \ \ \ \forall \epsilon\in (0,\bar{\epsilon}_k) \ .
$$
It is easy to verify (using the fact that $\beta_N(\epsilon, T)$ is  increasing in $T$) that $\bar{\epsilon}_1\geq \bar{\epsilon}_2\geq ... \bar{\epsilon}_k \geq ... $, and, hence, there exists a limit
$$
\lim_{k\rightarrow\infty}\bar{\epsilon}_k\BYDEF \bar{\epsilon}\geq 0 \ .
$$
Let us show that $\bar{\epsilon}=0 $. Assume it is not true and $\bar{\epsilon}>0 $. Then, for any $\epsilon\in [0,\bar{\epsilon})$ and for any fixed $j\leq k $,
$$\beta_N(\epsilon , j)\leq \beta_N(\epsilon , k)\leq   \frac{1}{2^{k}} \ . $$
By letting $k$ go to infinity in the last inequality, one comes to  the conclusion that $\beta_N(\epsilon , j)= 0 $, $j=1,2,...$ and, consequently,
to the conclusion that  $\beta_N(\epsilon , T)= 0 $ for any $T>0$ (due to monotonicity in $T$). The latter contradicts (\ref{e:SP-aug-5-1-2}). Thus,
\begin{equation}\label{e:SP-aug-5-1-3-7}
\lim_{k\rightarrow\infty}\bar{\epsilon}_k= 0 \ .
\end{equation}
Let $k_1<k_2<...k_l<...$ be a sequence of natural numbers such that $k_l\rightarrow\infty$ and such that
$\bar{\epsilon}_{k_1}> \bar{\epsilon}_{k_2}>...> \bar{\epsilon}_{k_l}> ... $ .
Define the function $T(\epsilon)$ on the interval $(0,\bar{\epsilon}_{k_1})$ by the equation
\begin{equation}\label{e:SP-aug-5-1-4}
T(\epsilon)= k_l \ \ for \ \ \epsilon \in [\bar{\epsilon}_{k_{l+1}}, \bar{\epsilon}_{k_l}) \ , \ \ l=1,2,... \ .
\end{equation}
It is easy to see that the function  $T(\epsilon)$ is increasing when $\epsilon$ is decreasing, and
\begin{equation}\label{e:SP-aug-5-1-5}
\lim_{\epsilon\rightarrow 0}T(\epsilon) = \infty \ .
\end{equation}
Also, according to the construction above,
\begin{equation}\label{e:SP-aug-5-1-6}
\beta_N(\epsilon, T(\epsilon)) =   \beta_N(\epsilon, k_l)  \leq \frac{1}{2^{k_l}} \ \ \ \forall \epsilon\in [\bar{\epsilon}_{k_{l+1}},\bar{\epsilon}_{k_l})
\end{equation}
$$
\Rightarrow \ \ \ \ \lim_{\epsilon\rightarrow 0}\beta_N(\epsilon , T(\epsilon)) = 0 \ .
$$
 \qquad
\endproof

{\it Proof of  Proposition \ref{Prop-SP-2}.}
To prove (\ref{equality-1-SP-new}),
let us first prove that the inclusion
\begin{equation}\label{incl-dang-1}
\Phi(\tilde{\mathcal{D}}_{di}(z_0))\subset \mathcal{D}_{di}^{\mathcal{A}}(z_0)
\end{equation}
is valid. Take an arbitrary $ \gamma\in \Phi(\tilde{\mathcal{D}}_{di}(z_0))$. That is, $\
\gamma = \Phi(p) \ $ for some $ \ p\in  \tilde{\mathcal{D}}_{di}(z_0) \ $. By (\ref{e:h&th-1}),
$$
\int_{{U\! \times Y \times Z}}[ \ \psi(z)\nabla \phi(y)^T f(u,y,z)\ ]\Phi(p)(du,dy,dz)
$$
$$
= \int_{F} [ \ \psi(z)\int_{U\times Y}\nabla \phi(y)^T f(u,y,z)\mu(du,dy)\ ]
p(d\mu,dz).
$$
By definition of $F$ (see (\ref{e:graph-w})),
 $$
\int_{U\times Y}\nabla \phi(y)^T f(u,y,z)\mu(du,dy) = 0 \ \ \ \ \forall (\mu, z)\in F .
$$
Consequently (see (\ref{e:SP-W-3})),
\begin{equation}\label{e: SP-Prop-2-1}
\int_{{U\! \times Y \times Z}}[ \ \psi(z)\nabla \phi(y)^T f(u,y,z)\ ]\Phi(p)(du,dy,dz)=0\ \ \ \Rightarrow
\ \ \ \Phi(p)\in \mathcal{D}.
\end{equation}
Also from (\ref{e:h&th-1}) and from the fact that $ \ p\in  \tilde{\mathcal{D}}_{di}(z_0) \ $ it follows that
$$
\int_{U\times Y\times Z}[\ \nabla \phi(z)^T g(u,y,z)\ + \ C (\ \psi(z_0)-\psi(z) \ )\ ]\Phi(p)(du,dy,dz)
$$
$$
= \int_{F} [\ \nabla \phi(z)^T \tilde g(\mu ,z)\ + \ C (\ \psi(z_0)-\psi(z) \ )\ ]
p(d\mu , dz) \ = \ 0 \ \ \ \Rightarrow \ \ \ \Phi(p)\in \mathcal{A}_{di}(z_0).
$$
Thus, $\gamma = \Phi(p)\subset \mathcal{D}\cap\mathcal{A}_{di}(z_0) $. This proves (\ref{incl-dang-1}).
Let us now show that the converse inclusion
\begin{equation}\label{incl-dang-2}
\Phi(\tilde{\mathcal{D}}_{di}(z_0))\supset \mathcal{D}_{di}^{\mathcal{A}}(z_0)
\end{equation}
is valid.
To this end, take  $\ \gamma\in  \mathcal{D}_{di}^{\mathcal{A}}(z_0) $  and show that $\gamma\in \Phi (\tilde{\mathcal{D}}_{di}(z_0))$.
Due to (\ref{tilde-W-1}), $ \gamma $ can be presented in the form (\ref{e:SP-W-4-extra}) with
 $\mu(du,dy|z)\in W(z) $ for $\nu$ almost all $z\in Z$.
Changing values of $\mu $
on a subset of $Z$ having the $\nu$ measure  $0$, one can come to the conclusion
that  $\ \gamma$ can be presented in the form (\ref{e:SP-W-4-extra}) with
\begin{equation}\label{e:SP-disintegration-1}
\mu(du,dy|z)\in W(z) \ \ \ \ \ \forall z\in Z.
\end{equation}
Let $\mathcal{L} $ be a subspace of $C[F ]$ defined by the equation
\begin{equation}\label{e:SP-L-subspace-1}
\mathcal{L}\BYDEF \{\tilde q (\cdot,\cdot) \ : \ \tilde q (\mu,z)=\int_{U\times Y}q(u,y,z)\mu(du,dy) , \ \ q\in C[U\times Y\times Z] \}.
\end{equation}
For every $\tilde q \in \mathcal{L} $, let $p_{\mathcal{L}}(\tilde q): \mathcal{L} \rightarrow \R^1 $ be defined by the equation
\begin{equation}\label{e:SP-L-subspace-2}
p_{\mathcal{L}}(\tilde q )\BYDEF \int_{z\in Z}\tilde q(\mu(\cdot |z),z) \nu(dz) = \int_{z\in Z}[\int_{U\times Y}q(u,y,z) \mu(du,dy|z)\ ]\nu(dz)
\end{equation}
$$ = \int_{U\times Y}q(u,y,z) \gamma(du,dy,dz).$$
Note that $p_{\mathcal{L}} $ is a positive linear functional on $\mathcal{L} $. That is, if $\tilde q_1(\mu,z)\leq \tilde q_2(\mu,z) \ \forall (\mu , z)\in F \ $, then $p_{\mathcal{L}}(\tilde q_1) \leq p_{\mathcal{L}}(\tilde q_2) $.
 Note also that  $1\in \mathcal{L} $. Hence, by Kantorovich theorem (see, e.g.,
\cite{Aliprantis}, p. 330), $p_{\mathcal{L}} $ can be extended
to a positive linear functional $p $ on the whole $C[F ]$, with
\begin{equation}\label{e:SP-L-subspace-3}
p(\tilde q)=p_{\mathcal{L}}(\tilde q ) \ \ \forall \tilde q \in \mathcal{L} .
\end{equation}
 Due to the fact that $p $ is positive, one obtains that
 \begin{equation}\label{e:SP-L-subspace-4}
sup_{\zeta(\cdot)\in \bar B }\ p(\zeta(\cdot))\leq sup_{\zeta(\cdot)\in \bar B}\ p(|\zeta(\cdot)|)\leq p(1) =1,
\end{equation}
where $\bar B$ is the closed unit ball in $C[F ]$ (that is,
$ \ \bar B\BYDEF \{\zeta(\cdot)\in  C[F ] \ : \ \max_{(\mu , z)\in F}
|\zeta(\mu,z)|\leq 1\}$).
Thus, $p\in C^*[F ] $, and, moreover, $||p||=p(1)=1 $. This implies that there exists a unique probability measure
$p(d\mu , dz)\in {\cal P} (F) $ such that, for any $\zeta(\mu,z)\in C[F ] $,
\begin{equation}\label{e:SP-L-subspace-6}
p(\zeta) = \int_{F}\zeta(\mu,z)p(d\mu , dz)
\end{equation}
(see, e.g., Theorem 5.8  on page 38 in \cite{Parth}).
 Using this relationship for $\zeta(\mu,z)=\tilde q (\mu, z)\in \mathcal{L} $, one obtains (see (\ref{e:SP-L-subspace-2}) and (\ref{e:SP-L-subspace-3})) that
\begin{equation}\label{e:SP-L-subspace-7}
\int_{F}\tilde q(\mu,z)p(d\mu , dz) = \int_{U\times Y}q(u,y,z) \gamma(du,dy,dz).
\end{equation}
Since the latter is valid for any $ q(u,y,z)\in C[U\times Y\times Z]$, it follows  that
\begin{equation}\label{e:SP-L-subspace-8}
\gamma = \Phi (p).
\end{equation}
Considering  now (\ref{e:SP-L-subspace-7})  with $q(u,y,z)=\nabla \psi (z)^T g(u,y,z) + C(\psi (z_0) - \psi (z))$,
and taking into account that, in this case, $\tilde q (\mu,z)=  \nabla \psi (z)^T \tilde g(\mu,z) + C(\psi (z_0) - \psi (z))$ one obtains that
$$
\int_{F} [\nabla \psi (z)^T \tilde g(\mu,z) + C(\psi (z_0) - \psi (z))]  p(d\mu , dz)
$$
$$
=\int_{U\times Y\times Z} [\nabla \psi (z)^T g(u,y,z) + C(\psi (z_0) - \psi (z))]\gamma (du,dy,dz)=0 ,
$$
where the equality to zero follows from the fact that $\gamma\in \mathcal{A}_{di}(z_0) $ (see (\ref{e:SP-M})). This implies
that $p\in \tilde{\mathcal{D}}_{di}(z_0)$. Hence, by (\ref{e:SP-L-subspace-8}), $\gamma\in \Phi (\tilde{\mathcal{D}}_{di}(z_0))$.
 This proves (\ref{equality-1-SP-new}).

 The validity of
(\ref{e-ave-LP-opt-1}) as well as the fact that $\gamma = \Phi(p) $ is  optimal in (\ref{SP-IDLP-0}) if and only if
$p$ is  optimal on (\ref{e-ave-LP-opt-di}) follow from (\ref{equality-1-SP-new}) and the definition of the map $\Phi (\cdot) $ (see (\ref{e:h&th-1})).
\endproof


\section{Proof of Theorem \ref{Main-SP-Nemeric}}\label{Sec-Main-Main}\
Note, first of all, that there exists an optimal solution $p^{N,M}$ of the problem (\ref{e-ave-LP-opt-di-MN}) which is
presented as a convex combination of (no more than $N+1$)
 Dirac measures (see, e.g., Theorems A.4 and A.5 in \cite{Rubio}).
  That is,
  \begin{equation}\label{e-NM-minimizer-proof-2}
p^{N,M}= \sum^{K^{N,M}}_{k=1}p_k^{N,M} \delta_{(\mu_k^{N,M},z_k^{N,M})},
\end{equation}
  where $\delta_{(\mu_k^{N,M},z_k^{N,M})}$
is the Dirac measure concentrated at $(\mu_k^{N,M},z_k^{N,M})$ and
 \begin{equation}\label{e-NM-minimizer-proof-1}
 (\mu_k^{N,M},z_k^{N,M})\in F_M, \ \ \  \  p_k^{N,M}>0, \ \ k=1,...,K^{N,M}\leq N+1;\ \ \ \ \   \   \sum^{K^{N,M}}_{k=1}p_k^{N,M}=1 .
\end{equation}

\begin{Lemma}\label{opt-1}
For any $ k=1,...,K^{N,M} $,
\begin{equation}\label{e-NM-minimizer-proof-2-2}
\mu_k^{N,M} = {\rm argmin}_{\mu\in W_M(z_k^{N,M})}\{\tilde G(\mu, z_k^{N,M})  + \nabla \zeta^{N,M} (z_k^{N,M})^T  \tilde g (\mu , z_k^{N,M}) \}.
\end{equation}
That is, $\mu_k^{N,M}$ is a minimizer of the problem
\begin{equation}\label{e-NM-minimizer-proof-2-3}
\min_{\mu\in W_M(z_k^{N,M})}\{ \tilde G(\mu, z_k^{N,M}) + \nabla \zeta^{N,M} (z_k^{N,M})^T  \tilde g (\mu , z_k^{N,M})\ \}.
\end{equation}
\end{Lemma}
\begin{proof}
From (\ref{e:DUAL-AVE-0-approx-MN}) and (\ref{e:DUAL-AVE-0-approx-2-1}) it follows that
\begin{equation}\label{e:DUAL-AVE-0-approx-2-1-Lemma-opt-1}
\tilde{G}^{N,M}_{di}(z_0)= \min_{(\mu ,z) \in F_M } \{\tilde G(\mu , z)+ \nabla \zeta^{N,M}(z)^T \tilde g (\mu , z) + C (\zeta^{N,M}(z_0) - \zeta^{N,M}(z))\ \}.
\end{equation}
Also, for any $p\in \tilde{\mathcal{D}}_{di}^{N,M}(z_0)$,
$$
\int_{F_M}\tilde {G}(\mu, z)p(d\mu, dz)
$$
\vspace{-.2in}
$$
= \int_{F_M}[\tilde G(\mu , z)+ \nabla \zeta^{N,M}(z)^T \tilde g (\mu , z) + C (\zeta^{N,M}(z_0) - \zeta^{N,M}(z))]p(d\mu, dz).
$$
Consequently, for $p=p^{N,M} $,
$$
\tilde{G}^{N,M}_{di}(z_0) = \int_{F_M}\tilde {G}(\mu, z)p^{N,M}(d\mu, dz)
$$
\vspace{-.2in}
$$
= \int_{F_M}[\tilde G(\mu , z)+ \nabla \zeta^{N,M}(z)^T \tilde g (\mu , z) + C (\zeta^{N,M}(z_0) - \zeta^{N,M}(z))]p^{N,M}(d\mu, dz)
$$
\vspace{-.2in}
$$
=\sum^{K^{N,M}}_{k=1}p_k^{N,M} [\tilde G(\mu_k^{N,M} , z_k^{M,N})+ \nabla \zeta^{N,M}(z_k^{N,M})^T \tilde g (\mu_k^{N,M} , z_k^{N,M}) +
 C (\zeta^{N,M}(z_0) - \zeta^{N,M}(z_k^{N,M}))].
$$
Since $(\mu_k^{N,M} , z_k^{N,M})\in F_M $, from the equalities above and from (\ref{e:DUAL-AVE-0-approx-2-1-Lemma-opt-1}) it follows
that
$$
\tilde G(\mu_k^{N,M} , z_k^{N,M})+ \nabla \zeta^{N,M}(z_k^{N,M})^T \tilde g (\mu_k^{N,M} , z_k^{N,M}) + C (\zeta^{N,M}(z_0) - \zeta^{N,M}(z_k^{N,M}))
$$
$$
=\min_{(\mu ,z) \in F_M } \{\tilde G(\mu , z)+ \nabla \zeta^{N,M}(z)^T \tilde g (\mu , z) + C (\zeta^{N,M}(z_0) - \zeta^{N,M}(z))\ \},  \ \ k=1,...,  K^{N,M}.
$$
That is, for $k=1,...,  K^{N,M} $,
$$
(\mu_k^{M,N} , z_k^{M,N})= {\rm argmin}_{(\mu ,z) \in F_M } \{\tilde G(\mu , z)+ \nabla \zeta^{N,M}(z)^T \tilde g (\mu , z) + C (\zeta^{N,M}(z_0) - \zeta^{N,M}(z))\ \}.
$$
The latter imply (\ref{e-NM-minimizer-proof-2-2}).
\end{proof}

\begin{Lemma}\label{opt-2}
In the presentation (\ref{e-NM-minimizer-proof-2}) of an optimal solution $p^{N,M} $ of the problem (\ref{e-ave-LP-opt-di-MN}), $\mu_k^{N,M} $
can be chosen as follows:
\begin{equation}\label{e-NM-minimizer-proof-3}
\mu_k^{N,M}= \sum_{j=1}^{J^{N,M,k}} q_j^{N,M,k}\delta_{(u_j^{N,M,k},y_j^{N,M,k})}, \ \ \ \ k= 1,...,K^{N,M},
\end{equation}
where
\begin{equation}\label{e-NM-minimizer-proof-4}
 \ q_j^{N,M,k}> 0, \ \ j=1,..., J^{N,M,k}, \ \ \ \  \sum_{j=1}^{J^{N,M,k}} q_j^{N,M,k}=1,
\end{equation}
and
\begin{equation}\label{e-NM-minimizer-proof-3-1}
J^{N,M,k}\leq N+M+2 .
\end{equation}
In (\ref{e-NM-minimizer-proof-3}), $\delta_{(u_j^{N,M,k},y_j^{N,M,k})}\in \mathcal{P}(U\times Y)$  are the Dirac measures  concentrated at $\ (u_j^{N,M,k},y_j^{N,M,k})\in U\times Y, \ \ j=1,..., J^{N,M,k}$, with
$$
u_j^{N,M,k}= argmin_{u\in U}\{G(u,y_j^{N,M,k},z_k^{N,M})+\nabla \zeta^{N,M} (z_k^{N,M})^T  g (u,y_j^{N,M,k} , z_k^{N,M})
$$
\begin{equation}\label{e-NM-minimizer-proof-5}
 + \nabla \eta^{N,M} (y_j^{N,M,k})^T  f(u,y_j^{N,M,k},z_k^{N,M})\}.
\end{equation}
\end{Lemma}
\begin{proof}
Assume that $p_k^{N,M}, \  k=1,...,  K^{N,M}, $ in (\ref{e-NM-minimizer-proof-2}) are fixed. Then
 $\mu_k^{N,M}, \  k=1,...,  K^{N,M},$ form an optimal solution of the following problem
\begin{equation}\label{e-NM-minimizer-proof-6}
 \min_{\{\mu_k\}}\{\sum_{k=1}^{K^{N,M}}p_k^{N,M}\int_{U\times Y}G(u,y,z_k^{N,M})\mu_k(du,dy)\ \},
\end{equation}
where minimization is over $\mu_k \in \mathcal{P}(U\times Y), \ k=1,...,  K^{N,M}, $ that satisfy the following constraints
\begin{equation}\label{e-NM-minimizer-proof-7}
\sum_{k=1}^{K^{N,M}}p_k^{N,M}\int_{U\times Y}[(\nabla\psi_i(z_k^{N,M})^T g(u,y,z_k^{N,M}) + C (\psi_i (z_0) - \psi_i (z_k^{N,M}) )]\mu_k(du,dy) = 0, \ \ \ i=1,...,N,
\end{equation}
\begin{equation}\label{e-NM-minimizer-proof-8}
\int_{U\times Y}\nabla\phi_j(y)f(u,y,z_k^{N,M})\mu_k(du,dy) = 0,  \ \ \ j=1,...,M, \ \ \ \ \  k=1,...,  K^{N,M}.
\end{equation}
In fact, if $ \mu_k^{N,M}, \ \ k=1,...,  K^{N,M}$ is an optimal solution of the problem (\ref{e-NM-minimizer-proof-6})-(\ref{e-NM-minimizer-proof-8}),
then $\bar p^{N,M}=\sum^{K^{N,M}}_{k=1}p_k^{N,M} \delta_{( \mu_k^{N,M},z_k^{N,M})}  $ will be an optimal solution of the problem (\ref{e-ave-LP-opt-di-MN}).
Let us show that the former has an optimal solution that can be presented as the sum in the right-hand side of (\ref{e-NM-minimizer-proof-3}). To this end,
note that the problem (\ref{e-NM-minimizer-proof-6})-(\ref{e-NM-minimizer-proof-8})  can be rewritten in the following equivalent form
\begin{equation}\label{e-NM-minimizer-proof-9}
\min_{\{w_0^k, w_i^k, v_j^k\}}\{\sum_{k=1}^{K^{N,M}}p_k^{N,M}w_0^k\},
\end{equation}
where minimization is over $w_0^k, w_i^k, v_j^k$, $i=1,...,N, \  j=1,...,M, \ k=1,...,K^{N,M}, $ such that
\begin{equation}\label{e-NM-minimizer-proof-10}
\sum_{k=1}^{K^{N,M}}p_k^{N,M}w_i^k = 0, \ \ \  i=1,...,N,
\end{equation}
\begin{equation}\label{e-NM-minimizer-proof-11}
v_j^k = 0, \ \ \ j=1,...,M, \ \ \ \ \  k=1,...,  K^{N,M}
\end{equation}
and  such that
\begin{equation}\label{e-NM-minimizer-proof-12}
\{w_0^k, w_i^k, v_j^k,\ \  i=1,...,N, \ \ j=1,...,M, \  \}\in \bar{co}V_k, \ \ \ k=1,...,  K^{N,M}
\end{equation}
where
$$
 V_k = \{w_0^k, w_i^k, v_j^k\ : \ w_0^k= G(u,y,z_k^{N,M}), \ \ w_i^k = (\nabla\psi_i(z_k^{N,M})^T g(u,y,z_k^{N,M}) + C (\psi_i (z_0) - \psi_i (z_k^{N,M}) ),
$$
\vspace{-.2in}
$$
v_j^k= \nabla\phi_j(y)f(u,y,z_k^{N,M}), \ \ i=1,...,N, \  j=1,...,M;  \ (u,y)\in U\times Y\}.
$$
By Caratheodory's theorem,
$$
\bar{co}V_k = \cup_{\{q_l\}}\{q_1V_k + .... + q_{N+M+2}V_k\ \},
$$
where the union is taken over all $ q_l\geq 0 , \ l=1,..., N+M+2, $ such that $\sum_{l=1}^{N+M+2}q_l = 1 $.
Thus, an optimal solution of the problem (\ref{e-NM-minimizer-proof-10})-(\ref{e-NM-minimizer-proof-12}) can be presented in the form
$$
\bar w_0^k = \sum_{l=1}^{N+M+2}\bar q_l^k G(u_l^k,y_l^k,z_k^{N,M}), \ \ \ \bar w_i^k =\sum_{l=1}^{N+M+2}\bar q_l^k [\nabla\psi_i(z_k^{N,M})^T g(u_l^k,y_l^k,z_k^{N,M})
$$
\vspace{-.2in}
$$
+ C (\psi_i (z_0) - \psi_i (z_k^{N,M}) ], \ \ \ \bar v_j^k =\sum_{l=1}^{N+M+2}\bar q_l^k  \nabla\phi_j(y_l^k)f(u_l^k,y_l^k,z_k^{N,M}), \ \ \ \ \ i=1,...,N, \  j=1,...,M.
$$
The latter implies that there exists an optimal solution of the problem (\ref{e-NM-minimizer-proof-6})-(\ref{e-NM-minimizer-proof-8}) that is presentable
in the form (\ref{e-NM-minimizer-proof-3}).

Let us now show that the relationships (\ref{e-NM-minimizer-proof-5}) are valid.
Note, firstly, that from (\ref{e:dec-fast-4-Associated}) and (\ref{e:DUAL-AVE-0-approx-1-Associate-2}) it follows that
\begin{equation}\label{e:DUAL-AVE-0-approx-1-Associate-2-NM}
\sigma_{N,M}^*(z)=\min_{(u,y)\in U\times Y} \{ G(u,y,z)+\nabla \zeta^{N,M} (z)^T  g (u,y , z) + \nabla \eta^{N,M} (y)^T  f(u,y,z)\}.
\end{equation}
By Lemma \ref{opt-1}, $\mu_k^{N,M}$ is an optimal solution of the problem (\ref{e-NM-minimizer-proof-2-3}). That is,
$$
\int_{ U\times Y}[G(u,y, z_k^{N,M}) + \nabla \zeta^{N,M} (z_k^{N,M})^T  g (u,y , z_k^{N,M})]\mu_k^{N,M}(du,dy)
$$
\vspace{-.2in}
$$
= \min_{\mu\in W_M(z_k^{N,M})}\int_{U\times Y} [G(u,y,z_k^{N,M})+\nabla \zeta^{N,M} (z_k^{N,M})^T  g (u,y , z_k^{N,M})]\mu(du,dy) = \sigma_{N,M}^*(z_k^{N,M}),
$$
the latter equality being due to the duality relationships between the problem (\ref{e:dec-fast-4-Associated}) and (\ref{e:dec-fast-4-Associated-1}).
Since $\mu_k^{N,M}\in W_M(z_k^{N,M}) $,
$$
\int_{U\times Y} [G(u,y,z_k^{N,M})+\nabla \zeta^{N,M} (z_k^{N,M})^T  g (u,y , z_k^{N,M})]\mu_k^{N,M}(du,dy)
$$
\vspace{-.2in}
$$
= \int_{U\times Y}[G(u,y,z_k^{N,M})+\nabla \zeta^{N,M} (z_k^{N,M})^T  g (u,y , z_k^{N,M}) +   \nabla \eta^{N,M} (y)^T  f(u,y,z_k^{N,M})]\mu_k^{N,M}(du,dy).
$$
Consequently,
$$
\int_{U\times Y}[G(u,y,z_k^{N,M})+\nabla \zeta^{N,M} (z_k^{N,M})^T  g (u,y , z_k^{N,M}) +   \nabla \eta^{N,M} (y)^T  f(u,y,z_k^{N,M})]\mu_k^{N,M}(du,dy)
$$
\vspace{-.2in}
$$
 =\sigma_{N,M}^*(z_k^{N,M}).
$$
After the  substitution of (\ref{e-NM-minimizer-proof-3}) into the equality above and taking into account (\ref{e-NM-minimizer-proof-4}), one can obtain
$$
\sum_{j=1}^{J^{N,M,k}} q_j^{N,M,k}[G(u_j^{N,M,k},y_j^{N,M,k},z_k^{N,M})+\nabla \zeta^{N,M} (z_k^{N,M})^T  g (u_j^{N,M,k},y_j^{N,M,k} , z_k^{N,M})
$$
\vspace{-.2in}
\begin{equation}\label{e:DUAL-AVE-0-approx-1-Associate-2-NM-1}
 +   \nabla \eta^{N,M} (y_j^{N,M,k})^T  f(_j^{N,M,k},y_j^{N,M,k},z_k^{N,M}) - \sigma_{N,M}^*(z_k^{N,M})]= 0.
\end{equation}
By (\ref{e:DUAL-AVE-0-approx-1-Associate-2-NM}), from (\ref{e:DUAL-AVE-0-approx-1-Associate-2-NM-1}) it follows that
$$
G(u_j^{N,M,k},y_j^{N,M,k},z_k^{N,M})+\nabla \zeta^{N,M} (z_k^{N,M})^T  g (u_j^{N,M,k},y_j^{N,M,k} , z_k^{N,M})
$$
\vspace{-.2in}
$$
+   \nabla \eta^{N,M} (y_j^{N,M,k})^T  f(u_j^{N,M,k},y_j^{N,M,k},z_k^{N,M})= \sigma_{N,M}^*(z_k^{N,M}) \ \ \ \ \forall j = 1, ..., J^{M,N,k}.
$$
Also by (\ref{e:DUAL-AVE-0-approx-1-Associate-2-NM}), the latter implies
$$
(u_j^{N,M,k},y_j^{N,M,k})= {\rm argmin}_{(u,y)\in U\times Y}\{G(u,y,z_k^{N,M})+\nabla \zeta^{N,M} (z_k^{N,M})^T  g (u,y , z_k^{N,M})
$$
\vspace{-.2in}
\begin{equation}\label{e:DUAL-AVE-0-approx-1-Associate-2-NM-2}
 +  \ \nabla \eta^{N,M} (y)^T  f(u,y,z_k^{N,M})\},
\end{equation}
which, in turn, implies (\ref{e-NM-minimizer-proof-5}).
\end{proof}

\begin{Lemma}\label{convergence-important}
Let Assumption \ref{SET-1} be satisfied. Then, for any $t\in [0,\infty) $, there exists a sequence
\begin{equation}\label{e:convergence-important-0}
(\mu_{k^{N,M}}^{N,M},  z_{k^{N,M}}^{N,M})\in  \{(\mu_k^{N,M}, z_k^{N,M}),\ k=1,...,K^{N,M} \}, \ \ N=1,2,..., \ \ M=1,2,...,
\end{equation}
(with $\{(\mu_k^{N,M}, z_k^{N,M}),\ k=1,...,K^{N,M} \} $ being the set of  concentration points of the
Dirac measures  in  (\ref{e-NM-minimizer-proof-2}))
such that
\begin{equation}\label{e:convergence-important-1}
 \lim_{N\rightarrow \infty}\limsup_{M\rightarrow \infty} [\rho(\mu^*(t),\mu_{k^{N,M}}^{N,M}) + ||z^*(t)-z_{k^{N,M}}^{N,M}||]=0.
\end{equation}

Also, if  $(\mu_{k^{N,M}}^{N,M},  z_{k^{N,M}}^{N,M}) $ is as in (\ref{e:convergence-important-1}),
  then for any $\tau\in [0,\infty) $, there exists a sequence
    \begin{equation}\label{e:convergence-important-2-2}
\ (u^{N,M,k_{N,M}}_{j_{N,M}}, y^{N,M,k_{N,M}}_{j_{N,M}})\in \{(u_j^{N,M,k^{N,M}},y_j^{N,M,k^{N,M}}),\ j=1,...,J^{N,M,k^{N,M}} \}, \ \ N=1,2,..., \ \ M=1,2,...,
\end{equation}
 ($\{(u_j^{N,M,k^{N,M}},y_j^{N,M,k^{N,M}}),\ j=1,...,J^{M,N,k^{N,M}} \}$ being the set of
  concentration points of the Dirac measures in (\ref{e-NM-minimizer-proof-3}) taken with $k=k_{N,M} $)
  such that
  \begin{equation}\label{e:convergence-important-2}
\lim_{N\rightarrow \infty}\limsup_{M\rightarrow \infty}[ ||u^*_t(\tau)-u^{N,M,k_{N,M}}_{j_{N,M}}|| +
||y^*_t(\tau))- y^{N,M,k_{N,M}}_{j_{N,M}}||] = 0. \end{equation}

\end{Lemma}
\begin{proof}
Note first of all that from the fact that the optimal solution $p^*$ of the  IDLP problem (\ref{e-ave-LP-opt-di}) is unique and  the equality
(\ref{C-result-di-ave}) is valid (see Assumption \ref{SET-1}(i)) it follows that $p^*$ is the discounted occupational measure generated by
$(\mu^*(\cdot), z^*(\cdot)) $.

Let
$$
\Theta^*\BYDEF \{(\mu,z)\ : \ (\mu,z)=(\mu^*(t), z^*(t)) \ \ \ {\rm for\ some}\ \ \ t\in [0,\infty)\}
$$
and let $cl \Theta^* $ stand for the closure of $\Theta^* $ in $\mathcal{P}(U\times Y)\times \R^n $. It is easy to show that
from Assumption \ref{SET-1}(ii) it follows that, if $(\bar \mu, \bar z)\in cl \Theta^* $, then,
for any $r>0 $, the open set $\ \mathcal{B}_r(\bar \mu, \bar z) \BYDEF \{(\mu,z) \ : \ \rho(\mu,\bar \mu) + ||z-\bar z||< r \}$
has a nonzero $p^* $-measure. That is,
\begin{equation}\label{e:convergence-important-3}
p^*(\mathcal{B}_r(\bar \mu, \bar z))> 0.
\end{equation}
In fact, if $(\bar \mu, \bar z)\in cl \Theta^* $, then there exists a sequence $t_l, \ l=1,2,..., $ such that
$\ (\bar \mu, \bar z) =\lim_{l\rightarrow\infty}(\mu^*(t_l), z^*(t_l))$, with $(\mu^*(t_l), z^*(t_l))\in \mathcal{B}_r(\bar \mu, \bar z) $ for $l$ large
enough. Taking one such $l$, one may conclude that there exists $\alpha > 0 $ such that
$(\mu^*(t'), z^*(t'))\in \mathcal{B}_r(\bar \mu, \bar z)\ \ \forall t'\in (t_l-\alpha, t_l ] $ if $\mu^*(\cdot) $ is continuous from the left at $t_l$
and $(\mu^*(t'), z^*(t'))\in \mathcal{B}_r(\bar \mu, \bar z)\ \ \forall t'\in [t_l, t_l+\alpha )$  if $\mu^*(\cdot) $ is continuous from the right at $t_l$.
By the definition of the discounted occupational measure generated by $(\mu^*(\cdot), z^*(\cdot))$ (see (\ref{e:occup-C})),
this implies (\ref{e:convergence-important-3}).

Assume now that (\ref{e:convergence-important-1}) is not true. Then there exists a number $r>0$ and sequences $(\mu_i , z_i)\in \Theta^* $, $N_i $, $M_{i,j} $
($ i=1,2,..., $ and $j=1,2,... $) such that
$$
\lim_{i\rightarrow\infty}(\mu_i , z_i) = (\bar \mu , \bar z ), \ \ \ \ \ \ \lim_{i\rightarrow\infty}N_i = \infty,
\ \ \ \ \ \ \lim_{j\rightarrow\infty}M_{i,j}= \infty
$$
and such that
\begin{equation}\label{e:convergence-important-4}
dist ((\mu_i , z_i), \Theta^{N_i, M_{i,j}})\geq 2r ,
\end{equation}
where $\Theta^{N, M}$ is the set of the concentration points of the
Dirac measures  in  (\ref{e-NM-minimizer-proof-2}), that is,
$$
\Theta^{N, M}\BYDEF \{(\mu_k^{N,M}, z_k^{N,M}), k=1,...,K^{N,M} \},
$$
taken with $N=N_i $ and $M=M_{i,j} $, and where
$$
dist((\mu,z), \Theta^{N, M})\BYDEF \min_{(\mu',z')\in \Theta^{N, M}}\{\rho(\mu,\mu') + ||z-z'|| \}.
$$
From (\ref{e:convergence-important-4}) it follows that
\begin{equation}\label{e:convergence-important-4-1}
 \ dist ((\bar \mu , \bar z ), \Theta^{N_i, M_{i,j}})\geq r
\end{equation}
for $i\geq i_0 $ (large enough). Hence,
$$
(\mu_k^{N_i,M_{i,j}}, z_k^{N_i,M_{i,j}}) \notin \mathcal{B}_r(\bar \mu, \bar z), \ \ k=1,...,K^{N_i,M_{i,j}}, \ \ i\geq i_0,  \ \ j= 1,2,... \ .
$$
The latter implies that
\begin{equation}\label{e:convergence-important-5}
 p^{N_i, M_{i,j}} (\mathcal{B}_r(\bar \mu, \bar z)) = 0,\ \ i\geq i_0,  \ \ j= 1,2,... \ ,
\end{equation}
where $p^{N,N}$ is defined by (\ref{e-NM-minimizer-proof-2}). Due to the fact that the optimal solution $p^*$ of the  IDLP problem (\ref{e-ave-LP-opt-di})
 is unique (Assumption \ref{SET-1}(i)), the relationship  (\ref{e:graph-w-M-102-101}) is valid. Consequently,
\begin{equation}\label{e:convergence-important-6}
\lim_{i\rightarrow\infty }\limsup_{j\rightarrow\infty}\rho( p^{N_i, M_{i,j}},p^*) = 0.
\end{equation}
From (\ref{e:convergence-important-5}) and (\ref{e:convergence-important-6}) it follows that
$$
p^* (\mathcal{B}_r(\bar \mu, \bar z))\leq \lim_{i\rightarrow\infty }\limsup_{j\rightarrow\infty}p^{N_i, M_{i,j}}(\mathcal{B}_r(\bar \mu, \bar z))=0.
$$
The latter contradicts to (\ref{e:convergence-important-3}). Thus, (\ref{e:convergence-important-1}) is proved.

Let us now prove the validity of (\ref{e:convergence-important-2}). 
Assume it is is not valid.
Then there exists  $r>0$ and sequences $N_i $ , $M_{i,j} $ with
$i=1,..., \ j=1,...,\ $ and with
$\ \lim_{i\rightarrow\infty}N_i = \infty,
\ \ \ \lim_{j\rightarrow\infty}M_{i,j}= \infty  $ such that
\begin{equation}\label{e:convergence-important-4-(u,y)}
dist ((u^*_t(\tau) , y^*_t(\tau)), \theta^{N_i, M_{i,j}})\geq r , \ \ \ \ i, j= 1,2,... \ ,
\end{equation}
where $\theta^{N, M} $ is  the set of the concentration points of the
Dirac measures  in  (\ref{e-NM-minimizer-proof-3}),
$$
\theta^{N, M}\BYDEF \{(u_j^{N,M,k^{N,M}},y_j^{N,M,k^{N,M}}),\ \ j=1,...,J^{N,M,k^{N,M}} \},
$$
taken with $k=k^{N,M} $ and with $N=N_i $, $M=M_{i,j} $, and where
$$
dist((u,y), \theta^{N, M})\BYDEF \min_{(u',y')\in \theta^{N, M}}\{||u-u'|| + ||y-y'|| \}.
$$
From (\ref{e:convergence-important-4-(u,y)}) it follows that
$$
(u_j^{N_i,M_{i,j},k^{N_i,M_{i,j}}}, y_j^{N_i,M_{i,j},k^{N_i,M_{i,j}}}) \notin B_r(u^*_t(\tau),y^*_t(\tau)), \ \ j=1,...,J^{N_i,M_{i,j},k^{N_i,M_{i,j}}}, \ \ \ \ \ \ \ i, j= 1,2,... \ .
$$
The latter implies that
\begin{equation}\label{e:convergence-important-5-(u,y)}
\mu_{k^{N_i,M_{i,j}}}^{N_i, M_{i,j}} (B_r(u^*_t(\tau),y^*_t(\tau)) = 0, \ \ \ \ \ \ \ i, j= 1,2,... \ .
\end{equation}
where $\mu_k^{N,M}$ is defined by (\ref{e-NM-minimizer-proof-3}) (taken with $k=k^{N_i,M_{i,j}}, \ N=N_i $ and $M=M_{i,j} $).

From (\ref{e:convergence-important-1}) it follows, in particular, that
\begin{equation}\label{e:convergence-important-1-(u,y)}
 \lim_{N\rightarrow \infty}\limsup_{M\rightarrow \infty} \rho(\mu^*(t),\mu_{k^{N,M}}^{N,M})=0 \ \ \ \ \ \Rightarrow \ \ \ \ \
 \lim_{i\rightarrow \infty}\limsup_{j\rightarrow \infty}\rho(\mu^*(t),\mu_{k^{N_i,M_{i,j}}}^{N_i, M_{i,j}} )=0.
\end{equation}
The later and (\ref{e:convergence-important-5-(u,y)}) lead to
$$
\mu^*(t)(B_r(u^*_t(\tau),y^*_t(\tau))\leq \lim_{i\rightarrow \infty}\limsup_{j\rightarrow \infty}\mu_{k^{N_i,M_{i,j}}}^{N_i, M_{i,j}}(B_r(u^*_t(\tau),y^*_t(\tau))) = 0,
$$
which contradicts to (\ref{e:convergence-important-7}). Thus (\ref{e:convergence-important-2}) is proved.

\end{proof}

{\it Proof of Lemma \ref{fast-convergence}.}
Let $\tau\in [0,\infty) $ be such that the ball $B_{t,\tau}$ is not empty.
Let  also $(\mu_{k^{N,M}}^{N,M},  z_{k^{N,M}}^{N,M}) $ be as in (\ref{e:convergence-important-1}) and $(u^{N,M,k_{N,M}}_{j_{N,M}}, y^{N,M,k_{N,M}}_{j_{N,M}}) $
be as in (\ref{e:convergence-important-2}).
 Note that, due to (\ref{e-NM-minimizer-proof-5}),
\begin{equation}\label{e:convergence-important-11}
u^{N,M,k_{N,M}}_{j_{N,M}}= u(y^{N,M,k_{N,M}}_{j_{N,M}}, z_{k^{N,M}}^{N,M}),
\end{equation}
where $u(y,z) $ is as in (\ref{e-NM-minimizer-1}).  From (\ref{e:convergence-important-1}) and (\ref{e:convergence-important-2}) it follows that
$ \ (z_{k^{N,M}}^{N,M}, y^{N,M,k_{N,M}}_{j_{N,M}})\in Q_t\times B_{t,\tau} $ for $N $ and $M$ large enough. Hence, one can use (\ref{e:HJB-1-17-0-per-Lip-u-NM})
to obtain
$$
||u^*_t(\tau) - u^{N,M}(y^*_t(\tau), z^*(t))||\leq ||u^*_t(\tau)- u^{N,M,k_{N,M}}_{j_{N,M}}  || + ||u(y^{N,M,k_{N,M}}_{j_{N,M}}, z_{k^{N,M}}^{N,M}) -
u^{N,M}(y^*_t(\tau), z^*(t))||
$$
\vspace{-.2in}
$$
\leq ||u^*_t(\tau)- u^{N,M,k_{N,M}}_{j_{N,M}}  || + L(||y^*_t(\tau)-y^{N,M,k_{N,M}}_{j_{N,M}}|| + ||z^*(t)- z_{k^{N,M}}^{N,M}||).
$$
By (\ref{e:convergence-important-1}) and (\ref{e:convergence-important-2}), the latter implies
\begin{equation}\label{e:convergence-important-12}
 \lim_{N\rightarrow \infty}\limsup_{M\rightarrow \infty}||u^*_t(\tau) - u^{N,M}(y^*_t(\tau), z^*(t))|| = 0.
 \end{equation}
Since $B_{t,\tau}$  is not empty for almost all $\tau\in [0,\infty) $ (Assumption \ref{SET-3} (i)), the convergence (\ref{e:convergence-important-12})
takes place for almost all $\tau\in [0,\infty) $.

Let us take an arbitrary $\tau\in [0,\infty) $ and subtract the equation
 \begin{equation}\label{e:L-1-4}
y^*_t(\tau)=y^*_t(0) + \int_0^{\tau} f(u_t^*(\tau'), y^*_t(\tau'), z^*(t)) d\tau'
\end{equation}
from the equation
  \begin{equation}\label{e:L-1-5}
 y^{N,M}_{t} (\tau) =y^*_t(0) + \int_0^{\tau} f(u^{N,M}(y^{N,M}_{t} (\tau'),z^*(t)),  y^{N,M}_{t} (\tau'), z^*(t)) d\tau'.
\end{equation}
We will obtain
$$
|| y_{t}^{N,M}(\tau)-y^*_t(\tau)||\leq \int_0^{\tau} ||f(u^{N,M}( y^{N,M}_{t} (\tau'),z^*(t)),  y_{t}^{N,M}(\tau'), z^*(t))
$$
\vspace{-.2in}
$$
-  f(u_t^*(\tau'), y^*_t(\tau'), z^*(t))|| d\tau'\leq \int_0^{\tau} ||f(u^{N,M}( y^{N,M}_{t} (\tau'),z^*(t)),  y_{t}^{N,M}(\tau'), z^*(t))
$$
\vspace{-.2in}
$$
-f(u^{N,M}(y^*_t(\tau'),z^*(t)), y^*_t(\tau'), z^*(t))|| d\tau'
$$
\vspace{-.2in}
 \begin{equation}\label{e:L-1-6}
 +  \int_0^{\tau} || f(u^{N,M}(y^*_t(\tau'),z^*(t)), y^*_t(\tau'), z^*(t))- f(u_t^*(\tau'), y^*_t(\tau'), z^*(t))|| d\tau'.
\end{equation}
 Using Assumption \ref{SET-3} (ii),(iii), one can derive that
$$
\int_0^{\tau} ||f(u^{N,M}( y^{N,M}_{t} (\tau'),z^*(t)),  y_{t}^{N,M}(\tau'), z^*(t))
-  f(u^{N,M}(y^*_t(\tau'),z^*(t)), y^*_t(\tau'), z^*(t))|| d\tau'
$$
\vspace{-.2in}
$$
\leq \int_{\tau'\notin P_{t,\tau}(N,M)} ||f(u^{N,M}( y^{N,M}_{t} (\tau'),z^*(t)),  y_{t}^{N,M}(\tau'), z^*(t))
-  f(u^{N,M}(y^*_t(\tau'),z^*(t)), y^*_t(\tau'), z^*(t))|| d\tau'
$$
\vspace{-.2in}
$$
+ \int_{\tau'\in P_{t,\tau}(N,M)}[||f(u^{N,M}( y^{N,M}_{t} (\tau'),z^*(t)),  y_{z^*(t)}^{N,M}(\tau'), z^*(t))||
$$
\vspace{-.2in}
$$
+ ||f(u^{N,M}(y^*_t(\tau'),z^*(t)), y^*_t(\tau'), z^*(t))||] d\tau'
$$
\vspace{-.2in}
\begin{equation}\label{e:L-1-6-MN}
\leq L_1 \int_0^{\tau}|| y^{N,M}_{t} (\tau') - y^*_t(\tau')|| d\tau' + L_2 meas \{P_{t,\tau}(N,M)\},
\end{equation}
where $L_1$ is a constant defined (in an obvious way) by Lipschitz constants
of $f(\cdot)$ and $u^{N,M}(\cdot) $, and $\ L_2\BYDEF 2\ max_{(u,y,z)\in U\times Y\times Z}\{||f(u,y,z)||\} $.

Also, due to (\ref{e:convergence-important-12})  and the dominated convergence theorem (see, e.g., p. 49 in \cite{Ash})
 \begin{equation}\label{e:L-1-8}
 \lim_{N\rightarrow \infty}\limsup_{M\rightarrow \infty}\int_0^{\tau} || f(u^{N,M}(y^*_t(\tau'),z^*(t)), y^*_t(\tau'), z^*(t))- f(u_t^*(\tau'), y^*_t(\tau'), z^*(t))|| d\tau' = 0.
\end{equation}
Let us introduce the notation
$$
\kappa_{t,\tau}(N,M)\BYDEF L_2\ meas\{P_{t,\tau}(N,M)\}
$$
\vspace{-.2in}
$$
+\int_0^{\tau} || f(u^{N,M}(y^*_t(\tau'),z^*(t)), y^*_t(\tau'), z^*(t))- f(u_t^*(\tau'), y^*_t(\tau'), z^*(t))|| d\tau'
$$
and rewrite the inequality (\ref{e:L-1-6}) in the form
 \begin{equation}\label{e:L-1-9}
|| y_{t}^{N,M}(\tau)-y^*_t(\tau)||\leq L_1 \int_0^{\tau}||y^{N,M}_{t} (\tau') - y^*_t(\tau')|| d\tau' + \kappa_{t,\tau}(N,M).
\end{equation}
By Gronwall-Bellman lemma, it follows that
 \begin{equation}\label{e:L-1-10}
 \max_{\tau'\in [0,\tau]}|| y_{t}^{N,M}(\tau')-y^*_t(\tau')||  \leq \kappa_{t,\tau}(N,M)e^{L_1\tau}.
\end{equation}
Since, by (\ref{e:HJB-1-17-1-N-z-const}) and (\ref{e:L-1-8}),
\begin{equation}\label{e:L-1-11}
\lim_{N\rightarrow \infty}\limsup_{M\rightarrow \infty}\kappa_{t,\tau}(N,M) = 0,
\end{equation}
the inequality (\ref{e:L-1-10}) implies (\ref{e:HJB-1-17-1-N-z-const-fast-1}).

By (\ref{e:HJB-1-17-1-N-z-const-fast-1}), $\  y_{t}^{N,M}(\tau)\in B_{t,\tau} $ for $N $ and $M$ large enough (for $\tau\in [0,\infty) $  such that the ball
 $B_{t,\tau}$ is not empty). Hence,
$$
|| u^{N,M}(  y_{t}^{N,M}(\tau), z^*(t)) -u_t^*(\tau)||\leq || u^{N,M}(  y_{t}^{N,M}(\tau), z^*(t)) - u^{N,M}(y^*_t(\tau), z^*(t))||
$$
\vspace{-.2in}
$$
+||u^{N,M}(y^*_t(\tau), z^*(t)) - u_t^*(\tau)||\leq L ||y_{z^*(t)}^{N,M}(\tau) - y^*_t(\tau)|| +||u^{N,M}(y^*_t(\tau), z^*(t)) - u_t^*(\tau)||.
$$
The latter implies (\ref{e:HJB-1-17-1-N-z-const-fast-2}) (by (\ref{e:HJB-1-17-1-N-z-const-fast-1}) and (\ref{e:convergence-important-12})).

\endproof

{\it Proof of Theorem \ref{Main-SP-Nemeric}.} Let $q(u,y) $ be continuous. By (\ref{e-opt-OM-1-0-NM-star}) and (\ref{e-opt-OM-1-0-NM}),
for an arbitrary small $\alpha > 0 $ there exists
 $T> 0$ such that
 \begin{equation}\label{e:LA-Airport-2}
 | T^{-1}\int_0^T q(u_t^*(\tau),y_{t}^{*}(\tau))d\tau - \int_{U\times Y}q(u,y)\mu^{*}(t)(du,dy)|\leq \frac{\alpha}{2}
\end{equation}
and
\begin{equation}\label{e:LA-Airport-1}
 | T^{-1}\int_0^T q(u^{N,M}(  y_{t}^{N,M}(\tau), z^*(t)),y_{t}^{N,M}(\tau))d\tau - \int_{U\times Y}q(u,y)\mu^{N,M}(z^*(t))(du,dy)|\leq \frac{\alpha}{2} .
\end{equation}
Using (\ref{e:LA-Airport-1}) and (\ref{e:LA-Airport-2}), one can obtain
$$
|\int_{U\times Y}q(u,y)\mu^{N,M}(z^*(t))(du,dy) - \int_{U\times Y}q(u,y)\mu^{*}(t)(du,dy)|
$$
\vspace{-.2in}
$$
\leq \ | | T^{-1}\int_0^T q(u^{N,M}(  y_{t}^{N,M}(\tau), z^*(t)),y_{t}^{N,M}(\tau))d\tau - T^{-1}\int_0^T q(u_t^*(\tau),y_{t}^{*}(\tau))d\tau| + \alpha .
$$
Due to Lemma \ref{fast-convergence}, the latter implies the following inequality
$$
\lim_{N\rightarrow\infty}\limsup_{M\rightarrow\infty}|\int_{U\times Y}q(u,y)\mu^{N,M}(z^*(t))(du,dy) - \int_{U\times Y}q(u,y)\mu^{*}(t)(du,dy)|\leq \alpha ,
$$
which, in turn, implies
\begin{equation}\label{e:LA-Airport-3}
 \lim_{N\rightarrow\infty}\limsup_{M\rightarrow\infty}|\int_{U\times Y}q(u,y)\mu^{N,M}(z^*(t))(du,dy) - \int_{U\times Y}q(u,y)\mu^{*}(t)(du,dy)|=0
\end{equation}
(due to the fact that $\alpha $ can be arbitrary small). Since $q(u,y) $ is an arbitrary continuous function, from (\ref{e:LA-Airport-3}) it follows that
\begin{equation}\label{e:HJB-16-NM-proof-2}
\lim_{N\rightarrow\infty}\limsup_{M\rightarrow\infty }\rho(\mu^{N,M}(z^*(t)), \mu^*(t)) = 0 ,
\end{equation}
the latter being valid for almost all $t\in [0,\infty) $.

 Taking an arbitrary $t\in [0,\infty) $ and subtracting the equation
 \begin{equation}\label{e:HJB-16-NM-proof-3}
z^*(t)=z_0 + \int_0^t \tilde g(\mu^*(t'), z^*(t')) dt'
\end{equation}
from the equation
 \begin{equation}\label{e:HJB-16-NM-proof-4}
z^{N,M}(t)=z_0 + \int_0^t \tilde g(\mu^{N,M}(z^{N,M}(t')), z^{N,M}(t')) dt',
\end{equation}
one obtains
$$
||z^{N,M}(t)-z^*(t)||\leq \int_0^t ||\tilde g(\mu^{N,M}(z^{N,M}(t')), z^{N,M}(t'))- \tilde g(\mu^*(t'), z^*(t'))|| dt'
$$
\vspace{-.2in}
$$
\leq \int_0^t ||\tilde g(\mu^{N,M}(z^{N,M}(t')), z^{N,M}(t')) - \tilde g(\mu^{N,M}(z^{*}(t')), z^{*}(t'))||dt'
$$
\vspace{-.2in}
 \begin{equation}\label{e:HJB-16-NM-proof-5}
 +  \int_0^t || \tilde g(\mu^{N,M}(z^{*}(t')), z^{*}(t')) - \tilde g(\mu^*(t'), z^*(t'))|| dt'.
\end{equation}
From (\ref{e:HJB-1-17-1-per-Lip-g}) and from the definition of the set $A_t(N,M)$ (see (\ref{e:intro-0-4-N-A-t})), it follows that
$$
\int_0^t ||\tilde g(\mu^{N,M}(z^{N,M}(t')), z^{N,M}(t')) - \tilde g(\mu^{N,M}(z^{*}(t')), z^{*}(t'))||dt'
$$
\vspace{-.2in}
$$
\leq \int_{t'\notin A_t(N,M)}||\tilde g(\mu^{N,M}(z^{N,M}(t')), z^{N,M}(t')) - \tilde g(\mu^{N,M}(z^{*}(t')), z^{*}(t'))||dt'
$$
\vspace{-.2in}
$$
+ \int_{t'\in A_t(N,M)}[||\tilde g(\mu^{N,M}(z^{N,M}(t')), z^{N,M}(t'))|| + ||\tilde g(\mu^{N,M}(z^{*}(t')), z^{*}(t'))||\ ]dt'
$$
\vspace{-.2in}
 \begin{equation}\label{e:HJB-16-NM-proof-6}
\leq L\int_0^t ||z^{N,M}(t')- z^{*}(t') || + 2 L_g meas\{A_t(N,M)\},
\end{equation}
where $L_g\BYDEF  \max_{(u,y,z)\in U\times Y\times Z}||g(u,y,z||$.
This and (\ref{e:HJB-16-NM-proof-5}) allows one to obtain the inequality
\begin{equation}\label{e:HJB-16-NM-proof-6-11}
||z^{N,M}(t)-z^*(t)||\leq L\int_0^t ||z^{N,M}(t')- z^{*}(t') || +\kappa_t(N,M),
\end{equation}
 where
 $$
 \kappa_t(N,M)\BYDEF  2 L_g meas\{A_t(N,M)\} + \int_0^t || \tilde g(\mu^{N,M}(z^{*}(t')), z^{*}(t')) - \tilde g(\mu^*(t'), z^*(t'))|| dt'.
 $$
Note that, by (\ref{e:HJB-16-NM-proof-2}),
\begin{equation}\label{e:HJB-16-NM-proof-7}
 \lim_{N\rightarrow\infty}\limsup_{M\rightarrow\infty }\int_0^t || \tilde g(\mu^{N,M}(z^{*}(t')), z^{*}(t')) - \tilde g(\mu^*(t'), z^*(t'))|| dt'= 0,
\end{equation}
which, along with (\ref{e:HJB-1-17-1-N}), imply that
\begin{equation}\label{e:HJB-16-NM-proof-9}
\lim_{N\rightarrow\infty}\limsup_{M\rightarrow\infty }\kappa_t(N,M)=0.
\end{equation}
By Gronwall-Bellman lemma, from (\ref{e:HJB-16-NM-proof-6-11}) it follows that
$$
max_{t'\in [0,t]}||z^{N,M}(t')-z^*(t')||\leq  \kappa_t(N,M) e^{Lt}.
$$
The latter along with (\ref{e:HJB-16-NM-proof-9})  imply (\ref{e:HJB-1-19-2-NM}).

Let us now establish the validity of (\ref{e:HJB-1-19-1-NM}). Let  $t\in [0,\infty)$ be such that the ball $Q_t $ introduced in Assumption
\ref{SET-2}  is not empty. By triangle inequality,
\begin{equation}\label{e:summarizing-N-M-convergence}
\rho (\mu^{N,M}(z^{N,M}(t)),\mu^*(t))\leq \rho (\mu^{N,M}(z^{N,M}(t)),\mu^{N,M}(z^*(t))) + \rho (\mu^{N,M}(z^*(t)), \mu^*(t))
\end{equation}
Due to (\ref{e:HJB-1-19-2-NM}), $z^{N,M}(t)\in Q_t\ $ for $M$ and $N$ large enough. Hence, by (\ref{e:HJB-1-17-1-per-cont-mu-NM-g-1}),
$$
     \rho(\mu^{N,M}(z^{N,M}(t)), \mu^{N,M}(z^*(t)))\leq \kappa(||z^{N,M}(t')-z^*(t')||),
   $$
which implies that
$$
\lim_{N\rightarrow\infty}\limsup_{M\rightarrow\infty}\rho (\mu^{N,M}(z^{N,M}(t)),\mu^{N,M}(z^*(t)))=0.
$$
The latter, along with (\ref{e:HJB-16-NM-proof-2}) and (\ref{e:summarizing-N-M-convergence}), imply (\ref{e:HJB-1-19-1-NM}).

To prove (\ref{e:HJB-19-NM}), let us recall that
\begin{equation}\label{e:L-1-12-recall-1}
\tilde{V}^*_{di}(z_0) =\int_0 ^ { + \infty }e^{- C t}\tilde{G}(\mu^*(t), z^*(t))dt.
 \end{equation}
For an arbitrary $\delta > 0 $, choose $T_{\delta}> 0 $ in such a way that
$$
2M\int_{T_{\delta}}^{ + \infty }e^{-Ct}dt\leq \frac{\delta}{2}\ , \ \ \ \ \ \ \ M\BYDEF\max_{(u,y,z)\in U\times Y\times Z}\{|G(u,y,z)|\}.
$$
Then
$$
|\tilde V_{di}^{N,M}(z_0)-\tilde{V}^*_{di}(z_0)|\leq \int_0^{T_{\delta}}e^{-Ct}|\tilde G(\mu^{N,M}(z^{N,M}(t)),
z^{N,M}(t))- \tilde{G}(\mu^*(t), z^*(t))|dt + \delta
$$
By (\ref{e:HJB-1-19-2-NM}) and (\ref{e:HJB-1-19-1-NM}),
$$
\lim_{N\rightarrow\infty}\limsup_{M\rightarrow\infty}\int_0^{T_{\delta}}e^{-Ct}|\tilde G(\mu^{N,M}(z^{N,M}(t)),
z^{N,M}(t))- \tilde{G}(\mu^*(t), z^*(t))|dt = 0.
$$
Hence,
$$
\lim_{N\rightarrow\infty}\limsup_{M\rightarrow\infty}|\tilde V_{di}^{N,M}(z_0)-\tilde{V}^*_{di}(z_0)| \leq \delta.
$$
Since $\delta > 0 $ can be arbitrary small, the latter implies (\ref{e:HJB-19-NM}).

\endproof

\section{Proof of Theorem \ref{Prop-convergence-measures-discounted}}\label{Sec-Main-AVE}\

{\it Proof of Theorem \ref{Prop-convergence-measures-discounted}}.
Without loss of generality, one may assume  that
$r_{\delta}$ is  decreasing with $\delta $ and that
$r_{\delta}\leq \delta$ (the later can be achieved by replacing $r_{\delta}$ with  $\min\{\delta , r_{\delta} \} $  if necessary). Having this
in mind, define $\bar\delta(\epsilon)$ as the solution of the problem
\begin{equation}\label{e-implicit-1-0}
\min \{ \delta \ : \ r_{\delta}\geq \Delta^{\frac{1}{2}}(\epsilon)\}.
\end{equation}
The solution of this problem exists since $r_{\delta} $ is right-continuous and, by definition,
\begin{equation}\label{e-implicit-1}
r_{\bar{\delta}(\epsilon)} = \Delta^{\frac{1}{2}}(\epsilon).
\end{equation}
Fix an arbitrary $T>0 $ and introduce the notation
\begin{equation}\label{e-implicit-1-0-100}
\delta(\epsilon)\BYDEF \max \{\bar{\delta}(\epsilon),  \delta_T(\epsilon)\},
\end{equation}
where $ \delta_T(\epsilon) $ is an the statement of the theorem.
Note that, by construction,
\begin{equation}\label{e-implicit-2}
\lim_{\epsilon\rightarrow 0}\delta(\epsilon) = 0, \ \ \ \ \ \ \ \ \ \ \     \delta(\epsilon)\geq \Delta^{\frac{1}{2}}(\epsilon).
\end{equation}
As can be readily seen,
\begin{equation}\label{e-implicit-4}
\max_{t\in [t_l, t_{l+1}]}||z(t)-z(t_l)||\leq M\Delta(\epsilon),
\end{equation}
where $z(\cdot)$ is the solution of (\ref{e-opt-OM-1}) and
$\ M\BYDEF \max_{(u,y,z)\in U\times Y\times Z}||g(u,y,z)||$.
Hence,
$$
z(t)\in z(t_l)+ r_{\delta(\epsilon)} B \ \ \forall t\in [t_l, t_{l+1}]
$$
for all $\epsilon $ small enough. Consequently, due to the assumed weak piecewise Lipschitz continuity of the
ACG family under consideration
(see Definition \ref{ASS-locally-Lipschitz}),
\begin{equation}\label{e-implicit-5-1}
 ||\tilde g_{\mu}(z(t))- \tilde g_{\mu}(z(t_l)) ||\leq L_g ||z(t)-z(t_l)||\leq L_gM\Delta(\epsilon) \ \ \ \ \ {\rm if}  \ \ \ \
 \ t_l\notin\cup_{i=1}^k ( \bar t_i - \delta(\epsilon) , \bar t_i+ \delta (\epsilon)),
\end{equation}
where $L_g$ is a Lischitz constant of $\tilde g_{\mu}(\cdot) $.
Being the solution of (\ref{e-opt-OM-1}), $z(\cdot)$  satisfies the equality
\begin{equation}\label{e-implicit-3}
z(t_{l+1})-z(t_l) - \int_{t_l}^{t_{l+1}}\tilde g_{\mu}(z(t))dt = 0,
\end{equation}
which along with (\ref{e-implicit-5-1}) allow one to obtain
$$
|| z(t_{l+1})-z(t_l) - \Delta (\epsilon) \tilde g_{\mu}(z(t_l))|| =
|| \int_{t_l}^{t_{l+1}}\tilde g_{\mu}(z(t))dt - \Delta (\epsilon) \tilde g_{\mu}(z(t_l)) ||
$$
\begin{equation}\label{e-implicit-6}
\leq \int_{t_l}^{t_{l+1}}||\tilde g_{\mu}(z(t)) -\tilde g_{\mu}(z(t_l))||dt  \leq L_gM\Delta^2(\epsilon)
\ \ \ \ \ {\rm if} \ \ \ \ \ t_l\notin\cup_{i=1}^k ( \bar t_i - \delta(\epsilon) , \bar t_i+ \delta (\epsilon)).
\end{equation}
In addition to the above, one can obtain (by (\ref{e-implicit-4}))
$$
|| z(t_{l+1})-z(t_l) - \Delta (\epsilon) \tilde g_{\mu}(z(t_l))|| \leq || z(t_{l+1})-z(t_l)|| + \Delta (\epsilon) ||\tilde g_{\mu}(z(t_l))||
$$
\begin{equation}\label{e-implicit-7}
\leq 2M \Delta(\epsilon)\ \ \ \ \ {\rm if} \ \ \ \ \ t_l\in\cup_{i=1}^k ( \bar t_i - \delta(\epsilon) , \bar t_i+ \delta (\epsilon)).
\end{equation}
To continue the proof, let us rewrite the SP system (\ref{e:intro-0-1})-(\ref{e:intro-0-2}) in the \lq\lq stretched" time scale
$\tau = t\epsilon^{-1} $
\begin{eqnarray}\label{e:intro-0-1-STR}
\frac{dy(\tau)}{d\tau} &=&f(u(\tau),y(\tau),z(\tau)) ,   \\
\frac{dz(\tau)}{d\tau}&=&\epsilon g(u(\tau),y(\tau),z(\tau)).  \label{e:intro-0-2-STR}
\end{eqnarray}
Let us also introduce the following notations
\begin{equation}\label{e-implicit-7-1}
\tau_l \BYDEF t_l \epsilon^{-1}, \ \ \ \ \ \ \ \ S(\epsilon)\BYDEF \Delta(\epsilon)\epsilon^{-1}.
\end{equation}
In these notations, the control $u^{\epsilon}(\cdot)$ defined by (\ref{e:contr-rev-100-3}) and (\ref{e:contr-rev-100-4}) is rewritten in the form
\begin{equation}\label{e-implicit-7-1-101}
u_{\epsilon}(\tau) = u_{y_{\epsilon}(\tau_l),z_{\epsilon}(\tau_l)}(\tau - \tau_l) \ \ \ \ {\rm for} \ \ \ \ \tau\in [\tau_l, \tau_{l+1}),
 \ \ \ \ \ \ \ \ l=0,1,...\ ,
\end{equation}
and the solution $(y_{\epsilon}(\cdot),z_{\epsilon}(\cdot))$  of the system (\ref{e:intro-0-1-STR})-(\ref{e:intro-0-2-STR}) obtained with this control satisfies
the equations
\begin{equation}\label{e-implicit-8}
z_{\epsilon}(\tau)-z_{\epsilon}(\tau_l) - \epsilon \int_{\tau_l}^{\tau}g(u_{y_{\epsilon}(\tau_l),z_{\epsilon}(\tau_l)}(\tau ' - \tau_l),y_{\epsilon}(\tau '),
z_{\epsilon}(\tau '))d\tau ' = 0 \ \ \ \  \forall \tau\in [\tau_l , \tau_{l+1}],
\end{equation}
\vspace{-.2in}
\begin{equation}\label{e-implicit-8-1}
y_{\epsilon}(\tau)-y_{\epsilon}(\tau_l) -  \int_{\tau_l}^{\tau}f(u_{y_{\epsilon}(\tau_l),z_{\epsilon}(\tau_l)}(\tau ' - \tau_l),y_{\epsilon}(\tau '),
z_{\epsilon}(\tau '))d\tau ' = 0 \ \ \ \  \forall \tau\in [\tau_l , \tau_{l+1}].
\end{equation}
Note that the estimate (\ref{e:convergence-to-gamma-z-estimate}), which we are going to prove, is rewritten
in the stretched time scale as follows
\begin{equation}\label{e:convergence-to-gamma-z-estimate-tau}
max_{t\in [0,\frac{T}{\epsilon}]}||z_{\epsilon}(\tau)-z(\tau \epsilon)||\leq \beta_T(\epsilon), \ \ \ \ \ \ \lim_{\epsilon\rightarrow 0}\beta_T(\epsilon) = 0.
\end{equation}
Let us consider (\ref{e-implicit-8}) with $\tau = \tau_{l+1} $ and subtract  it from the expression
$\ z(t_{l+1})-z(t_l) - \Delta (\epsilon) \tilde g_{\mu}(z(t_l)) $. Using (\ref{e-implicit-6}), one can
obtain
$$
||\ [z(t_{l+1})-z_{\epsilon}(\tau_{l+1})] - [z(t_{l})-z_{\epsilon}(\tau_{l})]
$$
\vspace{-.2in}
$$
- \Delta(\epsilon) [\ \tilde g_{\mu}(z(t_l)) -
S^{-1}(\epsilon)\int_{\tau_l}^{\tau_{l+1}}g(u_{y_{\epsilon}(\tau_l),z_{\epsilon}(\tau_l)}(\tau ' - \tau_l),y_{\epsilon}(\tau '),
z_{\epsilon}(\tau '))d\tau' ||
$$
\vspace{-.2in}
$$
\leq L_gM\Delta^2(\epsilon) \ \ \ \ \ {\rm if} \ \ \ \ \ t_l\notin\cup_{i=1}^k ( \bar t_i - \delta(\epsilon) , \bar t_i+ \delta (\epsilon))
$$
$$
\Rightarrow \ \ \ ||z(t_{l+1})-z_{\epsilon}(\tau_{l+1})||\leq ||z(t_{l})-z_{\epsilon}(\tau_{l})||
$$
\vspace{-.3in}
$$
+ \ \Delta(\epsilon)\ ||\ \tilde g_{\mu}(z(t_l)) - S^{-1}(\epsilon)\int_{0}^{S(\epsilon)}g(u_{y_{\epsilon}(\tau_l),z_{\epsilon}(\tau_l)}
(\tau),y_{y_{\epsilon}(\tau_l),z_{\epsilon}(\tau_l)}(\tau),z_{\epsilon}(\tau_l)
)d\tau ||
$$
\vspace{-.2in}
$$
+ \ \Delta(\epsilon)\  ||\ S^{-1}(\epsilon)\int_{0}^{S(\epsilon)}g(u_{y_{\epsilon}(\tau_l),z_{\epsilon}(\tau_l)}
(\tau),y_{y_{\epsilon}(\tau_l),z_{\epsilon}(\tau_l)}(\tau),z_{\epsilon}(\tau_l)
)d\tau
$$
\vspace{-.2in}
$$
- \ S^{-1}(\epsilon)\int_{0}^{S(\epsilon)}g(u_{y_{\epsilon}(\tau_l),z_{\epsilon}(\tau_l)}(\tau),y_{\epsilon}(\tau + \tau_l),z_{\epsilon}(\tau + \tau_l))d\tau ||
$$
\vspace{-.2in}
\begin{equation}\label{e-implicit-9}
+ \ L_gM\Delta^2(\epsilon) \ \ \ \ \ {\rm if} \ \ \ \ \ t_l\notin\cup_{i=1}^k ( \bar t_i - \delta(\epsilon) , \bar t_i+ \delta (\epsilon)),
\end{equation}
where $y_{y_{\epsilon}(\tau_l),z_{\epsilon}(\tau_l)}(\cdot)$ is the solution of the associated system (\ref{e:intro-0-3}) considered with
 $z=z_{\epsilon}(\tau_l) $  and with the initial condition $y(0)= y_{\epsilon}(\tau_l)$.
 The control $u_{y_{\epsilon}(\tau_l),z^{\epsilon}(\tau_l)}(\tau) $
 steers the associated system to $\mu(du,dy|z^{\epsilon}(\tau_l))$, with the estimate (\ref{e-opt-OM-2-105}) being uniform with respect
 to the initial condition and the values of $z$. This implies that there exists a function $\phi_g(S)$,
  $\
 \lim_{S\rightarrow\infty}\phi_g(S) = 0
 $,
  such that
 \begin{equation}\label{e-opt-OM-2-105-extra}
 ||\ S^{-1}(\epsilon) \int_0^{S(\epsilon)} g(u_{y_{\epsilon}(\tau_l),z_{\epsilon}(\tau_l)}(\tau),y_{y_{\epsilon}(\tau_l)),z_{\epsilon}(\tau_l)}(\tau),z_{\epsilon}(\tau_l))
 d\tau
  - \tilde g_{\mu}(z_{\epsilon}(\tau_l))|| \leq \phi_g(S(\epsilon)).
 \end{equation}
Consequently, for $\ t_l\notin\cup_{i=1}^k ( \bar t_i - \delta(\epsilon) , \bar t_i+ \delta (\epsilon)) $ and $t_l\leq T $,
$$
\Delta(\epsilon)\ ||\ \tilde g_{\mu}(z(t_l)) - S^{-1}(\epsilon)\int_{0}^{S(\epsilon)}g(u_{y_{\epsilon}(\tau_l),z_{\epsilon}(\tau_l)}
(\tau),y_{y_{\epsilon}(\tau_l),z_{\epsilon}(\tau_l)}(\tau),z_{\epsilon}(\tau_l)
)d\tau ||
$$
\vspace{-.2in}
\begin{equation}\label{e-opt-OM-2-105-extra-11}
\leq \Delta(\epsilon)||\tilde g_{\mu}(z(t_l)) - \tilde g_{\mu}(z_{\epsilon}(\tau_l))|| + \Delta(\epsilon)\phi_g(S(\epsilon))
\leq \Delta(\epsilon) \tilde L_g ||z(t_l)) - z_{\epsilon}(\tau_l))|| + \Delta(\epsilon)\phi_g(S(\epsilon)),
 \end{equation}
where $\tilde L_g $ is a Lipschitz constant of $\tilde g_{\mu}(\cdot) $. Note that the last
inequality is valid since $\tilde g_{\mu}(\cdot) $ satisfies Lipschitz condition on $z(t_l)+ a_{t_l}B$
and since $\ z_{\epsilon}(\tau_l)\in  z(t_l)+ a_{t_l}B$
(as $\ t_l\in [0,T] \setminus \cup_{i=1}^k ( \bar t_i -  \delta(\epsilon),
\bar t_i+  \delta(\epsilon) ) $; see (\ref{e:convergence-to-gamma-extra-condition}) and (\ref{e-implicit-1-0-100})).

Also,
$$
\ \Delta(\epsilon)\  ||\ S^{-1}(\epsilon)\int_{0}^{S(\epsilon)}g(u_{y_{\epsilon}(\tau_l),z_{\epsilon}(\tau_l)}
(\tau),y_{y_{\epsilon}(\tau_l),z_{\epsilon}(\tau_l)}(\tau),z_{\epsilon}(\tau_l)
)d\tau
$$
\vspace{-.2in}
$$
- \ S^{-1}(\epsilon)\int_{0}^{S(\epsilon)}g(u_{y_{\epsilon}(\tau_l),z_{\epsilon}(\tau_l)}(\tau),y_{\epsilon}(\tau + \tau_l),z_{\epsilon}(\tau + \tau_l))d\tau ||
$$
\vspace{-.2in}
\begin{equation}\label{e-opt-OM-2-105-extra-12}
\leq L_g \Delta(\epsilon) (\max_{\tau\in [0,S(\epsilon)]}||y_{y_{\epsilon}(\tau_l),z_{\epsilon}(\tau_l)}(\tau) - y_{\epsilon}(\tau + \tau_l)|| + M\Delta(\epsilon)\ ),
\end{equation}
where it has been taken into account that, by (\ref{e-implicit-8}),
\begin{equation}\label{e-opt-OM-2-105-extra-13}
\max_{\tau\in [0,S(\epsilon)]} ||z_{\epsilon}( \tau_l)-z_{\epsilon}(\tau + \tau_l)||\leq M(\epsilon S(\epsilon))= M\Delta(\epsilon).
\end{equation}
By definition, $\ y_{y_{\epsilon}(\tau_l),z_{\epsilon}(\tau_l)}(\cdot) $ satisfies the equation
\begin{equation}\label{e-opt-OM-2-105-extra-14}
y_{y_{\epsilon}(\tau_l),z_{\epsilon}(\tau_l)}(\tau)-y_{\epsilon}(\tau_l) -  \int_{0}^{\tau}f(u_{y_{\epsilon}(\tau_l),z_{\epsilon}(\tau_l)}(\tau ),
y_{y_{\epsilon}(\tau_l),z_{\epsilon}(\tau_l)}(\tau '),
z_{\epsilon}(\tau_l))d\tau ' = 0 \ \ \ \  \forall \tau\in [0 , S(\epsilon)].
\end{equation}
Rewriting (\ref{e-implicit-8-1}) in the form
$$
y_{\epsilon}(\tau +\tau_l)-y_{\epsilon}(\tau_l) -  \int_{0}^{\tau}f(u_{y_{\epsilon}(\tau_l),z_{\epsilon}(\tau_l)}(\tau ' ),y_{\epsilon}(\tau '+\tau_l),
z_{\epsilon}(\tau '+\tau_l))d\tau ' = 0 \ \ \ \  \forall \tau\in [0 , S(\epsilon)]
$$
and subtracting it from (\ref{e-opt-OM-2-105-extra-14}), one can obtain (by (\ref{e-opt-OM-2-105-extra-13}))
$$
||y_{y_{\epsilon}(\tau_l),z_{\epsilon}(\tau_l)}(\tau) - y_{\epsilon}(\tau +\tau_l)||\leq L_f \int_{0}^{\tau}
||y_{y_{\epsilon}(\tau_l),z_{\epsilon}(\tau_l)}(\tau') - y_{\epsilon}(\tau' +\tau_l)||d\tau' + L_f M \epsilon S^2(\epsilon)
\ \ \  \forall \tau\in [0 , S(\epsilon)],
$$
where $L_f $ is a Lipschitz constant of $f(\cdot)$. By Gronwall-Bellman lemma, the latter implies (see also (\ref{e:contr-rev-100-0-1}) and (\ref{e-implicit-7-1}))
\begin{equation}\label{e-opt-OM-2-105-extra-15}
\max_{\tau\in [0,S(\epsilon)]}||y_{y_{\epsilon}(\tau_l),z_{\epsilon}(\tau_l)}(\tau) - y_{\epsilon}(\tau + \tau_l)||\leq M \epsilon S^2(\epsilon)
e ^{L_f S(\epsilon)} = M \epsilon (\frac{1}{2L_f}\ln \frac{1}{\epsilon})^2 \epsilon^{-\frac{1}{2}}\leq \epsilon^{\frac{1}{4}}
\end{equation}
for $\epsilon $ small enough.

Taking (\ref{e-opt-OM-2-105-extra-11}), (\ref{e-opt-OM-2-105-extra-12}) and (\ref{e-opt-OM-2-105-extra-15}) into account, one can
rewrite (\ref{e-implicit-9}) in the form
$$
 ||z(t_{l+1})-z_{\epsilon}(\tau_{l+1})||\leq ||z(t_{l})-z_{\epsilon}(\tau_{l})|| + L_1\Delta (\epsilon)||z(t_{l})-z_{\epsilon}(\tau_{l})||
$$
\vspace{-.2in}
\begin{equation}\label{e-implicit-12}
+ L_2\Delta (\epsilon) \zeta(\epsilon) \ \ \ \ \ {\rm if} \ \ \ \ \ t_l\notin\cup_{i=1}^k ( \bar t_i - \delta(\epsilon) , \bar t_i+ \delta (\epsilon)),\ \ \ \ t_l\leq T,
\end{equation}
where $L_1, L_2 $ are appropriately chosen constants and
\begin{equation}\label{e-implicit-13}
 \zeta(\epsilon)\BYDEF \phi_g(S(\epsilon)) +   \Delta (\epsilon) + \epsilon^{\frac{1}{4}}.
\end{equation}
Note that $\lim_{\epsilon\rightarrow 0} \zeta(\epsilon) = 0 $.

By subtracting (\ref{e-implicit-8}) (taken with $\tau=\tau_{l+1} $)  from the expression
$\ z(t_{l+1})-z(t_l) - \Delta (\epsilon) \tilde g_{\mu}(z(t_l)) $ and taking into account
(\ref{e-implicit-7}), one obtains
$$
||\ [z(t_{l+1})-z_{\epsilon}(\tau_{l+1})] - [z(t_{l})-z_{\epsilon}(\tau_{l})]
$$
\vspace{-.2in}
$$
- \Delta(\epsilon) [\ \tilde g_{\mu}(z(t_l)) - S^{-1}(\epsilon)
\int_{\tau_l}^{\tau_{l+1}}g(u_{y^{\epsilon}(\tau_l),z(t_l)}(\tau - \tau_l),y_{\epsilon}(\tau),z_{\epsilon}(\tau))d\tau ||
$$
\vspace{-.2in}
$$
\leq 2M\Delta(\epsilon) \ \ \ \ \ {\rm if} \ \ \ \ \ t_l\in\cup_{i=1}^k ( \bar t_i - \delta(\epsilon) , \bar t_i+ \delta (\epsilon)),
$$
which leads to the estimate
\begin{equation}\label{e-implicit-14}
 ||z(t_{l+1})-z_{\epsilon}(\tau_{l+1})||\leq ||z(t_{l})-z_{\epsilon}(\tau_{l})|| + L_3\Delta (\epsilon)
 \ \ \ \ \  {\rm if} \ \ \ \ \ t_l\in\cup_{i=1}^k ( \bar t_i - \delta(\epsilon) , \bar t_i+ \delta (\epsilon)),
\end{equation}
where $L_3 = 4M $.
Denoting $\ ||z(t_{l})-z_{\epsilon}(\tau_{l})|| $ as $\kappa_l$, one can come to the conclusion (based on (\ref{e-implicit-12}) and (\ref{e-implicit-14})
and on Lemma \ref{Lemma-Estimates-sigmas}  stated below) that, for any $T>0$, there exists $\xi_T(\epsilon) $,
 $ \ \lim_{\epsilon\rightarrow 0}\xi_T(\epsilon) = 0$, such that
\begin{equation}\label{e-implicit-15}
||z(t_{l})-z_{\epsilon}(\tau_{l})||\leq \xi_T(\epsilon),\ \ \ \ l=0,1,..., \ \lfloor \frac{T}{\Delta(\epsilon)} \rfloor ,
\end{equation}
where $\lfloor \cdot \rfloor $ stands for the floor function   ($\lfloor x \rfloor $ is the maximal integer number that is less or equal than  $x $).
Using (\ref{e-opt-OM-2-105-extra-13}), (\ref{e-implicit-15})  and
having in mind the fact  that $t_l = \tau_l \epsilon $ and that the inequality (\ref{e-implicit-4}) can be rewritten as
$$
\max_{\tau\in [\tau_l, \tau_{l+1}]}||z(\tau \epsilon)-z(\tau_l \epsilon)||\leq M\Delta(\epsilon),
$$
one  can establish the validity of  (\ref{e:convergence-to-gamma-z-estimate-tau}) with
$\beta_T(\epsilon) \BYDEF \xi_T(\epsilon) +2M\Delta(\epsilon)$.

To show that the discounted occupational measure $\gamma_{di}^{\epsilon} $ generated by
$(u_{\epsilon}(\cdot), y_{\epsilon}(\cdot),z_{\epsilon}(\cdot))$ converges to the measure $\gamma$ defined in (\ref{e-opt-OM-2}) it is  sufficient
to show that, for any Lipschitz continuous function $q(u,y,z)$,
\begin{equation}\label{e-implicit-16}
\lim_{\epsilon\rightarrow 0} \int_{U\times Y\times Z}q(u,y,z)\gamma_{di}^{\epsilon}(du,dy,dz)=  \int_{ Z}\tilde q_{\mu}(z)\nu_{di}(dz),
\end{equation}
where $\nu_{di}(dz)$ is the discounted occupational measure generated by the solution $z(t)$ of (\ref{e-opt-OM-1}) and
$\tilde q_{\mu}(\cdot) $ is defined by (\ref{e-opt-OM-1-101}).
By the definition of $\nu_{di}(dz)$
(see (\ref{e-opt-OM-1-extra-101})),
\begin{equation}\label{e-implicit-18}
  \int_{ Z}q_{\mu}(z)\nu_{di}(dz) = C\int_0^{\infty}e^{-Ct} \tilde q_{\mu}(z(t))dt.
\end{equation}
Also, by the definition of the discounted occupational measure (see (\ref{e:oms-0-2})) and due to the fact that the triplet
 $(u_{\epsilon}(\cdot), y_{\epsilon}(\cdot),z_{\epsilon}(\cdot))$ is considered in the stretched time scale,
\begin{equation}\label{e-implicit-19}
\int_{U\times Y\times Z}q(u,y,z)\gamma^{\epsilon}_{di}(du,dy,dz)=  \epsilon C \int_0^{\infty}e^{-\epsilon C\tau}
q(u_{\epsilon}(\tau), y_{\epsilon}(\tau),z_{\epsilon}(\tau))d \tau .
\end{equation}
As can be readily seen,
$$
\lim_{T\rightarrow\infty}C\int_T^{\infty}e^{-Ct} \tilde q_{\mu}(z(t))dt = 0, \ \ \ \ \
\lim_{T\rightarrow\infty}\epsilon C \int_{\frac{T}{\epsilon}}^{\infty}e^{-\epsilon C\tau}
q(u_{\epsilon}(\tau), y_{\epsilon}(\tau),z_{\epsilon}(\tau))d \tau = 0,
$$
with the convergence  being uniform in $\epsilon $ (in the second case). Thus, to prove (\ref{e-implicit-16}),
it is sufficient to prove that
\begin{equation}\label{e-implicit-20}
\lim_{\epsilon\rightarrow 0}\epsilon \int_0^{\frac{T}{\epsilon}}e^{-\epsilon C\tau}
q(u_{\epsilon}(\tau), y_{\epsilon}(\tau),z_{\epsilon}(\tau))d \tau = \int_0^{T}e^{-Ct} \tilde q_{\mu}(z(t))dt.
\end{equation}
The main steps in proving (\ref{e-implicit-20}) are as follows. Let $N_{\epsilon}\BYDEF \lfloor \frac{T}{\Delta(\epsilon)} \rfloor $. As can be readily seen,
the following estimates are valid:
\begin{equation}\label{e-implicit-6-103-1}
 |\int_0^{T}e^{-Ct} \tilde q_{\mu}(z(t))dt - \sum_{l=0}^{N_{\epsilon}-1}e^{-Ct_l} \tilde q_{\mu}(z(t_l))\Delta (\epsilon)|\leq \xi_T (\epsilon), \ \ \ \
 {\rm where} \ \ \ \  \lim_{\epsilon\rightarrow 0}\xi_T (\epsilon) = 0,
\end{equation}
$$
| \epsilon \int_0^{\frac{T}{\epsilon}}e^{-\epsilon C\tau}
q(u_{\epsilon}(\tau), y_{\epsilon}(\tau),z_{\epsilon}(\tau))d \tau
$$
\vspace{-.3in}
\begin{equation}\label{e-implicit-6-104}
- \  \Delta(\epsilon ) \sum_{l=0}^{N_{\epsilon}-1}S^{-1}(\epsilon)\int_0^{S(\epsilon)}e^{-\epsilon C(\tau + \tau_l)}
q(u_{y_{\epsilon}(\tau_l),z_{\epsilon}(\tau_l)}(\tau), y_{\epsilon}(\tau +\tau_l),z_{\epsilon}(\tau +\tau_l))d \tau| \ \leq \ M \Delta(\epsilon),
\end{equation}
where $M$ is a constant.
Similarly to (\ref{e-opt-OM-2-105-extra-11}) and (\ref{e-opt-OM-2-105-extra-12}), one can obtain (using the estimates
(\ref{e-opt-OM-2-105-extra-15}) and (\ref{e-implicit-15}))
\begin{equation}\label{e-implicit-6-105-7-1}
\Delta (\epsilon) |S^{-1}(\epsilon)\int_0^{S(\epsilon)}e^{-\epsilon C(\tau + \tau_l)}
q(u_{y_{\epsilon}(\tau_l),z_{\epsilon}(\tau_l)}(\tau), y_{\epsilon}(\tau +\tau_l),z_{\epsilon}(\tau +\tau_l))d \tau
- \ e^{-Ct_l} \ \tilde q_{\mu}(z(t_l)) | \leq \bar \zeta_T(\epsilon)\Delta (\epsilon)
\end{equation}
for $ \
 t_l\notin\cup_{i=1}^k ( \bar t_i - \delta(\epsilon) , \bar t_i+ \delta (\epsilon))$, where $\ \lim_{\epsilon\rightarrow 0}\bar \zeta_T(\epsilon) $.
In addition to that, one has the following estimate
\begin{equation}\label{e-implicit-6-105-8}
\Delta (\epsilon)|S^{-1}(\epsilon)\int_0^{S(\epsilon)}
q(u_{y_{\epsilon}(\tau_l),z_{\epsilon}(\tau_l)}(\tau), y_{\epsilon}(\tau +\tau_l),z_{\epsilon}(\tau +\tau_l))d \tau - e^{-Ct_l} \ \tilde q_{\mu}(z(t_l)) |\leq M\Delta(\epsilon)
\end{equation}
for $\ t_l\in\cup_{i=1}^k ( \bar t_i - \delta(\epsilon) , \bar t_i+ \delta (\epsilon))$, where $M$ is a constant.
From (\ref{e-implicit-6-105-7-1}) and (\ref{e-implicit-6-105-8}) it follows that
$$
\Delta (\epsilon)
\sum_{l=0}^{N_{\epsilon}-1}| \ S^{-1}(\epsilon)\int_0^{S(\epsilon)}e^{-\epsilon C(\tau + \tau_l)}
q(u_{y_{\epsilon}(\tau_l),z_{\epsilon}(\tau_l)}(\tau), y_{\epsilon}(\tau +\tau_l),z_{\epsilon}(\tau +\tau_l))d \tau
- e^{-Ct_l}\tilde q_{\mu}(z(t_l))| \leq \bar{\bar {\zeta}}_T(\epsilon),
$$
where $\ \lim_{\epsilon\rightarrow 0}\bar{\bar {\zeta}}_T (\epsilon)= 0 $. The latter along with (\ref{e-implicit-6-103-1}) and (\ref{e-implicit-6-104}) prove (\ref{e-implicit-20}).
\endproof

\begin{Lemma}\label{Lemma-Estimates-sigmas}
 Let $\delta(\epsilon)> 0 $ be as in (\ref{e-implicit-1-0-100})
and $\zeta(\epsilon)$ be as in (\ref{e-implicit-13}).
 Assume that $\kappa_0=0 $ and that the  numbers
$\kappa_l \geq 0, \ l=1,...,N_{\epsilon}\ \  (N_{\epsilon}= \lfloor \frac{T}{\Delta(\epsilon)} \rfloor )$,
 satisfy the inequalities
 \begin{equation}\label{piecewiseLip-1}
 \kappa_{l+1}\leq \kappa_l + L_1\Delta(\epsilon)\kappa_l +L_2 \Delta(\epsilon) \zeta(\epsilon)\ \ \ \ {\rm if} \ \ \ \
 \ t_l\notin\cup_{i=1}^k ( \bar t_i - \delta(\epsilon) , \bar t_i+ \delta (\epsilon))
 \end{equation}
 and
\begin{equation}\label{piecewiseLip-2}
 \kappa_{l+1}\leq \kappa_l + L_3\Delta(\epsilon) \ \ \ \ {\rm if} \ \ \ \ \ t_l\in\cup_{i=1}^k ( \bar t_i - \delta(\epsilon) , \bar t_i+ \delta (\epsilon)).
 \end{equation}
 Then
\begin{equation}\label{piecewiseLip-3}
 \kappa_{l}\leq L_4\delta(\epsilon) + L_5\zeta(\epsilon) \ \ \forall \ l=1,...,N_{\epsilon}\ ,
 \end{equation}
where $L_i\ , \ i=4,5 $,  may depend on $T$.
\end{Lemma}

\begin{proof} The proof is straightforward.
 \end{proof}

\bigskip

{\bf Acknowledgment.} We are indebted to Adil Bagirov and Musa Mammadov for making their optimization routines available to us.

\end{document}